\documentclass[12pt]{elsarticle}
\usepackage[utf8]{inputenc}

\usepackage[margin=1in]{geometry}
\usepackage{amsmath}  
\usepackage{amsthm}
\usepackage{amsfonts}
\usepackage{amssymb}
\usepackage{graphicx}  
\usepackage{float}
\usepackage{natbib}
\usepackage{hyperref}
\usepackage{bm}

\usepackage{caption}
\usepackage{subcaption}

\newtheorem{lemma}{Lemma}
\newtheorem{theorem}{Theorem}

\newtheorem{proposition}{Proposition}

\theoremstyle{definition}
\newtheorem{definition}{Definition}

\newcommand{\proj}[1]{\bm{\Pi}^{\varepsilon}_{1,E} #1}
\newcommand{\projLtwo}[2][\ell]{\bm{\Pi}^{0}_{#1,E} \bm{\varepsilon} (#2)}

\newcommand{\grad}[1]{\nabla #1}
\newcommand{\diverge}[1]{\nabla \cdot #1}
\newcommand{\polysym}{{\mathbb{P}}_{\ell}(E)^{2\times 2}_{\text{sym}}}
\newcommand{\VEM}{\bm{V}_{1,\ell}^E}
\newcommand{\symgrad}[1]{\nabla_s #1}
\newcommand{\pder}[1]{\frac{\partial }{\partial #1}}
\newcommand{\largeset}{\mathcal{EN}_{1,\ell}^E}
\newcommand{\doubleline}[1]{\overline{\overline{#1}}}

\newcommand{\Nmatrix}{\bm{N}^{\partial E}}

\newcommand{\Hone}{\bm{H}^1_0}
\newcommand{\Pker}{{\mathbb{P}_{\ell}^{\text{ker}}(E)}}
\hypersetup{
    colorlinks,
    citecolor=black,
    filecolor=black,
    linkcolor=black,
    urlcolor=black
}

\newcommand{\acmajor}[1]{#1}
\newcommand{\acrev}[1]{#1}

\journal{Elsevier}

\begin{document}

\begin{frontmatter}

\title{Stabilization-free virtual element method for plane elasticity}

\cortext[cor1]{Corresponding authors}

\author[inst1]{Alvin Chen\corref{cor1}}
\ead{avnchen@ucdavis.edu}

\affiliation[inst1]{organization={Department of Mathematics},
            addressline={University of California}, 
            city={Davis},
            postcode={95616}, 
            state={CA},
            country={USA}}

\author[inst2]{N. Sukumar\corref{cor1}}
\ead{nsukumar@ucdavis.edu}
\affiliation[inst2]{organization={Department of Civil and Environmental Engineering},
            addressline={University of California}, 
            city={Davis},
            postcode={95616}, 
            state={CA},
            country={USA}}

\begin{abstract}
We present the construction and application of a first order stabilization-free virtual element method to problems in plane elasticity. Well-posedness and error estimates of the discrete problem are established. The method is assessed on a series of well-known benchmark problems from linear elasticity and numerical results are presented that affirm the optimal convergence rate of the virtual element method in the $L^2$ norm and the energy seminorm.
\end{abstract}


\end{frontmatter}

\section{Introduction}\label{sec:intro}
In the past few years there has been considerable interest in the study of extensions of finite element methods to arbitrary polygonal meshes. The Virtual Element Method (VEM) is one such method 
introduced 
in~\cite{AHMAD2013376,basicprinciple,Veiga2014TheHG,daveiga2014virtual} 
for Poisson and other scalar elliptic boundary-valued problems. In~\cite{elastic,Artioli:2020:CMAME,elasticdaveiga} and~\cite{Gain:2014:VEM}, the approach has been extended to problems in two- and three-dimensional elasticity, respectively. \acmajor{The original VEM relies on a choice of a suitable stabilization operator to preserve the coerciveness of the problem. This has led to many studies on the choice of the stability term for 
elliptic problems in two dimensions~\cite{Veiga:2017:m3as,Mascotto:2018:nmpde}
and three dimensions~\cite{DASSI:2018:camwa}, as well as nonlinear elasticity~\cite{debellis:2019:cas, Reddy:2022:springer, Wriggers:2017:cmech}.} \acmajor{In} Berrone et al.~\cite{berrone2021lowest}, a stabilization-free VEM was introduced for the two-dimensional Poisson equation, which retains optimal order error estimates without a stability term. The main idea in this approach 
is to modify the standard first order virtual element space to allow for 
the computation of a higher order polynomial $L^2$ projection of the gradient. The degree of the polynomial on each element is chosen so that the discrete problem remains bounded and coercive.   
A related method for plane elasticity is proposed by D'Altri et al.~\cite{enhanced:VEM}, in which a $k$-th order polynomial space is enhanced with higher order
polynomials, and static condensation is then applied. In certain cases,
this approach leads to a stabilization-free VEM. \acmajor{The construction of the stabilization-free space can be seen as an extension of the space first defined in~\cite{AHMAD2013376}, while the enhancement of the strain projection resembles
assumed strain approaches and the method
of incompatible modes that are used in
the finite element method~\cite{Simo:1990:ijnme,Wriggers:1996:cmech}}.
The form of the stabilization term in the virtual element method is similar to that
in the hourglass stabilized finite \acmajor{element method}~\cite{hourglass,Flanagan:1981:AUS}.  The need 
for the
stabilization term is an undesirable attribute of these methods, 
\acmajor{since it can be problem dependent, is not uniquely defined and is
more involved to construct for problems with anisotropic material behavior and geometric or material nonlinearities.
The stabilization-free method retains both the flexibility with respect to
meshing and the optimal convergence rates of standard VEM, while only using information from the mesh to ensure coercivity. It is therefore of interest to develop stabilization-free virtual element methods~\cite{berrone2021lowest,enhanced:VEM}.}

In this paper, we extend the approach proposed in~\cite{berrone2021lowest} to problems in plane elasticity. In 
Section~\ref{sec:elastostatics}\acmajor{,} we set up the model problem of plane elasticity, and in
Sections~\ref{sec:mathprelim} and~\ref{sec:projections}\acmajor{,} 
we introduce the polynomial approximations and projections used in our constructions. 
The extension of the work from~\cite{berrone2021lowest} to
the vector-valued case is described in Section~\ref{sec:EEVEM}. In Section~\ref{sec:implementation}, we
present the construction and implementation of the 
projection matrices and the element stiffness matrix. Section~\ref{sec:theory} contains the theoretical results of well-posedness and error estimates using approximation techniques detailed in~\cite{berrone2021lowest,brenner,estimatesVEM,mixed:fem,ciarlet}. In Section~\ref{sec:numerical}, we solve several benchmark elasticity problems: 
patch test, bending of a cantilever beam, plate with a circular hole under uniaxial tension, and a hollow cylinder under internal pressure. The convex polygonal meshes that are used in the numerical study are generated using PolyMesher~\cite{talischi2012polymesher}. The rates of convergence in
the numerical simulations are found to be 
in agreement with the theoretical a priori error estimates. 
\acmajor{We close with a summary of our main findings 
in Section~\ref{sec:conclusions}.}

\section{Elastostatic Model Problem and Weak Form}\label{sec:elastostatics}
We consider an elastic body that occupies the region $\Omega \subset \mathbb{R}^2$ with
boundary $\partial \Omega$. Assume that the boundary $\partial \Omega$ can be written as the disjoint union of two parts $\Gamma_D$ and $\Gamma_N$ with prescribed Dirichlet and Neumann conditions on $\Gamma_D$ and $\Gamma_N$, respectively. 
The strong form for the elastostatic problem is:
\begin{subequations}\label{strongproblem}
\begin{align}
&\diverge{\bm{\sigma}} + \bm{f} = \bm{0}  \ \ \textrm{in } \Omega, 
\quad \bm{\sigma} = \bm{\sigma}^T \ \ \textrm{in } \Omega, \\
&\bm{\varepsilon}(\bm{u}) =\symgrad{\bm{u}} = \frac{1}{2}(\grad{\bm{u}} + \grad{\bm{u}}^T),\\
&\bm{\sigma}(\bm{u}) = \mathbb{C} : \bm{\varepsilon}(\bm{u}), \\ 
&\bm{u} = \bar{\bm{u}} \quad \text{on } \Gamma_D ,\\
&\bm{\sigma} \cdot \mathbf{n} = \bar{\bm{t}} \quad \text{on } \Gamma_N,
\end{align}
\end{subequations}
where $\bm{f} \in [L^2(\Omega)]^2$  is the body force per unit volume, $\bm{\sigma}$ is the Cauchy stress tensor,
$\bm{\varepsilon}$ is the small-strain tensor with $\symgrad{(\cdot)}$ being the symmetric gradient operator, 
\acmajor{$\bm{u}$} is the displacement field, 
$\bar{\bm{u}}$ and $\bar{\bm{t}}$ are the imposed essential boundary and traction boundary data, and 
$\bm{n}$ is the unit outward normal on the boundary. Linear elastic constitutive material relation ($\mathbb{C}$ is the material moduli tensor) and small-strain kinematics are assumed. 

The associated weak form of the boundary-value problem posed in~\eqref{strongproblem} is to find the displacement field $\bm{u} \in \bm{U}$, where $\bm{U} := \{\bm{u} : \bm{u} \in [H^{1}(\Omega)]^2, \ 
\bm{u} = \bar{\bm{u}} \ \textrm{on } \Gamma_D \}$, such that
\begin{subequations}\label{weakbilinearform}
\begin{align}
    a(\bm{u},\bm{v}) &= b(\bm{v}) \quad \forall \bm{v} \in \bm{U}_0,
\intertext{where $\bm{U}_0 = [H^1_0(\Omega)]^2$ and}
    a(\bm{u},\bm{v}) &= \int_{\Omega} \bm{\sigma}(\bm{u}): \bm{\varepsilon}(\bm{v}) \, d \bm{x}, \\ 
    b(\bm{v}) &= \int_{\Omega} \bm{f} \cdot \bm{v} \, d \bm{x} + \int_{\Gamma_N} \bar{\bm{t}}\cdot \bm{v} \, ds .
\end{align}
\end{subequations}
In~\eqref{weakbilinearform}, $H^1(\Omega)$ is the Hilbert space that consists of square-integrable functions up to order $1$ and $H_0^1(\Omega)$ is the subspace of
$H^1(\Omega)$ that contains functions that vanish on $\acmajor{\Gamma_D}$.

\section{Mathematical Preliminaries}\label{sec:mathprelim}
Let $\mathcal{T}^h$ be the decomposition of the region $\Omega$ into nonoverlapping polygons. For each polygon $E \in \mathcal{T}^h$, we denote its diameter by $h_E$ and its centroid by $\bm{x}_E$. Each
polygon $E$ consists of $N_E$ vertices (nodes) with $N_E$ edges. Let $\acmajor{\mathcal{E}_E}$ be the set of all edges of $E$. We denote the coordinate of each vertex by $\bm{x}_i := (x_i,y_i)$. In the VEM, standard mesh assumptions are placed on $\mathcal{T}^h$ (e.g., star-convexity of $E$)~\cite{basicprinciple}.

\subsection{Polynomial basis}
Over each element $E$, we define $[{\mathbb{P}_1}(E)]^2$ 
as the space of of two-dimensional vector-valued polynomials of degree less than or equal to $1$. On each $E$, we will also need to choose a basis. In particular, we choose the basis 
as: 
\begin{subequations}
\begin{align}
    \bm{\widehat{M}}(E) = \begin{bmatrix}  
    \begin{Bmatrix}
    1 \\ 0 
    \end{Bmatrix}, 
    \begin{Bmatrix}
    0 \\ 1
    \end{Bmatrix}, 
    \begin{Bmatrix}
    -\eta \\ \xi
    \end{Bmatrix}, 
    \begin{Bmatrix}
    \eta \\ \xi
    \end{Bmatrix}, 
    \begin{Bmatrix}
    \xi \\ 0
    \end{Bmatrix}, 
    \begin{Bmatrix}
    0 \\ \eta
    \end{Bmatrix}
    \end{bmatrix},\label{poly_2_basis}
\end{align}
where 
\begin{align}
    \xi = \frac{x-x_E}{h_E} , \quad \eta = \frac{y-y_E}{h_E} .
\end{align}
\end{subequations}
The $\alpha$-th element of the set $\bm{\widehat{M}}(E)$ is denoted by $\bm{m}_\alpha$.

\smallskip
We also define the space 
$\polysym$ that represents $2\times 2$ symmetric matrix polynomials of degree less than or equal to $\ell$. Since the matrices are symmetric we can represent them in terms of $3 \times 1$ vectors. On each element $E$, we choose the basis
\begin{subequations}
\begin{align}
    \widehat{\bm{M}}^{2\times 2}(E)=
    \begin{bmatrix}
    \begin{Bmatrix}
        1 \\ 0 \\ 0 
    \end{Bmatrix},
    \begin{Bmatrix}
        0 \\ 1 \\ 0
    \end{Bmatrix},
    \begin{Bmatrix}
        0 \\ 0 \\ 1
    \end{Bmatrix},
    \begin{Bmatrix}
        \xi \\ 0 \\ 0
    \end{Bmatrix},
    \begin{Bmatrix}
        0 \\ \xi \\ 0
    \end{Bmatrix},
    \begin{Bmatrix}
        0 \\ 0 \\ \xi
    \end{Bmatrix},\dots
    \begin{Bmatrix}
        \eta^\ell \\ 0 \\ 0
    \end{Bmatrix},
    \begin{Bmatrix}
        0 \\ \eta^\ell \\0 
    \end{Bmatrix},
    \begin{Bmatrix}
        0 \\ 0 \\ \eta^\ell
    \end{Bmatrix}
    \end{bmatrix}.
\end{align}
We denote the $\alpha$-th vector in this set as $\widehat{\bm{m}}_\alpha$ and define the matrix $\bm{N}^p$ that
contains these basis elements as 
\begin{align}
    \bm{N}^p := \begin{bmatrix}
        1 & 0 & 0 & \xi & 0 & 0 &\dots & \eta^\ell & 0 & 0 \\
        0 & 1 & 0 & 0 &\xi & 0 &\dots & 0 & \eta^\ell & 0 \\
        0 & 0 & 1 & 0 & 0 & \xi & \dots& 0 & 0 & \eta^\ell
    \end{bmatrix}.\label{matrix_N_p}
\end{align}
\end{subequations}

\subsection{Matrix-vector representation}
For later computations, it is more convenient to reduce the tensor expressions into equivalent matrix and vector representations. We first note that for plane elasticity we can express the components of the stress and strain tensors as symmetric $2\times 2$ matrices. \acrev{However, instead of using symmetric matrices, we adopt Voigt notation to represent the matrices as $3\times 1$ arrays. In particular, for any symmetric $2\times 2$ matrix $\bm{A}$, denote its Voigt representation $\overline{\bm{A}}$ by:}
\acrev{\begin{align*}
    \bm{A} = \begin{bmatrix}
        a_{11} & a_{12} \\ a_{12} & a_{22}
    \end{bmatrix}, \quad \overline{\bm{A}}= 
    \begin{Bmatrix}
        a_{11} \\ a_{22} \\ a_{12}
    \end{Bmatrix}.
\end{align*}}
On using Voigt (engineering) notation, we can write \acrev{the stress and strain} in terms of $3\times 1$ \acrev{arrays}:
\begin{align}\label{eq:matrixvector}
    \overline{\bm{\sigma}} = \begin{Bmatrix}
        \sigma_{11} \\ \sigma_{22} \\ \sigma_{12} 
    \end{Bmatrix}, \quad
    \overline{\bm{\varepsilon}} = \begin{Bmatrix}
       \varepsilon_{11} \\ \varepsilon_{22} \\ 2\varepsilon_{12}
  \end{Bmatrix}.
\end{align}
Furthermore, on using these conventions we can also express the strain-displacement relation and the constitutive law in matrix form as:
\begin{subequations}
\begin{align}
    \overline{\bm{\sigma}} &= \bm{C}\overline{\bm{\varepsilon}}, \quad
    \overline{\bm{\varepsilon}} = \bm{S}\bm{u}, \\
\intertext{where $\bm{S}$ is a matrix differential operator that is given by}
    \bm{S} &= \begin{bmatrix}
        \pder{x} & 0 \\ 0 & \pder{y} \\ \pder{y} & \pder{x}
    \end{bmatrix},
\end{align}
\end{subequations}
and $\bm{C}$ is the associated matrix representation of the material tensor \acmajor{that is given by
\begin{align*}
    \bm{C} &= \frac{E_Y}{(1-\nu^2)} \begin{bmatrix}
    1 & \nu & 0 \\ \nu & 1 & 0 \\ 0 & 0 & \frac{1-\nu}{2} 
    \end{bmatrix} \qquad \qquad \qquad \qquad  (\textrm{plane stress}),\\
    \bm{C} &= \frac{E_Y}{(1+\nu)(1-2\nu)} \begin{bmatrix}
    1-\nu & \nu & 0 \\ \nu & 1-\nu & 0 \\ 0 & 0 & \frac{1-2\nu}{2}
    \end{bmatrix} \quad (\textrm{plane strain}),
\end{align*}}
\acmajor{where $E_Y$ is the 
Young's modulus and $\nu$ is the Poisson's ratio of the material.  }
\section{Projection operators}\label{sec:projections}
We present the derivation of the two projections that are used in the stabilization-free VEM: energy projection of the displacement field and $L^2$ projection of the strain field.
\subsection{Energy projection of the displacement field}
Let $E$ be any generic element with
$\bm{H}^1(E) := [H^1(E)]^2$. 
We now define the energy projection operator $\proj :\bm{H}^1(E) \to [\bm{\mathbb{P}}_{1}(E)]^2 $ by the unique function that satisfies the orthogonality relation:
\begin{align}
    a^E(\bm{m}_\alpha,\bm{v}-\proj{\bm{v}}) = 0 \quad \forall \bm{m}_{\alpha} \in \bm{\widehat{M}}(E).
\end{align}
Note that for $\alpha =1,2,3$, which corresponds to the rigid-body modes, we obtain $\bm{\sigma}(\bm{m}_\alpha)=\bm{0}$. So we obtain three trivial equations, $0=0$. To fully define the projection,
we need to choose a suitable projection operator $P_0: \bm{H}^1(E) \times \bm{H}^1(E) \to \mathbb{R}$. In particular, we 
select it as a  discrete $L^2$ inner product on $E$: 
\begin{align}
    P_0(\bm{u},\bm{v}) := \frac{1}{N_E}\sum_{j=1}^{N_E}{\bm{u}(\bm{x}_j)\cdot \bm{v}(\bm{x}_j)},
\end{align}
and require the condition 
\begin{align}
    P_0(\bm{m}_\alpha,\bm{v}-\proj{\bm{v}}) =\frac{1}{N_E}\sum_{j=1}^{N_E} {(\bm{v}-\proj{\bm{v}})(\bm{x}_j) \cdot \bm{m}_\alpha(\bm{x}_j)} = 0 \quad (\alpha =1,2,3).\label{P0_condition}
\end{align}
On writing out the expressions, we have the equivalent system 
\begin{subequations}
\begin{align}
    \int_{E}{\bm{\sigma}(\bm{m}_\alpha):\bm{\varepsilon}(\proj{\bm{v}})\,d\bm{x}} &= \int_{E}{\bm{\sigma}(\bm{m}_\alpha):\bm{\varepsilon}(\bm{v})\,d\bm{x}} \quad (\alpha=4,5,6),\label{energy_eq_1} \\ 
    \frac{1}{N_E}\sum_{j=1}^{N_E}{\proj{\bm{v}}(\bm{x}_j)\cdot \bm{m}_\alpha(\bm{x}_j)} &= \frac{1}{N_E}\sum_{j=1}^{N_E}{\bm{v}(\bm{x}_j)\cdot \bm{m}_\alpha(\bm{x}_j)} \quad (\alpha =1,2,3)\label{energy_eq_2}.
\end{align}
\end{subequations}
We can also rewrite this using the matrix-vector representation. For the right-hand side of~\eqref{energy_eq_1}, we \acrev{use~\eqref{eq:matrixvector}} to write 
\begin{align*}
    \bm{\sigma}(\bm{m}_\alpha):\bm{\varepsilon}(\bm{v}) = \overline{\bm{\varepsilon}(\bm{v})}\cdot \overline{\bm{\sigma}(\bm{m}_\alpha)} &= \left(\overline{\bm{\varepsilon}(\bm{v})}\right)^T \overline{\bm{\sigma}(\bm{m}_\alpha)} \nonumber \\
    &=\left(\bm{S}\bm{v}\right)^T\left(\bm{CS}\bm{m}_{\alpha}\right).
\end{align*}
Similarly, the left-hand side can be written as
\begin{align*}
    \bm{\sigma}(\bm{m}_\alpha):\bm{\varepsilon}(\proj{\bm{v}}) &:= \left(\overline{\bm{\varepsilon}(\proj{\bm{v}})}\right)^T\left(\bm{CS}\bm{m}_{\alpha}\right) \nonumber 
    =\left(\bm{S}\proj{\bm{v}}\right)^T\left(\bm{CS}\bm{m_
    \alpha}\right).
\end{align*}
Therefore, we can express~\eqref{energy_eq_1} in matrix-vector form as:
\begin{align}
    \int_{E}{\left(\bm{S}\proj{\bm{v}}\right)^T\left(\bm{CS}\bm{m}_{\alpha}\right)\,d\bm{x}} &= \int_{E}{\left(\bm{S}\bm{v}\right)^T\left(\bm{CS}\bm{m}_{\alpha}\right)\,d\bm{x}}.\label{energy_matrix_eq_1}
\end{align}

\subsection{\texorpdfstring{L\textsuperscript{2}}{L2} projection of the strain field}
We define the associated $L^2$ projection operator $\projLtwo{.} :\bm{H}^1(E) \to \polysym$ of the strain tensor by the unique operator that satisfies 
\begin{subequations}\label{defL2proj}
\begin{align}
    \label{defL2proj-a}
    (\bm{\varepsilon}^p,\bm{\varepsilon}(\bm{v})- \projLtwo{\bm{v}})_E &= 0 \quad \forall \bm{\varepsilon}^p \in \polysym , \\
    \intertext{where we use the standard $L^2$ inner product:}
    \label{defL2proj-b}
    (\bm{\varepsilon}^p,\bm{\varepsilon})_E &= \int_{E}{\bm{\varepsilon}^p:\bm{\varepsilon} \,d\bm{x}}.
\end{align}
\end{subequations}
Writing out the expression in~\eqref{defL2proj-a}, we have
\begin{align}
    \int_{E}{\bm{\varepsilon}^p}:\projLtwo{\bm{v}} \,d\bm{x} = \int_{E}{\bm{\varepsilon}^p:\bm{\varepsilon}(\bm{v})\,d\bm{x}}.\label{L2_tensor_eq}
\end{align}
On expanding the right-hand side of~\eqref{L2_tensor_eq}, and on
applying integration by parts and the divergence theorem, we obtain
\begin{align*}
    \int_{E}{\bm{\varepsilon}^p:\bm{\varepsilon}(\bm{v})}d\bm{x} &= \int_{E}{\grad  \cdot (\bm{v}\cdot \bm{\varepsilon}^p)\,d\bm{x} }- \int_{E}{\bm{v}\cdot (\diverge{\bm{\varepsilon}^p})\,d\bm{x}} \\ &= \int_{\partial E}{\bm{n}\cdot (\bm{v} \cdot \bm{\varepsilon}^p)\,ds} -\int_{E}{\bm{v}\cdot (\diverge{\bm{\varepsilon}^p})\,d\bm{x}}. 
\end{align*}
Then,~\eqref{L2_tensor_eq} becomes
\begin{equation}\label{L2_eq}
    \int_{E}{\bm{\varepsilon}^p:\projLtwo{\bm{v}}}\,d\bm{x} =  \int_{\partial E}{\bm{v}\cdot (\bm{\varepsilon}^p\cdot \bm{n})\,ds} -\int_{E}{\bm{v}\cdot (\diverge{\bm{\varepsilon}^p})\,d\bm{x}}.
\end{equation}
On using the matrix-vector representation in~\eqref{eq:matrixvector},
the first term on the right-hand side of~\eqref{L2_eq} becomes
\begin{subequations}
\begin{align}
    \bm{v} \cdot (\bm{\varepsilon}^p\cdot \bm{n} ) =  \bm{v}^T\begin{bmatrix}
        \varepsilon^p_{11} & \varepsilon^p_{12} \\ \varepsilon^p_{12} & \varepsilon^p_{22}
    \end{bmatrix}\begin{Bmatrix}
        n_1 \\ n_2
    \end{Bmatrix} 
    &= \bm{v}^T\begin{bmatrix}
        n_1 & 0 & n_2 \\
        0 & n_2 & n_1
    \end{bmatrix}
    \begin{Bmatrix}
        \varepsilon^p_{11} \\ \varepsilon^p_{22} \\ \varepsilon^p_{12}
    \end{Bmatrix} 
    :=\bm{v}^T \acmajor{\Nmatrix}\overline{\bm{\varepsilon}^p},
\intertext{where $\acmajor{\Nmatrix}$ is the matrix of element normal components, which is defined as}
    \acmajor{\Nmatrix} &:= \begin{bmatrix}
        n_1 & 0 & n_2 \\
        0 & n_2 & n_1 
    \end{bmatrix}.\label{normal_matrix}
\end{align}
\end{subequations}
For the second term on the right-hand side of~\eqref{L2_eq},  we have
\begin{subequations}
\begin{align}
    \bm{v}\cdot (\diverge{\bm{\varepsilon}^p}) = \bm{v}^T\begin{Bmatrix}
        \frac{\partial \varepsilon^p_{11}}{\partial x} + \frac{\partial \varepsilon^p_{12}}{\partial y} \\ 
        \frac{\partial \varepsilon^p_{12}}{\partial x} + \frac{\partial \varepsilon^p_{22}}{\partial y}
    \end{Bmatrix} 
    &= \bm{v}^T\begin{bmatrix}
    \pder{x} & 0 &\pder{y} \\
    0 & \pder{y} & \pder{x}
    \end{bmatrix}
    \begin{Bmatrix}
        \varepsilon^p_{11} \\ \varepsilon^p_{22} \\ \varepsilon^p_{12}
    \end{Bmatrix} 
    :=\bm{v}^T\bm{\partial} \overline{\bm{\varepsilon}^p}, \\
\intertext{where $\bm{\partial}$ is a matrix operator that is defined as}
\bm{\partial} &:=\begin{bmatrix}
    \pder{x} & 0 &\pder{y} \\
    0 & \pder{y} & \pder{x}
    \end{bmatrix}.
\end{align}
\end{subequations}
Now we can express~\eqref{L2_eq} as 
\begin{equation}\label{eq:L2proj_1}
    \int_{E}{\bm{\varepsilon}^p:\projLtwo{\bm{v}}}\,d\bm{x} = \int_{\partial E}{\bm{v}^T \acmajor{\Nmatrix}\overline{\bm{\varepsilon}^p} \,ds} + \int_{E}{\bm{v}^T\bm{\partial} \overline{\bm{\varepsilon}^p}\,d\bm{x}}.
\end{equation}

Since $\projLtwo{\bm{v}}$ is the projection of the strain tensor onto symmetric matrix polynomials, we 
use~\eqref{eq:matrixvector} 
to represent it in terms of a vector. In particular, we set 
\begin{equation*}
    \overline{\projLtwo{\bm{v}}} = 
    \begin{Bmatrix}
        (\projLtwo{\bm{v}})_{11} \\ (\projLtwo{\bm{v}})_{22} \\ 2(\projLtwo{\bm{v}})_{12}
    \end{Bmatrix}.
\end{equation*}
Now, we can also write 
\begin{equation*}
     \bm{\varepsilon}^p:\projLtwo{\bm{v}} = \overline{\projLtwo{\bm{v}}} \cdot \overline{\bm{\varepsilon}^p} = \left(\overline{\projLtwo{\bm{v}}}\right)^T \overline{\bm{\varepsilon}^p} .
\end{equation*}
On using the above relations in~\eqref{eq:L2proj_1}, we seek 
the $L^2$ projection that satisfies
\begin{equation}
    \int_{E}{\left(\overline{\projLtwo{\bm{v}}}\right)^T \overline{\bm{\varepsilon}^p}\,d\bm{x}} = \int_{\partial E}{\bm{v}^T \acmajor{\Nmatrix}\overline{\bm{\varepsilon}^p} \,ds} + \int_{E}{\bm{v}^T\bm{\partial} \overline{\bm{\varepsilon}^p}\,d\bm{x}} \quad \forall  {\bm{\varepsilon}^p} \in \polysym. \label{L2_matrix_eq}
\end{equation}

\section{Enlarged Enhanced Virtual Element Space}\label{sec:EEVEM}
With the preliminary results in place, we now construct the discrete space for the stabilization-free virtual element method. Let $E$ be any polygonal element from $\mathcal{T}^h$, then 
following~\cite{berrone2021lowest},
we select the smallest value $\ell=\ell(E)$
that satisfies\footnote{In~\cite{enhanced:VEM},
the following inequality for $\ell=\ell(E)$ is proposed: $\frac{3}{2}(\ell+1)(\ell+2)\geq m-3$, 
where $m$ is the total number of degrees of freedom, 
which includes an additional $\ell(\ell+1)$ degrees of freedom due to
extending the vector polynomial approximation space.
\acmajor{However, a counterexample on regular polygons (A. Russo, personal communication, April 2022) shows that this condition is not sufficient.}}
\begin{align}\label{l_condition} 
    \acmajor{\frac{3}{2}(\ell+1)(\ell+2) - \text{dim}(\Pker) \geq 2N_E-3,}
\end{align}
\acmajor{where $N_E$ is the number of vertices (nodes) of element $E$ and $\Pker$ is the space defined by }
\begin{align*}
    \acmajor{\Pker:=\left\{ \bm{\varepsilon}^p \in \polysym :\int_{\partial E}{\left(\bm{v}-P_{r}(\bm{v})\right)\rvert_{\partial E}\cdot \left(\bm{\varepsilon}^p \cdot \bm{n}\right)\,ds}=0 \ \ \forall \bm{v} \right\}, } 
\end{align*}
\acmajor{where $P_{r}(\bm{v})$ is a projection of $\bm{v}$ onto rigid-body modes with $\bm{\varepsilon}(P_r(\bm{v}))=\bm{0}$. 
It can be shown that the dimension of the space $\Pker$ is bounded from above, and we include this result as a lemma.}
\acmajor{\begin{lemma}
Let $E$ be any polygonal element and $\ell \in \mathbb{N}$. Then 
\begin{align}\label{eq:pker_bound}
    \text{dim}({\Pker}) \leq \frac{\ell}{2}(3\ell+1).
\end{align}
\end{lemma}}
\acmajor{\begin{proof}
Following~\cite{berrone2021lowest}, we define for each element $E$, the subspace of polynomials
\begin{align*}
    \tilde{\mathbb{H}}_{\ell+1}(E) = \{\bm{p} \in [\mathbb{P}_{\ell+1}(E)]^2: \nabla \cdot \bm{\sigma}(\bm{p}) = \bm{0}\}. 
\end{align*}
For a given $\ell$, this space is shown to have dimension $4\ell+6$ in~\cite{Cao:2009:advances}. We then consider the space $\bm{\sigma}(\tilde{\mathbb{H}}_{\ell+1}(E))$, and it can be shown that this space has dimension $4\ell+3$. Both $\Pker$ and $\bm{\sigma}(\tilde{\mathbb{H}}_{\ell+1}(E))$ are subspaces of $\polysym$, so the sum $\Pker + \bm{\sigma}(\tilde{\mathbb{H}}_{\ell+1}(E)) $ is also a subspace and the dimension is bounded by:
\begin{align*}
    \text{dim}(\polysym) &\geq \text{dim}(\Pker + \bm{\sigma}(\tilde{\mathbb{H}}_{\ell+1}(E)) )\\ &=\text{dim}(\Pker) + \text{dim}({\bm{\sigma}(\tilde{\mathbb{H}}_{\ell+1}(E))} - \text{dim}({\Pker \cap \bm{\sigma}(\tilde{\mathbb{H}}_{\ell+1}(E))}).
\end{align*}
Now we show that ${\Pker \cap \bm{\sigma}(\tilde{\mathbb{H}}_{\ell+1}(E))} = \{\bm{0}\}$. To this end, let $\bm{p} \in \tilde{\mathbb{H}}_{\ell+1}(E)$, and assume 
that $\bm{\sigma}(\bm{p}) \in \Pker$. Then we have for any $\bm{v}\in \bm{H}^1(E)$, 
\begin{align*}
    \int_{E}{\nabla \cdot \bm{\sigma}(\bm{p})\cdot (\bm{v}-P_r(\bm{v})) \ d\bm{x}} = 0.
\end{align*}
On applying the divergence theorem and using the definition of $\Pker$, we can write 
\begin{align*}
    \int_{\partial E}{(\bm{v}-P_r({\bm{v}}))\rvert_{\partial E}\cdot (\bm{\sigma}(\bm{p})\cdot \bm{n})\ ds} - \int_{E}{\bm{\varepsilon}(\bm{v}-P_r(\bm{v})):\bm{\sigma}(\bm{p})} \, d\bm{x} = -\int_{E}{\bm{\varepsilon}(\bm{v}):\bm{\sigma}(\bm{p})} = 0. 
\end{align*}
This is true for all $\bm{v}$, which implies that $\bm{\sigma}(\bm{p})=\bm{0}$. Otherwise, suppose this is not true, then following a similar argument from~\cite{berrone2021lowest}, there exists an open set $\omega \subset E$ such that $\bm{\sigma}(\bm{p}) \neq \bm{0}$ and in particular $\bm{p} \neq \bm{0}$ over $\omega$ . Now define a (smooth) bump function by:
\begin{align*}
    \begin{cases}
    -\nabla \cdot \bm{\sigma}(\bm{b}_\omega) = \bm{p} \quad &\text{in } \omega,  \\
    \bm{b}_\omega = \bm{0} \quad &\text{on } E\setminus \omega. 
    \end{cases}
\end{align*}
Then, we consider
\begin{align*}
    0=(\bm{\sigma}(\bm{p}),\bm{\varepsilon}(\bm{b}_{\omega}))_{E}=(\bm{\sigma}(\bm{p}),\bm{\varepsilon}(\bm{b}_{\omega}))_{\omega} = (\bm{\varepsilon}(\bm{p}),\bm{\sigma}(\bm{b}_{\omega}))_{\omega} \end{align*}
On applying the divergence theorem, we obtain
\begin{align*}
    0=(\bm{\varepsilon}(\bm{p}),\bm{\sigma}(\bm{b}_{\omega}))_{\omega} = \int_{\omega}{\bm{\varepsilon}(\bm{p}) : \bm{\sigma}(\bm{b}_{\omega}) \ d\bm{x}} &= \int_{\partial \omega}{(\bm{\sigma}(\bm{b}_\omega) \cdot \bm{n})\cdot \bm{p}\ ds} - \int_{\omega}{\bm{p}\cdot (\diverge{\bm{\sigma}(\bm{b}_{\omega})}) \ d\bm{x}} \\ 
    & = \int_{\omega}{\bm{p}\cdot \bm{p} \ d\bm{x}} > 0 ,
\end{align*}
which leads to a contradiction, and therefore $\bm{\sigma}(\bm{p}) =\bm{0}$ holds on $E$. This implies that
 $\Pker \cap \bm{\sigma}(\tilde{\mathbb{H}}_{\ell+1}(E)) = \{\bm{0}\}$.
Now it follows that
\begin{align*}
    \text{dim}({\Pker}) &\leq \text{dim}(\polysym) -\text{dim}({\bm{\sigma}(\tilde{\mathbb{H}}_{\ell+1}(E))})+ \text{dim}({\Pker \cap \bm{\sigma}(\tilde{\mathbb{H}}_{\ell+1}(E))})\\ &= \frac{3}{2}(\ell+1)(\ell+2) - (4\ell+3) \\ &= \frac{\ell}{2}(3\ell+1).    \qedhere
\end{align*}
\end{proof}}
\acmajor{Combining~\eqref{eq:pker_bound} and~\eqref{l_condition}, we get a sufficient bound on the number of vertices required for any $\ell$. In particular, we have a more restrictive bound: 
\begin{align}\label{eq:strong_l_condition}
    N_E \leq 2\ell +3.
\end{align}}

On using this value of $\ell$, we define the set of all functions $\bm{v} \in \bm{H}^1(E)$ that satisfy the property that the inner product of the function and any vector polynomial in $[\mathbb{P}_{\ell+1}(E)]^2$ is equal to that of the inner product with the energy projection. That is, we define the set $\largeset$ as
\begin{align}
   \largeset=\left\{ \bm{v} :  \int_{E}{\bm{v}\cdot \bm{p}\,d\bm{x}} = \int_{E}{\proj{\bm{v}}\cdot \bm{p} \,d\bm{x} } \ \ \forall \bm{p} \in [\mathbb{P}_{\ell+1}(E)]^2\right\}.\label{enrichment_condition}
\end{align}
We define the local enlarged virtual element space as:
\begin{align}
    \bm{V}_{1,\ell}^{E} := \left\{ \bm{v}_h \in \largeset : \Delta \bm{v}_h \in [\mathbb{P}_{\ell+1}(E)]^2 , \ \gamma^{e}(\bm{v}_h) \in [\mathbb{P}_{1}(e)]^2  \ \forall e \in \mathcal{E}_E , \ \bm{v}_h \in [C^{0}(\partial E)]^2\right\} \!, 
    \label{local_vem} 
\end{align}
where $\gamma^{e_i}(\cdot)$ is the trace of a function (its argument) on an edge $e_i$.
In the above space we require functions to be linear on the edges, in which case we can take the degrees of freedom to be the values of the function at the vertices of the polygon $E$. There will be a total of $2N_E$ degrees of freedom on each element $E$.

With the local space so defined, we define the global enlarged virtual element space as
\begin{align}
    \bm{V}_{1,\bm{\ell}} := \{\bm{v}_h \in [H^1(\Omega)]^2 : \bm{v}_h\rvert_E \in \bm{V}_{1,\ell}^{E} \ \ \text{for } \ell=\ell(E)\}.
\end{align}
For each $E$, we assign a suitable basis to the local virtual element space $\VEM$. Let $\{\phi_{i}\}$ be the set of generalized barycentric coordinates (canonical basis functions)~\cite{Hormann:2017:GBC} that satisfy $\phi_i(\bm{x}_j)= \delta_{ij}$. 
We express the components of  any  $\bm{v}_h \in \VEM$ as the sum of these basis functions:
\begin{subequations}\label{vem_basis_matrix}
\begin{align}
\bm{v}_h &= \begin{Bmatrix}
    v_h^1 \\ v_h^2
\end{Bmatrix}   
=\begin{bmatrix}
    \phi_1 & \phi_2 & \dots &\phi_{N_E} & 0 &0 & \dots & 0 \\
    0 & 0 & \dots & 0 & \phi_1 & \phi_2 &\dots & \phi_{N_E}
\end{bmatrix}\begin{Bmatrix}
    v^1_1 \\ v^1_2 \\ \vdots \\ v^2_{N_E}
\end{Bmatrix}
:= \bm{N}^v \tilde{\bm{v}}_h, \\
\intertext{where we define $\bm{N}^v$ as the matrix of vectorial basis functions:}
    \bm{N}^v &= \begin{bmatrix}
        \phi_1 & \phi_2 & \dots &\phi_{N_E} & 0 &0 & \dots & 0 \\
    0 & 0 & \dots & 0 & \phi_1 & \phi_2 &\dots & \phi_{N_E}
    \end{bmatrix} 
    := \begin{bmatrix}
        \bm{\varphi}_1
        & \dots & \bm{\varphi}_{N_E} 
        & \dots & \bm{\varphi}_{2N_E} 
    \end{bmatrix}.
\end{align}
\end{subequations}

We now define the weak form of the virtual element method on this space.  On defining a discrete bilinear operator $a^E_h : \VEM \times \VEM \to \mathbb{R} $ and a discrete linear functional $b^{E}_h: \VEM \to \mathbb{R}$, we seek the solution to the problem: find $\bm{u}_h \in \VEM$ such that
\begin{align}\label{vem_weak_problem}
a_h^E(\bm{u}_h,\bm{v}_h) = b^{E}_h(\bm{v}_h) \quad \forall \bm{v}_h \in \VEM .
\end{align}
Following~\cite{berrone2021lowest}, we introduce the local discrete bilinear form in matrix-vector form:
\begin{align}
    a_h^E(\bm{u}_h,\bm{v}_h) := \int_{E}{\left(\overline{\projLtwo{\bm{v_h}}}\right)^T \bm{C} \, \overline{\projLtwo{\bm{u}_h}} \, d\bm{x}},
\end{align}
with the associated global operator defined as 
\begin{align}\label{global_bilinearform}
    a_h(\bm{u}_h,\bm{v}_h) := \sum_{E}{a_h^E(\bm{u}_h,\bm{v}_h)}.
\end{align}
We also define a local linear functional by 
\begin{align}
    b^E_h(\bm{v}_h)= \int_{E}{\bm{v}_h^T\bm{f}_h \, d\bm{x}} + \int_{\Gamma_N \cap \partial E}{\bm{v}_h^T\bar{\bm{t}} \, ds},\label{local_force}
\end{align}
with the associated global functional 
\begin{align}
    b_h(\bm{v}_h) = \sum_{E}{b^E_h(\bm{v}_h)},
\end{align}
where $\bm{f}_h$ is some approximation to $\bm{f}$. For first order methods it is sufficient to consider the $L^2$ projection onto constants, namely 
$\bm{f}_h = \bm{\Pi}_0^0 \bm{f}$. 

\section{Numerical Implementation}\label{sec:implementation}
With the definitions of the discrete spaces and projections on hand, we now 
detail the implementation of the method. We 
present the derivation of the equations 
to compute the energy projection, the $L^2$ projection, and the element stiffness matrix. 

\subsection{Implementation of energy projector}
We start with the energy projection. From (\ref{energy_matrix_eq_1}), we have for $\alpha =4,5,6,$ the equation
\begin{align*}
    \int_{E}{\left(\bm{S}\proj{\bm{v}_h}\right)^T\left(\bm{CS}\bm{m}_{\alpha}\right)\,d\bm{x}} &= \int_{E}{\left(\bm{S}\bm{v}_h\right)^T\left(\bm{CS}\bm{m}_{\alpha}\right)\,d\bm{x}}.
\end{align*}
In particular, we are interested in the case when $\bm{v}_h = \bm{\varphi}_i$, the basis functions in $\VEM$.
By definition of the energy projection, $\proj{\bm{\varphi}_i}$ is a vector polynomial of degree one. Therefore, we can expand it in terms of its basis functions:
\begin{align}
    \proj{\bm{\varphi}_i} = 
    \sum_{\beta=1}^{6}{s^i_{\beta}\bm{m}_{\beta}}.\label{expand_basis}
\end{align}
We can express the left-hand side as 
\begin{align}
    \int_{E}{\left(\bm{S}\proj{\bm{\varphi}_i}\right)^T(\bm{CS}\bm{m}_{\alpha})\,d\bm{x}} = \sum_{\beta=1}^{6}s^i_{\beta}\int_{E}{\left(\bm{S}\bm{m}_\beta\right)^T(\bm{CS}\bm{m}_{\alpha})\,d\bm{x}}.
\end{align}
Define the matrix $\tilde{\bm{G}}$ for $\beta=1,2,\dots,6$, and $\alpha=4,5,6$ by 
\begin{align}
    \tilde{\bm{G}}_{\alpha \beta} = \int_{E}{\left(\bm{S}\bm{m}_\beta\right)^T(\bm{CS}\bm{m}_{\alpha})\, d\bm{x}}.
\end{align}
Similarly, the matrix $\tilde{\bm{B}}$ representing the right-hand side of~\eqref{energy_matrix_eq_1} becomes
\begin{align}
    \tilde{\bm{B}}_{\alpha i}= \int_{E}{\left(\bm{S}\bm{\varphi}_i\right)^T\left(\bm{CS}\bm{m}_{\alpha}\right)\,d\bm{x}}.
\end{align}

To fully define these matrices for all $\alpha$, we consider
the additional projection equation~\eqref{energy_eq_2}. When $\bm{v} = \bm{\varphi}_i$, we obtain 

\begin{align*}
     \frac{1}{N_E}\sum_{j=1}^{N_E}{\proj{\bm{\varphi}_i}(\bm{x}_j)\cdot \bm{m}_\alpha(\bm{x}_j)} &= \frac{1}{N_E}\sum_{j=1}^{N_E}{\bm{\varphi}_i(\bm{x}_j)\cdot \bm{m}_\alpha(\bm{x}_j)}.
\end{align*}
As we have done previously, on expanding $\proj{\bm{\varphi}_i}$ with~\eqref{expand_basis} leads to
\begin{align}
        \sum_{\beta=1}^{6}{ s^i_{\beta}\frac{1}{N_E} \sum_{j=1}^{N_E}{\bm{m}_\beta(\bm{x}_j)\cdot \bm{m}_\alpha(\bm{x}_j)}} &= \frac{1}{N_E}\sum_{j=1}^{N_E}{\bm{\varphi}_i(\bm{x}_j)\cdot \bm{m}_\alpha(\bm{x}_j)}.
\end{align}
Now we can define the remaining $\alpha =1,2,3$ terms of the matrices $\tilde{\bm{G}}$ and $\tilde{\bm{B}}$ as 
\begin{align}
    \tilde{\bm{G}}_{\alpha \beta} = \frac{1}{N_E}\sum_{j=1}^{N_E}{\bm{m}_\beta(\bm{x}_j)\cdot \bm{m}_\alpha(\bm{x}_j)},
    \quad \tilde{\bm{B}}_{\alpha i} =\frac{1}{N_E}\sum_{j=1}^{N_E}{\bm{\varphi}_i(\bm{x}_j)\cdot \bm{m}_\alpha(\bm{x}_j)}. 
\end{align}
Combining the results, we obtain $\tilde{\bm{G}}$ for all $\beta=1,2,\dots,6$:
\begin{subequations}
\begin{align}
    \tilde{\bm{G}}_{\alpha \beta} =
    \begin{cases}
    \frac{1}{N_E}\sum_{j=1}^{N_E}{\bm{m}_\beta(\bm{x}_j)\cdot \bm{m}_\alpha(\bm{x}_j)} \quad (\alpha =1,2,3) \\
    \int_{E}{\left(\bm{S}\bm{m}_\beta\right)^T(\bm{CS}\bm{m}_{\alpha})d\bm{x}} \quad 
    (\alpha= 4,5,6),
    \end{cases}
\intertext{and for all $i=1,2\dots,2N_E$, we have}
    \tilde{\bm{B}}_{\alpha i} =
    \begin{cases}
    \frac{1}{N_E}\sum_{j=1}^{N_E}{\bm{\varphi}_i(\bm{x}_j)\cdot \bm{m}_\alpha(\bm{x}_j)} \quad (\alpha= 1,2,3) \\
    \int_{E}{\left(\bm{S}\bm{\varphi}_i\right)^T\left(\bm{CS}\bm{m}_{\alpha}\right)d\bm{x}} \quad (\alpha= 4,5,6) . \label{B_tilde}
    \end{cases}
\end{align}
\end{subequations}
After combining these equations, we can determine
the coefficients for the projection as the solution of the system:
\begin{align}
    \tilde{\bm{G}} \proj = \tilde{\bm{B}},
\end{align}
where $(\proj{})_{\beta i} = s^i_{\beta}$. We start by considering the matrix $\tilde{\bm{G}}$. For $\alpha=1,2,3$, $\tilde{\bm{G}}$ is 
the sum of polynomials evaluated at the vertex points, 
which can be directly computed. 
For $\alpha=4,5,6$, since
the basis \acrev{functions} $\bm{m}_{\alpha}$ are linear, the matrix differential operator 
acting on $\bm{m}_{\alpha}$ will result in a constant vector. For a constant material matrix $\bm{C}$, the expression $\left(\bm{S}\bm{m}_\beta\right)^T(\bm{C} \bm{S}\bm{m}_{\alpha})$ is a constant matrix. Therefore, we can write: 
\begin{align*}
    \tilde{\bm{G}}_{\alpha \beta} = \left(\bm{S}\bm{m}_\beta\right)^T(\bm{CS}\bm{m}_{\alpha}) |E| \quad (\alpha=4,5,6) .
\end{align*}
On using~\eqref{B_tilde} and simplifying, we can write $\tilde{\bm{B}}$ for $\alpha = 1,2,3$ as
\begin{align*}
    \tilde{\bm{B}}_{\alpha i} = 
    \begin{cases}
   \frac{1}{N_E}\sum_{j=1}^{N_E}{\begin{pmatrix}
      \phi_{i}(\bm{x}_j) \\ 0
   \end{pmatrix}} \cdot \bm{m}_{\alpha}(\bm{x}_j) =\frac{1}{N_E}{m^{1}_\alpha(\bm{x}_i)} \quad (i=1,2,\dots, N_E) \\
    \frac{1}{N_E}\sum_{j=1}^{N_E}{\begin{pmatrix}
         0 \\ \phi_{i}(\bm{x}_j)
    \end{pmatrix}} \cdot \bm{m}_{\alpha}(\bm{x}_j)=\frac{1}{N_E}{m^{2}_\alpha(\bm{x}_i)} \quad (i=N_E+1,N_E+2,\dots,2N_E),
    \end{cases}
\end{align*}
where  $m^{k}_\alpha$ is the $k$-th component of $\bm{m}_{\alpha}$. 
For $\alpha = 4,5,6$, we can apply the definition of the matrix differential operator 
and use the divergence theorem to write
\begin{align*}
    \tilde{\bm{B}}_{\alpha i} = \int_{E}{\left(\bm{S}\bm{\varphi}_i\right)^T\left(\bm{CS}\bm{m}_{\alpha}\right)\,d\bm{x}}
    =\left(\int_{E}{
    \left(\bm{S}\bm{\varphi}_i\right)^T\,d\bm{x}}\right)\bm{CS}\bm{m}_{\alpha}
    = \sum_{j=1}^{N_E} \left(\int_{e_j}{\bm{\varphi}_i}^T\acmajor{\Nmatrix} \,ds\right) \bm{CS}\bm{m}_{\alpha},
\end{align*}
where $e_j$ is the $j$-th edge of the element $E$ and $\acmajor{\Nmatrix}$ is the matrix of normal components given in~\eqref{normal_matrix}.
On simplification, we obtain for $\alpha=4,5,6$,
\begin{align*}
\tilde{\bm{B}}_{\alpha i}=
    \begin{cases}
     \Bigl(\int_{e_{i-1}}{\begin{pmatrix}
        \phi_i n_1^{(i-1)} & 0 & \phi_i n_2^{(i-1)}
    \end{pmatrix}\,ds} \\ \qquad + \int_{e_{i}}{\begin{pmatrix}
        \phi_i n_1^{(i)} & 0 & \phi_i n_2^{(i)}
    \end{pmatrix}\,ds} \Bigr) \bm{CS}\bm{m}_{\alpha} \quad (i=1,2,\dots ,N_E)\\
    \Bigl(\int_{e_{i-1}}{\begin{pmatrix}
        0 &\phi_{i}n_2^{(i-1)} & \phi_{i}n_1^{(i-1)}
    \end{pmatrix}ds} \\ \qquad + \int_{e_{i}}{\begin{pmatrix}
        0 & \phi_{i}n_2^{(i)} & \phi_{i} n_1^{(i)}
    \end{pmatrix}ds} \Bigr) \bm{CS}\bm{m}_{\alpha} \quad (i=N_E+1,N_E+2,\dots,2N_E).
    \end{cases}
\end{align*}
These are integrals of a linear function over a line segment, which are exactly computed using
a two-point Gauss-Lobatto quadrature scheme.

\subsection{Implementation of \texorpdfstring{L\textsuperscript{2}}{L2} projector }
Now that we have a computable form of the energy projection, we can construct the $L^2$ projection. From~\eqref{L2_matrix_eq}, we have
\begin{align}\label{eq:L2proj}
        \int_{E}{\left(\overline{\projLtwo{\bm{v}_h}}\right)^T \overline{\bm{\varepsilon}^p} \, d\bm{x}} = \int_{\partial E}{\bm{v}_h^T \acmajor{\Nmatrix}\overline{\bm{\varepsilon}^p} \,ds} + \int_{E}{\bm{v}_h^T\partial \overline{\bm{\varepsilon}^p}\,d\bm{x}}. 
\end{align}

On expanding $\bm{v}_h$ in terms of its basis in $\VEM$, we obtain
$\bm{v}_h = \bm{N}^v\tilde{\bm{v}}_h$. We can also expand the symmetric function $\overline{\bm{\varepsilon}^p}$ in terms of the polynomial basis in $\polysym$ with $\overline{\bm{\varepsilon}^p} = \bm{N}^p \tilde{\bm{\varepsilon}}^p$. Following~\cite{elastic}, we also define a matrix $\bm{\Pi}^m$ such that we can write the projected strain in terms of the polynomial basis in $\polysym$. In particular, we write   
\begin{align*}
\overline{\projLtwo{\bm{v}_h}} = \bm{N}^p\bm{\Pi}^m \tilde{\bm{v}}_h, 
\end{align*}
Substituting these into~\eqref{eq:L2proj}, we obtain 
\begin{align*}
    \int_{E}{\left( \bm{N}^p\bm{\Pi}^m \tilde{\bm{v}}_h\right)^T \bm{N}^p \tilde{\bm{\varepsilon}}^p \, d\bm{x}} = \int_{\partial E}{\left(\bm{N}^v\tilde{\bm{v}}_h\right)^T \acmajor{\Nmatrix}\bm{N}^p \tilde{\bm{\varepsilon}}^p \, ds} + \int_{E}{\left(\bm{N}^v\tilde{\bm{v}}_h\right)^T\bm{\partial} \bm{N}^p \tilde{\bm{\varepsilon}}^p \, d\bm{x}} ,
\end{align*}
which on simplifying becomes 
\begin{align*}
        \int_{E}{\tilde{\bm{v}}_h^T\left(\bm{\Pi}^m \bm{N}^p\right)^T\bm{N}^p \tilde{\bm{\varepsilon}}^p \, d\bm{x}} = \int_{\partial E}{\tilde{\bm{v}}_h^T(\bm{N}^v)^T\acmajor{\Nmatrix}\bm{N}^p\tilde{\bm{\varepsilon}}^p \, ds} + \int_{E}{\tilde{\bm{v}}_h^T(\bm{N}^v)^T\bm{\partial} \bm{N}^p\tilde{\bm{\varepsilon}}^p \, d\bm{x}} .
\end{align*}
Since this is true for all $\tilde{\bm{v}}_h$ and $\tilde{\bm{\varepsilon}}^p$, we can rewrite the equation as: 
\begin{align*}
        (\tilde{\bm{\varepsilon}}^p)^T\left(\int_{E}{\left(\bm{N}^p\right)^T\bm{N}^p d\bm{x} }\right)\bm{\Pi}^m\tilde{\bm{v}}_h = (\tilde{\bm{\varepsilon}}^p)^T\left(\int_{\bm{\partial} E}{\left(\acmajor{\Nmatrix}\bm{N}^p\right)^T\bm{N}^v ds } - \int_{E}{\left(\bm{\partial} \bm{N}^p\right)^T\bm{N}^v d\bm{x}}\right)\tilde{\bm{v}}_h.    
\end{align*}
So now we can solve for the projection matrix $\bm{\Pi}^m$ via 
\begin{subequations}\label{L2_projection_matrices}
\begin{align}
    &\qquad \qquad \qquad \bm{\Pi}^m = \bm{G^{-1}B},\label{eq:Pi-m} \\ 
    \intertext{where $\bm{G}$ and $\bm{B}$ are defined as}
    \bm{G} &:= \int_{E} {\left(\bm{N}^p\right)^T\bm{N}^p \, d\bm{x} }, \label{G_matrix}\\
    \bm{B} &:=\int_{\partial E} 
    {\left(\acmajor{\Nmatrix}\bm{N}^p\right)^T\bm{N}^v  \, ds}  - \int_{E}{\left(\bm{\partial} \bm{N}^p\right)^T\bm{N}^v \, d\bm{x}}\label{B_matrix}. 
\end{align}
\end{subequations}
We can explicitly construct the forms for $\bm{G}$ and $\bm{B}$. 
From~\eqref{G_matrix}, we expand the integrand $(\bm{N}^p)^T \bm{N}^p$, 
where $\bm{N}^p$ is given by~\eqref{matrix_N_p}.
If we let $\bm{I}$ be the $3\times 3$ identity matrix, we can write $\bm{N}^p$ as 
\begin{equation*}
        \bm{N}^p := 
    \begin{bmatrix}
    \bm{I} & \xi \bm{I} &\eta \bm{I} \dots \eta^{\ell} \bm{I} 
    \end{bmatrix}
\end{equation*}
and the product $(\bm{N}^p)^T \bm{N} ^p$ can be written in compact form as: 
\begin{align*}
    {(\bm{N}^p)^T\bm{N}^p} = \begin{bmatrix}
    \bm{I} & \xi \bm{I} & \eta \bm{I} & \dots & \eta^{l} \bm{I} \\ 
    \xi \bm{I} & \xi^2 \bm{I} & \xi \eta \bm{I} & \dots & \xi \eta ^{\ell} \bm{I} \\ 
    \eta \bm{I} & \xi \eta \bm{I} & \ddots  & \vdots \\
    \vdots &\vdots & \dots \\ 
    \eta^\ell \bm{I} & \xi \eta^\ell \bm{I} & \eta^{\ell+1} \bm{I} &\dots & \eta^{2\ell}\bm{I} 
    \end{bmatrix}.
\end{align*}
Integrating each term of the matrix, we find that we only need to determine integrals of the form
\begin{equation*}
    \int_{E}{\xi^r \eta ^k \, d\bm{x}}  \quad \text{for } 0 \leq r+k\leq 2\ell,
\end{equation*}
which can be computed either by 
partitioning $E$
into triangles and then adopting a Gauss quadrature rule on triangles or by using the schemes developed 
in~\cite{2021,chin:hal-01426581}.

The construction of the $\bm{B}$ matrix reveals the major difference between the 
stabilization-free method and a standard VEM for plane elasticity. 
For the first term in~\eqref{B_matrix}, we expand the integral over $\partial E$ as the sum of integrals over edge $e_i$:
\begin{equation*}
    \int_{\partial E}{\left(\acmajor{\Nmatrix}\bm{N}^p\right)^T\bm{N}^v \, ds } = \sum_{i=1}\int_{e_i}{\left(\acmajor{\Nmatrix}\bm{N}^p\right)^T\bm{N}^v \rvert_{e_i}\,ds}.
\end{equation*}
Now we examine $\bm{N}^v \rvert_{e_i}$, 
\begin{equation*}
    \bm{N}^v \rvert_{e_i} = \begin{bmatrix}
    \phi_1 & \phi_2 &\dots& \phi_{N_E} & 0 &\dots &0 \\
    0 & 0 & \dots & 0 & \phi_1 & \phi_2 & \dots &\phi_{N_E} 
    \end{bmatrix}\biggr\rvert_{e_i}.
\end{equation*}
By definition of the Lagrange property, each $\phi_i$ is only nonzero when evaluated at the $i$-th degree of freedom, therefore the only contributions along the edge $e_i$ \acrev{are from} $\phi_i \rvert_{e_i}$ and $\phi_{i+1} \rvert_{e_i}$. As a consequence, $\bm{N}^v \rvert_{e_i}$ has only four nonzero elements, namely
\begin{equation}\label{Nv}
        \bm{N}^v \rvert_{e_i} = \begin{bmatrix}
    0 & 0 & \dots& \phi_i\rvert_{e_i} & \phi_{i+1}\rvert_{e_i}&   \dots&  0  & 0 & \dots & 0 \\
    0 & 0 &\dots & 0 & 0 & \dots & \phi_{i}\rvert_{e_i} &\phi_{i+1}\rvert_{e_i} & \dots & 0
    \end{bmatrix}.
\end{equation}
We note from~\eqref{local_vem} that $\phi_i$ and $\phi_{i+1}$ are linear functions along the edges so they can be represented exactly via a parametrization of $e_i$. We also note that the product $\acmajor{\Nmatrix}\bm{N}^p$ is at most \acrev{polynomials} of degree $\ell$, so that the terms of the form $(\acmajor{\Nmatrix} \bm{N}^p)^T \bm{N}^v \rvert_{e_i}$ are at most a polynomial of degree $\ell+1$. This suggests that if we parametrize $e_i$ by $t \in [-1,1]$, we can use a one-dimensional
Gauss quadrature rule to compute these integrals.
In particular, let $r_i(t) : [-1,1] \to e_i$ be a parametrization of the $i$-th edge and let 
$\{\omega_1,\cdots,\omega_r\}$, $\{t_1,\cdots,t_r\}$ be the 
associated Gauss quadrature weights and nodes. Then, after simplifications we have 
\begin{align*}
    \int_{e_i}{\left(\acmajor{\Nmatrix}\bm{N}^p\right)^T\bm{N}^v \rvert_{e_i} \, ds} = \frac{|e_i|}{2}\int_{-1}^{1}{\left(\acmajor{\Nmatrix}\bm{N}^p\right)^T\bm{N}^v(r_i(t)) \, dt}
    =\frac{|e_i|}{2}\sum_{j=1}^{r}{\omega_j\left(\acmajor{\Nmatrix}\bm{N}^p\right)^T\bm{N}^v(r_i(t_j))}.
\end{align*}
\acmajor{On} examining the second term in~\eqref{B_matrix}, we note that
$\bm{\partial}$ is a matrix operator of first order derivatives, and $\bm{N}^p$ is a matrix of polynomials of degree less than or equal to $\ell$. This implies that the product $\bm{\partial N^p}$ is a matrix polynomial of degree at most $\ell-1$. Then the product
$(\bm{\partial N^p)}^T\bm{N}^v$ contains terms of the form $\int_{E}{\bm{p}_{\ell-1}\cdot\bm{\varphi}_j}$. On applying the definition of the space~\eqref{local_vem}, we can replace these integrals with the integrals of the elliptic projection, that is
\begin{align*}
    \int_{E}{\bm{p}_{\ell-1}\cdot \bm{\varphi}_j \, d\bm{x}} = \int_{E}{\bm{p}_{\ell-1}\cdot \proj{\bm{\varphi}_j} \, d\bm{x}} ,
\end{align*}
which in matrix form can be written as
\begin{subequations}\label{Bmatrix_2}
\begin{align}
    \label{Bmatrix_2-a}
    \int_{E}{(\bm{\partial} \bm{N}^p)^T\bm{N}^v \, d\bm{x}} &= \int_{E}{(\bm{\partial} \bm{N}^p)^T \proj{\bm{N}^v} \, d\bm{x}} ,\\
\intertext{where we have the natural definition}
    \label{Bmatrix_2-b}
        \proj{\bm{N}^v} &:= 
    \begin{bmatrix}
        \proj{\bm{\varphi}_1} & \proj{\bm{\varphi}_2} & \dots & \proj{\bm{\varphi}_{2N_E}}
    \end{bmatrix}.
\end{align}
\end{subequations}
The integral in~\eqref{Bmatrix_2-a} is computed using a cubature scheme.
With these matrices, we can compute the $L^2$ projection $\overline{\projLtwo{\bm{v}_h}}$ 
using~\eqref{L2_projection_matrices}.

\subsection{Implementation of element stiffness matrix}
To construct the element stiffness, we first rewrite the bilinear form $a^E_h$ in terms of the matrices that we have constructed:
\begin{align*}
    a_h^E(\bm{u}_h,\bm{v}_h) &:= \int_{E}{\left(\overline{\projLtwo{\bm{v}_h}}\right)^T \bm{C}\, \overline{\projLtwo{\bm{u}_h}} \, d\bm{x}} \\ 
    &= \int_{E}{\left(\bm{N}^p\bm{\Pi}^m\tilde{\bm{v}}_h\right)^T\bm{C} \left(\bm{N}^p\bm{\Pi}^m\tilde{\bm{u}}_h\right) \, d\bm{x}} \\ 
    &= \tilde{\bm{v}}_h^T \left(\bm{\Pi}^m\right)^T
    \left(\int_{E}{\left(\bm{N}^p\right)^T}\bm{C}\bm{N}^p \, d\bm{x}\right)\bm{\Pi}^m \tilde{\bm{u}}_h .
\end{align*}
Then, define the element stiffness matrix $\bm{K}_E$ by
\begin{align}
    \bm{K}_E :=\left(\bm{\Pi}^m\right)^T\left(\int_{E}{\left(\bm{N}^p\right)^T}\bm{C}\bm{N}^p \, d\bm{x}\right)\bm{\Pi}^m ,
\end{align}
where $\bm{\Pi}^m$ is given in~\eqref{L2_projection_matrices}.

\subsection{Implementation of element force vector}
We construct the forcing term given in~\eqref{local_force} as 
\begin{align*}
    b^E_h(\bm{v}_h)= \int_{E}{\bm{v}_h^T\bm{f}_h\,d\bm{x}} + \int_{\Gamma_N \cap \partial E}{\bm{v}_h^T\bar{\bm{t}}\,ds},
\end{align*}
which is rewritten in the form
\begin{equation}
    b_{h}^E(\bm{v}_h) =\left(\tilde{\bm{v}}_h\right)^T\left(\int_{E}{(\bm{N}^v)^T \bm{f}_h \, d\bm{x}} + \int \displaylimits_{\partial E \cap \Gamma_N}{(\bm{N}^v)^T \bar{\bm{t}}  \, ds}\right).
\end{equation}
Since we are using a low-order scheme, we use the approximation
\begin{align*}
    \int_{E}{(\bm{N}^v)^T f_h \, d\bm{x}}  &\approx \overline{(\bm{N}^v)^T} \int_{E}{\bm{f}_h \, d\bm{x}} \approx |E|\overline{(\bm{N}^v)^T} \bm{f}(\bm{x}_E),
\end{align*}
where 
$\overline{(\bm{N}^v)^T}$ is the matrix of average values of $\phi$.
Specifically, denoting the $j$-th vertex by $\bm{x}_j$, we define the average value as  
\begin{equation*}
    \overline{\phi} = \frac{1}{N_E}\sum_{j=1}^{N_E}{\phi(\bm{x}_j)},
\end{equation*}
and let
\begin{align*}
    \overline{\bm{N}^v} = 
        \begin{bmatrix}
        \overline{\phi_1} & \overline{\phi_2} & \dots & \overline{\phi_{N_E}} & 0 & 0 & \dots & 0 \\ 
        0 & 0 & \dots & 0 & \overline{\phi_1} & \overline{\phi_2} & \dots & \overline{\phi_{N_E}} 
    \end{bmatrix} 
    =     \begin{bmatrix}
    \frac{1}{N_E} & \frac{1}{N_E} & \dots& \frac{1}{N_E} & 0 & 0& \dots& 0 \\ 
    0 & 0 & \dots & 0 & \frac{1}{N_E} & \frac{1}{N_E} & \dots &\frac{1}{N_E} 
    \end{bmatrix}.
\end{align*}
\acmajor{For constant tractions, we obtain a closed-form solution for the traction integral:}
\begin{align*}
    \left(\int \displaylimits_{\partial E \cap \Gamma_N}{(\bm{N}^v)^T  \,ds}\right)\bar{\bm{t}} =\left( \sum_{e_j \in \partial E} {\int_{e_j}(\bm{N^{v}})^T\rvert_{e_j} \, ds}\right) \bar{\bm{t}}.
\end{align*}
Now applying a similar argument as in (\ref{Nv}), we can simplify 
\acmajor{this} integral as 
\begin{align*}
    \int_{e_j}(\bm{N^{v}})^T\rvert_{e_j} \, ds &= 
    |e_j|
    \begin{bmatrix}
    0 & 0 & \dots & \frac{1}{2} & \frac{1}{2} & \dots & 0 & 0 & \dots & 0 \\ 
    0 & 0 & \dots & 0 & 0 &\dots &\frac{1}{2} & \frac{1}{2} & \dots & 0
    \end{bmatrix}.
\end{align*}

\section{Theoretical Results}\label{sec:theory}
We examine the well-posedness of the discrete problem~\eqref{vem_weak_problem} and
derive error estimates in the $L^2$ norm and energy seminorm. To simplify the analysis we resort to the study of the boundary-value problem with homogeneous Dirichlet boundary data.
We expect the results can be extended to the inhomogeneous case. 

\subsection{Well-posedness of discrete problem}
The approach follows ideas from~\cite{berrone2021lowest}, and we start by showing that the energy seminorm is \acrev{equivalent with the chosen norm for the space} $\bm{V}_{1,\bm{\ell}}$, and use this norm to show that the bilinear form 
in~\eqref{global_bilinearform} satisfies the properties of the Lax-Milgram theorem. We begin by first defining a candidate discrete norm operator:
\begin{definition}
Let $a_h$ be the bilinear form defined in~\eqref{global_bilinearform}, then define an operator $\|.\|_{\bm{\ell}}: \bm{V}_{1,\bm{\ell}} \to \bm{\mathbb{R}}$ by 
\begin{align}
    \acrev{\|\bm{u}\|_{\bm{\ell}} := \left(a_h(\bm{u},\bm{u})\right)^{\frac{1}{2}} 
    = \left(\sum_{E}{a_h^E(\bm{u},\bm{u})}\right)^{\frac{1}{2}}.}\label{l_norm}
\end{align}
\end{definition}
For specific $\bm{\ell}$ values, this operator is a norm and is equivalent to the natural norm in the space $[H^{1}_0(\Omega)]^2$. The main difficulty is showing that the operator is positive definite, i.e., $\|\bm{u}\|_{\bm{\ell}} = 0 \implies \bm{u}=\bm{0}$. To this end, we introduce a theorem given in~\cite{berrone2021lowest}: 

\begin{theorem}\label{projeqzero}
Let $E$ be any element in the space, and $\bm{u} \in \VEM$. Choose $\ell\in \mathbb{N}$ satisfying 
\acmajor{\begin{align*}
    &\frac{3}{2}(\ell+1)(\ell+2) - \text{dim}(\Pker) \geq 2N_E -3,\\
\intertext{or in general choose $\ell\in \mathbb{N}$ satisfying}
    &\qquad \qquad N_E\leq 2\ell+3,
\end{align*}}
then we have 
\begin{align}
    \projLtwo{\bm{u}} = \bm{0} \implies \bm{\varepsilon}(\bm{u}) = \bm{0}.
\end{align}
\end{theorem}

To prove this theorem, we introduce the following lemma:
\begin{lemma}\label{lemma1}
Let $\bm{u} \in \VEM$, with $\ell\geq 1$, then the following implication holds
\begin{align}
    \projLtwo{\bm{u}} =\bm{0} \implies \bm{\varepsilon}(\proj{\bm{u}})=\bm{0} 
\end{align}
\end{lemma}
\begin{proof}
Assume that $\projLtwo{\bm{u}}=\bm{0}$, then by definition of the $L^2$ projection, we have 
\begin{align*}
    (\bm{\varepsilon}(\bm{u}) , \bm{\varepsilon}^p )_E = 0 \quad \forall \bm{\varepsilon}^p \in \polysym.
\end{align*}
In particular, if we let $\bm{p} \in [\mathbb{P}_{1}(E)]^2$ , then $\bm{\sigma}(\bm{p}) \in {\mathbb{P}}_{0}(E)^{2\times 2}_{\text{sym}} \subseteq \polysym$. So we have 
\begin{align*}
    (\bm{\varepsilon}(\bm{u}),\bm{\sigma}(\bm{p}))_E = 0 \quad \forall \bm{p} \in [\mathbb{P}_1(E)]^2.
\end{align*}
Applying the definition of the energy projection $\proj{\bm{u}}$, we get 
\begin{align*}
    (\bm{\varepsilon}(\proj{\bm{u}}),\bm{\sigma}(\bm{p}))_E=0.
\end{align*}
Since this is true for any $\bm{p} \in [\mathbb{P}_{1}(E)]^2$, 
this results in
\begin{align*}
    \bm{\varepsilon}(\proj{\bm{u}}) = \bm{0}.
    \tag*{\qedhere}
\end{align*}
\end{proof}

In order to show that the defined operator is a norm we use an inf-sup type argument. 
To establish the results, we construct some additional spaces and operators.
To motivate the constructions, we assume that the condition $\projLtwo{\bm{u}}=\bm{0}$ holds. This implies that the following equality holds:
\begin{align*}
    \int_{E}{\projLtwo{\bm{u}}: \bm{\varepsilon}^p\,d\bm{x}} = 0 \quad \forall \bm{\varepsilon}^p \in \polysym.
\end{align*}
Applying the definition of the $L^2$ projection in~\eqref{defL2proj}, we also obtain
\begin{align*}
    \int_{E}{\bm{\varepsilon}(\bm{u}) : \bm{\varepsilon}^p\,d\bm{x}} = 0.
\end{align*}
Using the divergence theorem, we can rewrite this equality as 
\begin{align*}
    \int_{E}{\bm{\varepsilon}(\bm{u}):\bm{\varepsilon}^p\,d\bm{x}} = \int_{\partial E}{\bm{u} \cdot \left(\bm{\varepsilon}^p \cdot \bm{n}\right)\,ds} - \int_{E}{\bm{u} \cdot \left(\diverge{\bm{\varepsilon}^p}\right)\,d\bm{x}}=0.
\end{align*}
We note that $\diverge{\bm{\varepsilon}^p} \in [\mathbb{P}_{l-1}]^2 \subseteq [\mathbb{P}_{l+1}]^2$, and using the definition of the space $\VEM$ , 
Lemma~\ref{lemma1} 
and applying the divergence theorem, the second term becomes
\begin{align*}
    \int_{E}{\bm{u} \cdot \left(\diverge{\bm{\varepsilon}^p}\right)\,d\bm{x}} &= \int_{E}{\proj{\bm{u}} \cdot\ \left(\diverge{\bm{\varepsilon}^p}\right)\,d\bm{x}}
    = \int_{\partial E}{\proj{\bm{u}} \cdot \left(\bm{\varepsilon}^p \cdot \bm{n}\right)\,ds} .
\end{align*}
This gives us the equality
\begin{align}
    0 = \int_{E}{\bm{\varepsilon}(\bm{u}):\bm{\varepsilon}^p\,d\bm{x}} &= \int_{\partial E}{\bm{u} \cdot \left(\bm{\varepsilon^p \cdot \bm{n}}\right)\,ds} - \int_{\partial E}{\proj{\bm{u}}\cdot \left(\bm{\varepsilon}^p \cdot 
    \bm{n}\right)\,ds} \nonumber \\ 
    &= \int_{\partial E}{\left(\bm{u}-\proj{\bm{u}}\right)\rvert_{\partial E}\cdot \left(\bm{\varepsilon}^p \cdot \bm{n}\right)\,ds},\label{bilinearequalzero}
\end{align}
 where we use the notation $\left(\bm{u}-\proj{\bm{u}}\right) \rvert_{\partial E}$ to explicitly indicate that the function is evaluated on the boundary. 
This suggests that we study the operator of the form 
$\int_{\partial E}{\bm{v}\cdot (\bm{Q}\cdot \bm{n})\,ds}$. 

\begin{definition}
Define the bilinear operator $b : R_Q(E) \times [\bm{V}]^2 \to \mathbb{R}$ 
by~\cite{berrone2021lowest} 
\begin{align}\label{adjustedbilinear}
    b(\bm{v},\bm{Q}) = \int_{\partial E}{\bm{v}\cdot (\bm{Q}\cdot \bm{n})\,ds},
\end{align}
where $\bm{v}$ is defined over the boundary $\partial E$. The spaces $R_Q(E)$ and $[\bm{V}]^2$ are chosen later. 
\end{definition}

\acmajor{In particular, we study the special case when
$\bm{v} = \left(\bm{u}-\proj{\bm{u}}\right) \rvert_{\partial E}$. Since we are interested in all such functions $\bm{u} \in \VEM$, we study the space of all linear combination of the basis functions $\left(\bm{\varphi}_i-\proj{\bm{\varphi}_i}\right) \rvert_{\partial E} $. This motivates the next definition: 
\begin{definition}
Define the space $Q(\partial E)$ by
\begin{align}
    Q(\partial E) := \text{span}{\{\left(\bm{\varphi}_i-\proj{\bm{\varphi}_i}\right) \rvert_{\partial E} : i=1,2\dots , 2N_E\}} .
\end{align}
\end{definition}}
Now given a function on $Q(\partial E)$, we need to extend it to a function defined on the entire element $E$. One way to achieve this is to first triangulate the polygon $E$. Let $\tau \subseteq E$ be any triangular subelement. 
Denote $\tau_i$ as the triangle with vertices $\bm{x}_i,\bm{x}_{i+1},\bm{x_c}$, for each $i=1,2,\dots, N_E$, where $\bm{x}_c$ is the centroid of $E$.
We denote the edge connecting the vertices $\bm{x}_i$ and $\bm{x}_c$ by $e_i$, and the 
unit outward normal as $\bm{n}^{e_i}$.    
 With this triangulation, we extend $\bm{v}$ to be a function $\doubleline{\bm{v}}$ on $E$ by requiring that $\doubleline{\bm{v}}$ agrees with $\bm{v} \rvert_{e}$ over each edge $e$ and $\doubleline{\bm{v}} \rvert_{\tau} \in [\mathbb{P}_{1}(\tau)]^2$ over every triangular element $\tau$. To obtain a unique vector-valued function, we require that $\doubleline{\bm{v}} (\bm{x}_c) =\bm{0}$. We use this to define the space $R_Q(E)$ of extended functions over the entire element $E$. 

\begin{definition}
Define the space $R_{Q}(E)$ by 
\begin{align}
    R_{Q}(E) := \{ \doubleline{\bm{v}} : \doubleline{\bm{v}} \rvert_{\tau} \in [\mathbb{P}_1(\tau)]^2 \ \ \forall \tau \subseteq E, \ \doubleline{\bm{v}}\rvert_{\partial E} \in Q(\partial E), \ \doubleline{\bm{v}}(\bm{x_c})=\bm{0} \}. 
\end{align}
\end{definition}

Using~\eqref{adjustedbilinear}, we express~\eqref{bilinearequalzero} as
\begin{align}\label{bilinearequalzero2}
    b(\bm{u}-\proj{\bm{u}},\bm{\varepsilon}^p)=0.
\end{align}
But the extended function $\doubleline{\bm{u}-\proj{\bm{u}}}$ is equal to ${\bm{u}-\proj{\bm{u}}}$ over the boundary, so applying the expression to the extended function, we get 
\begin{align*}
    b(\doubleline{\bm{u}-\proj{\bm{u}}},\bm{\varepsilon}^p)=0 \quad \forall \bm{\varepsilon}^p \in \polysym.
\end{align*}
To show that $\bm{\varepsilon}(\bm{u})=\bm{0}$, it is sufficient to establish that $\bm{u} =\proj{\bm{u}}$ is a rigid-body mode. This is equivalent to showing that
\begin{equation*}
    \|\bm{u}-\proj{\bm{u}}\| =0 
\end{equation*}
in some norm. From~\cite{berrone2021lowest}, it is sufficient to show an inf-sup condition:
\begin{align}\label{test_inf_sup}
    \sup_{\acrev{\bm{\varepsilon}^p \in \polysym}}{\frac{b(\bm{u}-\proj{\bm{u}},\bm{\varepsilon}^p)}{\|\bm{\varepsilon}^p\|}} \geq \beta \|\bm{u}-\proj{\bm{u}}\| .
\end{align}
To formalize this, we first construct a suitable space with a suitable norm. 

\begin{definition}
For every element $E$, let $\bm{H}^1_{\tau}(E)$ be the broken Sobolev space that is defined by 

\begin{align}
    \bm{H}^1_{\tau}(E) := \bigcup_{\tau}{\bm{H}^{1}(\tau)},
\end{align}
where $\bm{H}^1(\tau) = [H^1(\tau)]^2$ is the standard Sobolev space defined on a triangular subelement. 
On this space, equip the seminorm and norm:
\begin{subequations}
\begin{align}
|\bm{u}|^2_{\bm{H}^1_{\tau}(E)} &:= \sum_{\tau}{\|\grad{\bm{u}}\|^2_{\bm{L}^2(\tau)}} + \sum_{i=1}^{N_E}{\|[\![ \bm{u}]\!]_{e_i}\|^2_{\bm{L}^2(e_i)}}, \\
\|\bm{u}\|^2_{\bm{H}^1_{\tau}(E)} &:=|\bm{u}|^2_{\bm{H}^1_{\tau}(E)} + \sum_{\tau}{\|\bm{u}\|^2_{\bm{L}^2(\tau)}}.
\end{align}
\end{subequations}
Again, let $\gamma^{e_i}(.)$ be the trace  
of its argument on edge $e_i$. We then define  $[\![.]\!]_{e_i} : \bm{H}^1_{\tau} \to \bm{L}^2(e_i) $ as the jump across the $i$-th edge of the triangulation, which is 
given by
\begin{align*}
    [\![\bm{u}]\!]_{e_i} := \gamma^{e_i}(\bm{u}\rvert_{\tau_i}) - \gamma^{e_i}(\bm{u}\rvert_{\tau_{i-1}}).
\end{align*}
\end{definition}
We now define a space of functions with finite jumps across edges in the triangulation.
\begin{definition}
Define the space $\bm{V} = \bm{V}(E) \subseteq \acrev{\bigcup_{\tau}{\bm{H}(\textrm{div},\tau)}}$ by
\begin{align}
    \bm{V}(E) := \acrev{\left\{ \bm{v} \in \bigcup_{\tau}{\bm{H}(\textrm{div},\tau)} :  \| [\![\bm{v} ]\!]_{e_i} \|_{\bm{L}^{\infty}(e_i)} < \infty \ \ \forall e_i  \right\}},
\end{align}
\acrev{where ${\bm{H}(\textrm{div},\tau)} $ is the space of functions that have finite divergence in the $L^2$ norm over a triangular subelement.} 
On this space, define the seminorm and norm as
\begin{subequations}
\begin{align}
    |\bm{v}|^2_{\bm{V}} &:= \sum_{\tau}{\|\acrev{\diverge{\bm{v}}}\|^2_{\bm{L}^2(\tau)}} + \acrev{h_E^2}\|[\![\bm{v} ]\!]_{I_{E}}\|^2_{\bm{L}^{\infty}(I_E)},\\ 
    \|\bm{v}\|^2_{\bm{V}} &:= |\bm{v}|^2_{\bm{V}} + \sum_{\tau}{\|\bm{v}\|^2_{\bm{L}^2(\tau)}}, \\
\intertext{where}
    \|[\![\bm{v} ]\!]_{I_{E}}\|^2_{\bm{L}^{\infty}(I_E)} &= \max_{i}{\|[\![\bm{v} ]\!]_{e_i}\|^2_{\bm{L}^{\infty}(e_i)}} 
\end{align}
\end{subequations}
is the maximum of the jumps over all edges in the triangulation. 
\end{definition}

Now we show that the bilinear operator defined in~\eqref{adjustedbilinear} 
is continuous on the newly defined spaces $R_Q(E) \times [\bm{V}]^2$. 
\begin{lemma}
\acrev{Let $b$ be the bilinear form defined in~\eqref{adjustedbilinear}, then there exists a constant $C>0$, such that} 
\begin{align}
    |b(\bm{v},\bm{Q})| \leq C \|\bm{v}\|_{\bm{H}^1_{\tau}(E)} \| \bm{Q}\|_{\bm{[V]}^2} \  \forall \bm{v} \in R_{Q}(E) \ \textrm{and} \ \forall \bm{Q} \in [\bm{V}]^2.
\end{align}
\end{lemma}
\begin{proof}
By definition, we have 
\begin{align*}
    b(\bm{v},\bm{Q}) = \int_{\partial E}{\bm{v}\cdot (\bm{Q}\cdot \bm{n})\,ds}.
\end{align*}
We partition each element $E$ into a union of triangles $\{\tau_i\}$, and again letting $\{e_i\}$ denote the edge connecting the $i$-th vertex to the center, we rewrite the integral as  
\begin{align*}
    b(\bm{v},\bm{Q}) = \sum_{i}&\left[\int_{\partial \tau_i}{\bm{v}\cdot (\bm{Q}\cdot \bm{n})\,ds}-\int_{e_i}{ \gamma^{e_i}(\bm{v}\rvert_{\tau_i})\cdot (\bm{Q}^{e_i}_{\tau_i}\cdot \bm{n}_{e_i}^{\tau_i})\,ds}\right. \\ & \ \ - \left. \int_{e_{i}}{ \gamma^{e_i}(\bm{v}\rvert_{\tau_{i-1}})\cdot (\bm{Q}^{e_i}_{\tau_{i-1}}\cdot \bm{n}_{e_i}^{\tau_{i-1}})\,ds}\right].
\end{align*}
We first note that by assumption $\bm{v}\in R_Q(E)$, which implies that $\bm{v}$ along the $i$-th edge is the same from either triangle.
So we now have
\begin{align*}
    \ \gamma^{e_i}(\bm{v}\rvert_{\tau_i}) = \gamma^{e_i}(\bm{v}\rvert_{\tau_{i-1}}).
\end{align*}
In addition, since 
$\bm{n}_{e_i}^{\tau_i} = -\bm{n}_{e_i}^{\tau_{i-1}}$, we 
can rewrite $b(\bm{v},\bm{Q})$ as 
\begin{align}
    b(\bm{v},\bm{Q}) &=  \sum_{i}\int_{\partial \tau_i}{\bm{v}\cdot (\bm{Q}\cdot \bm{n})\,ds}-\int_{e_i}{ \gamma^{e_i}(\bm{v}\rvert_{\tau_i})\cdot (\bm{Q}^{e_i}_{\tau_i}-\bm{Q}^{e_i}_{\tau_{i-1}})\cdot \bm{n}_{e_i}^{\tau_i}\,ds} \nonumber \\
    &=\sum_{i}{\int_{\partial \tau_i}{\bm{v}\cdot (\bm{Q}\cdot \bm{n})\,ds}} - \int_{e_i}{\gamma^{e_i}(\bm{v}\rvert_{\tau_i})\cdot ([\![\bm{Q}]\!]_{e_i} \cdot \bm{n}_{e_i}^{\tau_i})\,ds}.\label{bsimplify}
\end{align}
For the first term in~\eqref{bsimplify}, we apply the 
divergence theorem to obtain
\begin{align*}
    \int_{\partial \tau_i}{\bm{v} \cdot (\bm{Q} \cdot \bm{n})\, ds} = \int_{\tau_i}{\nabla\cdot(\bm{v}\cdot \bm{Q}) \,d\bm{x}} = \int_{\tau_i} [ 
    {\grad{\bm{v}}:\bm{Q} + \bm{v}\cdot (\diverge{\bm{Q}}) ] \,d\bm{x}}.
\end{align*}
We now have 
\begin{align*}
    b(\bm{v},\bm{Q}) = \sum_{i}{\int_{\tau_i}{[ \grad{\bm{v}}:\bm{Q} + \bm{v}\cdot (\diverge{\bm{Q}}) ]\,d\bm{x}}-\int_{e_i}{\gamma^{e_i}(\bm{v}\rvert_{\tau_i})\cdot ([\![\bm{Q}]\!]_{e_i} \cdot \bm{n}_{e_i}^{\tau_i})\,ds}} ,
\end{align*}
and can bound $|b(\bm{v},\bm{Q}|$ in~\eqref{bsimplify} as 
\begin{align}
    |b(\bm{v},\bm{Q})| \leq \underbrace{|\sum_{i}{\int_{\tau_i}{[ \grad{\bm{v}}:\bm{Q} + \bm{v}\cdot (\diverge{\bm{Q}}) ]\,d\bm{x}}}}_{A} | + \underbrace{|\sum_{i}{\int_{e_i}{\gamma^{e_i}(\bm{v}\rvert_{\tau_i})\cdot ([\![\bm{Q}]\!]_{e_i} \cdot \bm{n}_{e_i}^{\tau_i})\,ds}} |}_{B} . \label{b_estimate}
\end{align}
We estimate each term in~\eqref{b_estimate} separately. For term $A$ in~\eqref{b_estimate}, we have 
\begin{align}
    &|\sum_{i}{\int_{\tau_i}{[ \grad{\bm{v}}:\bm{Q} + \bm{v}\cdot (\diverge{\bm{Q}}) ] \,d\bm{x}}}| \leq \sum_{i}\Bigr[ \|\grad{\bm{v}}\|_{\bm{L}^2(\tau_i)}\|\bm{Q}\|_{\bm{L}^2(\tau_i)} + \|\bm{v}\|_{\bm{L}^2(\tau_i)}\|\diverge{\bm{Q}}\|_{\bm{L}^2(\tau_i)} \Bigr] \notag \\
    &\leq \sum_{i}{\Bigr[ \|\grad{\bm{v}}\|_{\bm{L}^2(\tau_i)}\Bigl(\|\bm{Q}\|_{\bm{L}^2(\tau_i)}+\|\acrev{\diverge{\bm{Q}}}\|_{\bm{L}^2(\tau_i)}\Bigr)} + \|\bm{v}\|_{\bm{L}^2(\tau_i)}\Bigl(\|\bm{Q}\|_{\bm{L}^2(\tau_i)}+\|\acrev{\diverge{\bm{Q}}}\|_{\bm{L}^2(\tau_i)}\Bigr) \Bigr] \notag \\
    &\leq C\sum_{i}{\Bigl(\|\bm{v}\|_{\bm{L}^2(\tau_i)}+\|\grad{\bm{v}}\|_{\bm{L}^2(\tau_i)}\Bigr)\Bigl(\|\bm{Q}\|_{\bm{L}^2(\tau_i)}+\|\acrev{\diverge{\bm{Q}}}\|_{\bm{L}^2(\tau_i)}\Bigr)} \notag \\
    &\leq C\|\bm{v}\|_{\bm{H}^1_{\tau}}\sum_{i}{\Bigl(\|\bm{Q}\|_{\bm{L}^2(\tau_i)}+\|\acrev{\diverge{\bm{Q}}}\|_{\bm{L}^2(\tau_i)}\Bigr)} \label{estimate_div}.
\end{align}
Now for term $B$ in~\eqref{b_estimate}, we estimate
\begin{align*}
\sum_{i}{|\int_{e_i}{\gamma^{e_i}(\bm{v}\rvert_{\tau_i})\cdot ([\![\bm{Q}]\!]_{e_i} \cdot \bm{n}_{e_i}^{\tau_i})\,ds} |} \leq \sum_{i}{\|\gamma^{e_i}(\bm{v}\rvert_{\tau_i})\|_{\bm{L}^2(e_i)} \|[\![\bm{Q}]\!]_{e_i}\|_{\bm{L}^2(e_i)}}.
\end{align*}
Since $\bm{v} \in R_Q(E)$, it is linear on each of the edges $e_i$. It can be shown using a three-point Gauss-Lobatto quadrature scheme and equivalent norms, that 
\begin{align*}
    \|\gamma^{e_i}(\bm{v}\rvert_{\tau_i})\|_{\bm{L}^2(e_i)} =\sqrt{\frac{|e_i|}{3}}|\bm{v}(\bm{x}_i)|.
\end{align*}
We can also estimate that 
\begin{align*}
    \|[\![\bm{Q}]\!]_{e_i}\|_{\bm{L}^2(e_i)}&\leq \sqrt{|e_i|}  \|[\![\bm{Q}]\!]_{e_i}\|_{\bm{L}^{\infty}(e_i)}\\
    &\leq \sqrt{|e_i|} \|[\![\bm{Q} ]\!]_{I_{E}}\|_{\bm{L}^{\infty}(I_E)}.
\end{align*}
Combining the two terms and using equivalent norms, we get 
\begin{align}
\sum_{i}{|\int_{e_i}{\gamma^{e_i}(\bm{v}\rvert_{\tau_i})\cdot ([\![\bm{Q}]\!]_{e_i} \cdot \bm{n}_{e_i}^{\tau_i})\,ds} |} &\leq \sum_{i}{{\frac{|e_i|}{\sqrt{3}}}|\bm{v}(\bm{x}_i)| \|[\![\bm{Q} ]\!]_{I_{E}}\|_{\bm{L}^{\infty}(I_E)}} \notag \\
&\leq Ch_E\|[\![\bm{Q} ]\!]_{I_{E}}\|_{\bm{L}^{\infty}(I_E)} \|\bm{v}\|_{\bm{H}^1_{\tau}(E)}.\label{estimate_jump} 
\end{align}
Combining these two terms in~\eqref{estimate_div} and~\eqref{estimate_jump}, we find that
\begin{align*}
    |b(\bm{v},\bm{Q})|&\leq C_1\|\bm{v}\|_{\bm{H}^1_{\tau}}\sum_{i}{(\|\bm{Q}\|_{\bm{L}^2(\tau_i)}+\|\acrev{\diverge{\bm{Q}}}\|_{\bm{L}^2(\tau_i)})}  +C_2h_E\|\bm{v}\|_{\bm{H}^1_{\tau}(E)}\|[\![\bm{Q} ]\!]_{I_{E}}\|_{\bm{L}^{\infty}(I_E)} \\
    &\leq C\|\bm{v}\|_{\bm{H}^1_{\tau}(E)}\|\bm{Q}\|_{[\bm{V}]^2}.
    \tag*{\qedhere}
\end{align*}
\end{proof}

Using this bilinear form $b$ and the specific norms, we
formalize the inf-sup condition that is stated in~\eqref{test_inf_sup}. 
\begin{proposition}
Let $\bm{u} \in \VEM$ and $b$ as defined in~\eqref{adjustedbilinear}. If there exists a constant $\beta>0$, independent of $h_E$, such that
\begin{align}\label{inf-sup}
    \forall \bm{v} \in R_Q(E), \quad \sup_{\bm{Q} \in \polysym}\frac{b(\bm{v},\bm{Q})}{\|\bm{Q}\|_{[\bm{V}]^2}} \geq \beta \|\bm{v}\|_{\bm{H}^1_{\tau}(E)},
\end{align}
then 
\begin{align*}
    \projLtwo{\bm{u}}= \bm{0} \implies \bm{\varepsilon}(\bm{u}) =\bm{0}.
\end{align*}
\end{proposition}
\begin{proof}
Assume that $\projLtwo{\bm{u}}=0$, then by~\eqref{bilinearequalzero2}, we have
\begin{align*}
b(\doubleline{\bm{u}-\proj{\bm{u}}},\bm{\varepsilon}^p) = 0.    
\end{align*}
Then by assumption 
\begin{align*}
    \beta\|\doubleline{\bm{u}-\proj{\bm{u}}}\|_{\bm{H}^1_{\tau}(E)} = 0,
\end{align*}
which implies that 
\begin{align*}
    \doubleline{\bm{u}-\proj{\bm{u}}} = \bm{0}.
\end{align*}
Then we also have on the boundary, 
\begin{align*}
\bm{u}-\proj{\bm{u}}\rvert_{\partial E} = \bm{0}.
\end{align*}
But for $\bm{u} \in \VEM$, this implies that $\bm{u}=\proj{\bm{u}}$. Then by Lemma~\eqref{lemma1}, we get $\bm{\varepsilon}(\bm{u}) = \bm{0}$.
\end{proof}

In order for the previous proposition to hold for any constant, we include a stronger result as proven in~\cite{berrone2021lowest} for scalar equations. The proof of these results \acrev{relies} on the construction of a Fortin operator $\bm{\Pi}_E$, as shown for general cases in~\cite{mixed:fem}.
\begin{proposition}\label{prop2}
Assume there exists an operator $\bm{\Pi}_E :[\bm{V}]^2 \to [\mathbb{P}_l(E)]^{2 \times 2}$ satisfying~\cite{mixed:fem}
\begin{align}
    b(\bm{v},\bm{\Pi}_E\bm{Q}-\bm{Q})=0 \quad \forall \bm{v} \in R_Q(E)
\end{align}
and assume there is some constant $C_{\Pi}>0$, independent of $h_E$, such that
\begin{align}
\|\bm{\Pi}_E\bm{Q} \|_{\bm{[V]}^2} \leq C_{\Pi}\|\bm{Q}\|_{\bm[V]^2} \ \ \forall \bm{Q} \in [\bm{V}]^2.  
\end{align}
Assume further that there exists a $\eta >0$, independent of $h_E$, such that 
\begin{align}\label{infsup}
    \inf_{\bm{v}\in R_Q(E)} \sup_{\bm{Q}\in [\bm{V}]^2}\frac{b(\bm{v},\bm{Q})}{\|\bm{v}\|_{\bm{H}^1_{\tau}(E)} \|\bm{Q}\|_{\bm{[V]}^2}} \geq \eta.
\end{align}
Then the discrete inf-sup condition given in~\eqref{inf-sup} is satisfied. 
\end{proposition}


\begin{proposition}
Let b be defined by~\eqref{adjustedbilinear}, \acrev{then the inf-sup condition given in~\eqref{infsup} holds}. 
\end{proposition}

For the proof of these propositions we refer the reader to Propositions 2 and 3 in~\cite{berrone2021lowest}, and for the explicit construction of the operator $\acrev{\bm{\Pi}_E}$, we also point to Proposition \acrev{4} in~\cite{berrone2021lowest}. The construction methods appear to generalize directly to the vectorial case. 
We now show that the operator given in~\eqref{l_norm} satisfies the positive-definite property and is thus a norm. 
\begin{proposition}
For any $\bm{u} \in \bm{V}_{1,\bm{\ell}}$, with $\ell(E) \in \mathbb{N}$ satisfying~\eqref{l_condition} for all elements E,
\begin{align}
    \|\bm{u}\|_{\bm{\ell}} = 0 \implies \bm{u}=\bm{0}, 
\end{align}
where the norm $\|.\|_{\bm{\ell}}$ is defined in~\eqref{l_norm}.
\end{proposition}
\begin{proof}
Let $\bm{u} \in \bm{V}_{1,\bm{\ell}}$ and assume $\|\bm{u}\|^2_{\bm{\ell}}=0$. This implies that
\begin{align*}
    \sum_{E}{\int_{E}{\projLtwo{\bm{u}}:\mathbb{C}:\projLtwo{\bm{u}}}\,d\bm{x}}=0.
\end{align*}
Assuming $\mathbb{C}$ is a positive-definite material tensor, we must have
\begin{align*}
    \projLtwo{\bm{u}}=\bm{0}.
\end{align*}
Since $\ell$ satisfies~\eqref{l_condition}, we know by Theorem \ref{projeqzero} that $\bm{\varepsilon}(\bm{u})=\bm{0}$ for each $E$. This implies that $\bm{u}$ is a rigid-body mode. But due to homogeneous boundary conditions, no nonzero rigid-body modes are present, and therefore $\bm{u}=\bm{0}$. 
\end{proof}

We also have that under the condition~\eqref{l_condition},
that the norm~\eqref{l_norm} is equivalent to the standard norm in $\bm{H}^1_0$. 
\begin{lemma}\label{lemma3}
For all $\bm{u} \in \bm{V}_{1,\bm{\ell}}$, there exists a $C_1>0$ such that 
\begin{subequations}
\begin{align}
    \|\bm{u}\|_{\bm{\ell}} \leq C_1\|\bm{u}\|_{\bm{H}^1_0(\Omega)},\\
\intertext{and if 
for every element $E$, $\ell(E)$ satisfies~\eqref{l_condition}, there also exists a  constant $C_2>0$ such that}
    \|\bm{u}\|_{\bm{\ell}} \geq C_2\|\bm{u}\|_{\bm{H}^1_0(\Omega)}.
\end{align}
\end{subequations}
\end{lemma}
\begin{proof}
We first estimate
\begin{align*}
    \|\bm{u}\|^2_{\bm{\ell}} &= \sum_{E} \int_{E}{\projLtwo{\bm{u}}:\mathbb{C}:\projLtwo{\bm{u}}\,d\bm{x}} \\
    &= \sum_{E}\int_{E}{\bm{\varepsilon}(\bm{u}):\mathbb{C}:\projLtwo{\bm{u}}\,d\bm{x}} \\ 
    &\leq \sum_{E}\|\bm{\varepsilon}(\bm{u})\|_{\bm{L}^2(E)}\|\mathbb{C}:\projLtwo{\bm{u}}\|_{\bm{L}^2(E)} \\ 
    &\leq C\|\bm{\varepsilon}(\bm{u})\|_{\bm{L}^2(\Omega)}\|\bm{\varepsilon}(\bm{u})\|_{\bm{L}^2(\Omega)} \\ 
    &\leq C_1 \|\bm{u}\|_{\bm{H}^1_0(\Omega)} \|\bm{u}\|_{\bm{H}^1_0(\Omega)} \\
    &\leq C_1\|\bm{u}\|^2_{\bm{H}^1_0(\Omega)}. 
\end{align*}
Now if we have $\ell$ that satisfies~\eqref{l_condition} for all $E$ , then $\|.\|_{\bm{\ell}}$ is also a norm. Since both $\|.\|_{\bm{\ell}}$ and $\|.\|_{\bm{H}^1_0}$ are norms in the finite-dimensional subspace $\bm{V}_{1,\bm{\ell}}$, they are equivalent. In particular, there exists a constant $C_2>0$ such that 
\begin{align*}
    \|\bm{u}\|_{\bm{\ell}} \geq C_2 \|\bm{u}\|_{\bm{H}^1_0(\Omega)}. 
    \tag*{\qedhere}
\end{align*} 
\end{proof}

We now show that the discrete bilinear form $a_h$ is continuous and coercive, which by Lax-Milgram theorem implies that a unique solution exists. 
\begin{theorem}
If $\ell(E)$ satisfies~\eqref{l_condition} for each $E$, then there \acrev{exist} constants $C_1, C_2>0$ such that the bilinear form defined in~\eqref{global_bilinearform} satisfies the inequalities
\begin{subequations}
\begin{align}
    |a_h(\bm{u},\bm{v})| &\leq C_1\|\bm{u}\|_{\bm{H}^1_0(\Omega)}\|\bm{v}\|_{\bm{H}^1_0(\Omega)} \\
\intertext{and} 
    a_h(\bm{v},\bm{v}) &\geq C_2\|\bm{v}\|^2_{\bm{H}^1_0(\Omega)}.
\end{align}
\end{subequations}
\end{theorem}
\begin{proof}
We estimate  the first inequality:
\begin{align*}
|a_h(\bm{u},\bm{v})| &=\sum_{E} \int_{E}{\projLtwo{\bm{u}}:\mathbb{C}:\projLtwo{\bm{v}}\,d\bm{x}} \\
&\leq C\sum_{E}{\|\projLtwo{\bm{u}}\|_{\bm{L}^2(E)} \|\projLtwo{\bm{v}}\|_{\bm{L}^2(E)}} \\
&\leq C \|\bm{u}\|_{\bm{H}^1_0(\Omega)} \|\bm{v}\|_{\bm{H}^1_0(\Omega)}.  
\end{align*}
For the second inequality, on using the
definition of the bilinear form $a_h$ and Lemma~\ref{lemma3}, 
we have 
\begin{align*}
    a_h(\bm{v},\bm{v}) = \|\bm{v}\|^2_{\bm{\ell}} \geq C\|\bm{v}\|^2_{\bm{H}^1_0(\Omega)}. \tag*{\qedhere}
\end{align*}
\end{proof}

\subsection{Error estimates}
Now that we have well-posedness of the discrete problem, we study the errors of the approximation. In particular, we consider the errors in the $\bm{L}^2$ and $\bm{H}_0^1$ norms.
Many of the techniques and estimates 
are detailed in~\cite{brenner,estimatesVEM,mixed:fem,ciarlet}.
We introduce lemmas adapted from~\cite{berrone2021lowest} that we expect can be extended to our specific case. We first define an interpolation function $\bm{u}_I:\bm{H}^2(\Omega) \to \bm{V}_{1,\bm{\ell}}$ by 
\begin{align}
    \bm{u}_I= \sum_{i}{\texttt{dof}_{i}(\bm{u})\bm{\xi}_{i}}, \label{interpolation_funct}
\end{align}
where $\texttt{dof}_{i}(\bm{u})$ is the $i$-th degree of freedom of $\bm{u}$ and $\bm{\xi}_i$ is a global basis function satisfying $\texttt{dof}_{j}(\bm{\xi}_i) = \delta_{ij}$.
\begin{lemma}
Let $\bm{w}$ be any sufficiently smooth function, 
and let $\bm{w}_I \in \bm{V}_{1,\bm{\ell}} $ be the associated interpolation function~\eqref{interpolation_funct}. Then the following inequality holds for some constant $C>0$ and all $h>0$:
\begin{align}\label{L2plusH1estimate}
    \|\bm{w}-\bm{w}_I\|_{\bm{L}^2(\Omega)} + h\|\bm{w}-\bm{w}_I\|_{\bm{H}^1_0(\Omega)} \leq Ch^2|\bm{w}|_{\bm{H}^2(\Omega)}.
\end{align}
\end{lemma}

\begin{lemma}
For any sufficiently smooth function $\bm{w}$, there \acrev{exist} constants $C_1$, $C_2 >0$ such that 
\begin{subequations}
\begin{align}
    \|\bm{\Pi}_{\ell}^0{\bm{\varepsilon}(\bm{w})}-\bm{\varepsilon}(\bm{w})\|_{\bm{L}^2(\Omega)} &\leq C_1 h |\bm{w}|_{\bm{H}^2(\Omega)}, \label{strainestimate} \\ 
    \|\bm{\Pi}_{0}^0\bm{w}-\bm{w} \|_{\bm{L}^2(\Omega)} &\leq C_2 h \|\bm{w}\|_{\bm{H}^1_0(\Omega)} , \label{proj0estimate}
\end{align}
\end{subequations}
where we denote 
$\bm{\Pi}_{0}^0\bm{w}$ as the $L^2$ projection of $\bm{w}$ onto the space of constants. 
\end{lemma}

Now we consider the error in $\bm{H}^1_0$.
\begin{proposition}
Let $\bm{u}$ be the exact solution to the strong problem in~\eqref{strongproblem}, and $\bm{f}$ the associated body force. For $h$ sufficiently small, there exists a constant $C>0$ such that the error of the solution $\bm{u}_h$ to the discrete weak problem is bounded in the $\bm{H}^{1}_0$ norm by
\begin{align}\label{h0estimate}
    \|\bm{u}-\bm{u}_h\|_{\bm{H}^1_0(\Omega)} \leq Ch\left(|\bm{u}|_{\bm{H}^2(\Omega)} + \|\bm{f}\|_{\bm{L}^2(\Omega)}\right).
\end{align}
\end{proposition}
\begin{proof}
Let $\bm{u}_h$ be the unique solution to the discrete problem~\eqref{vem_weak_problem}, $\bm{u}$ the exact solution to~\eqref{strongproblem} and $\bm{u}_I$ the associated interpolation function~\eqref{interpolation_funct}. 
We can then estimate the error as: 
\begin{align}
    \|\bm{u}-\bm{u}_h\|_{\bm{H}^1_0(\Omega)} \leq \|\bm{u}-\bm{u}_I \|_{\Hone (\Omega)} + \|\bm{u}_I-\bm{u}_h \|_{\Hone (\Omega)}. 
\end{align}
For the first term, we apply~\eqref{L2plusH1estimate} to get the bound 
\begin{align}
    \|\bm{u}-\bm{u}_I\|_{\Hone (\Omega)} \leq Ch|\bm{u}|_{\bm{H}^2 (\Omega)}.
\end{align}
For the second term, we have the estimate
\begin{align}
    C\|\bm{u}_I-\bm{u}_h\|^2_{\Hone (\Omega)} &\leq \|\bm{u}_I-\bm{u}_h\|_{\ell}^2 = a_h(\bm{u}_I-\bm{u}_h,\bm{u}_I-\bm{u}_h) \nonumber \\
    &\leq-a_h(\bm{u}_h,\bm{u}_I-\bm{u}_h) + a_h(\bm{u}_I,\bm{u}_I-\bm{u}_h) \nonumber \\ 
    &\leq -(\bm{f}_h,\bm{u}_I-\bm{u}_h) + a_h(\bm{u}_I,\bm{u}_I-\bm{u}_h) \nonumber \\
    &\leq -(\bm{f}_h,\bm{u}_I-\bm{u}_h) + a_h(\bm{u}_I-\bm{u}+\bm{u},\bm{u}_I-\bm{u}_h) \nonumber \\ 
    &\leq \underbrace{(-\bm{f}_h,\bm{u}_I-\bm{u}_h)}_{A} + \underbrace{a_h(\bm{u}_I-\bm{u},\bm{u}_I-\bm{u}_h)}_{B}+\underbrace{a_h(\bm{u},\bm{u}_I-\bm{u}_h)}_{C}.\label{estimate_u_I}
\end{align}
We estimate each of the three terms. For term $B$ in~\eqref{estimate_u_I}, we use Cauchy-Schwarz and~\eqref{L2plusH1estimate} to estimate
\begin{align*}
a_h(\bm{u}_I-\bm{u},\bm{u}_I-\bm{u}_h) &= \sum_{E}{\int_{E}{\projLtwo{\bm{u}_I-\bm{u}}}:\mathbb{C}:\projLtwo{\bm{u}_I-\bm{u}_h}\,d\bm{x}}\\
&\leq C\sum_{E}{\|\projLtwo{\bm{u}_I-\bm{u}}\|_{\bm{L}^2(E)} \|\projLtwo{\bm{u}_I-\bm{u}_h}\|_{\bm{L}^2(E)}} \\ 
&\leq C \|\bm{\Pi}^{0}_{\ell}\bm{\varepsilon}{(\bm{u}_I-\bm{u})}\|_{\bm{L}^2(\Omega)} \|\bm{\Pi}^0_{\ell}\bm{\varepsilon}({\bm{u}_I-\bm{u}_h})\|_{\bm{L}^2(\Omega)} \\
&\leq C \|\bm{u}_I-\bm{u}\|_{\Hone (\Omega)} \|\bm{u}_I-\bm{u}_h\|_{\Hone (\Omega)} \\
&\leq Ch|\bm{u}|_{\bm{H}^2 (\Omega)}\|\bm{u}_I-\bm{u}_h\|_{\Hone (\Omega)}. 
\end{align*}
For term $C$ in~\eqref{estimate_u_I}, we write 
\begin{align*}
    a_h(\bm{u},\bm{u}_I-\bm{u}_h) &= \sum_{E}{\int_{E}{\projLtwo{\bm{u}} : \mathbb{C}: \projLtwo{\bm{u}_I-\bm{u}_h}}\,d\bm{x}} \\ 
    &= \sum_{E}{\int_{E}{\projLtwo{\bm{u}}:\mathbb{C}:\bm{\varepsilon}(\bm{u}_I-\bm{u}_h)}\,d\bm{x}} \\ 
    &= \sum_{E}{\int_{E}{[ (\projLtwo{\bm{u}}-\bm{\varepsilon}(\bm{u})+\bm{\varepsilon}(\bm{u})):\mathbb{C}:\bm{\varepsilon}(\bm{u}_I-\bm{u}_h)} ]\,d\bm{x}} \\
    &= \sum_{E}{\int_{E}{(\projLtwo{\bm{u}}-\bm{\varepsilon}(\bm{u})) : \mathbb{C}:\bm{\varepsilon}(\bm{u}_I-\bm{u}_h)}\,d\bm{x}} \\ &\quad \ + \sum_{E}{\int_{E}{\bm{\varepsilon}(\bm{u}):\mathbb{C}:\bm{\varepsilon}(\bm{u}_I-\bm{u}_h)}\,d\bm{x}}.
    \end{align*}
    Then applying the definition of the bilinear form (\ref{weakbilinearform}) and using Cauchy-Schwarz inequality, we write
\begin{align*}
    a_h(\bm{u},\bm{u}_I-\bm{u}_h) &= \sum_{E}{\Bigr [ \int_{E}{(\projLtwo{\bm{u}}-\bm{\varepsilon}(\bm{u})) : \mathbb{C}:\bm{\varepsilon}(\bm{u}_I-\bm{u}_h)}}\,d\bm{x} \Bigr ] + a(\bm{u},\bm{u}_I-\bm{u}_h) \\ 
    &= \sum_{E}{\Bigr [\int_{E}{(\projLtwo{\bm{u}}-\bm{\varepsilon}(\bm{u})) : \mathbb{C}:\bm{\varepsilon}(\bm{u}_I-\bm{u}_h)}}\,d\bm{x}\Bigr ] + (\bm{f},\bm{u}_I-\bm{u}_h) \\ 
    &\leq Ch|\bm{u}|_{\bm{H}^2(\Omega)}\|\bm{u}_I-\bm{u}_h\|_{\Hone (\Omega)} + (\bm{f},\bm{u}_I-\bm{u}_h).
\end{align*}
Combining the three terms, we have

\begin{align*}
    C\|\bm{u}_I-\bm{u}_h\|^2_{\Hone (\Omega)} \leq (\bm{f}-\bm{f}_h,\bm{u}_I-\bm{u}_h)+  C_1h|\bm{u}|_{\bm{H}^2(\Omega)}\|\bm{u}_I-\bm{u}_h\|_{\Hone (\Omega)}.
\end{align*}
To estimate the term $(\bm{f}-\bm{f}_h,\bm{u}_I-\bm{u}_h)$, it is sufficient to take $\bm{f}_h = \bm{\Pi}_{0}^0\bm{f}$ as the $L^2$ projection onto constants.  
\begin{align*}
    (\bm{f}-\bm{f}_h,\bm{u}_I-\bm{u}_h) &=(\bm{f}-\bm{\Pi}^0_{0}\bm{f},\bm{u}_I-\bm{u}_h) \\
    &=\sum_{E}\int_{E}{{(\bm{f}-\bm{\Pi}^0_{0}\bm{f} )\cdot (\bm{u}_I-\bm{u}_h) }\,d\bm{x}} \\
    &= \sum_{E}{\Bigr[\int_{E}{\bm{f}\cdot (\bm{u}_I-\bm{u}_h)\,d\bm{x}} - \int_{E}{\bm{\Pi}^0_{0}}\bm{f}\cdot (\bm{u}_I-\bm{u}_h)\,d\bm{x}\Bigr]} \\
    &= \sum_{E}{\Bigr[ \int_{E}{\bm{f}\cdot (\bm{u}_I-\bm{u}_h)\,d\bm{x}} - \int_{E}{\bm{f}\cdot \bm{\Pi}^0_{0}}(\bm{u}_I-\bm{u}_h)\,d\bm{x} \Bigr]} \\
    &= \sum_{E}{\int_{E}{\bm{f}\cdot \left[(\bm{u}_I-\bm{u}_h)-\bm{\Pi}^0_{0}(\bm{u}_I-\bm{u}_h)\right]}\,d\bm{x}}\\
    &\leq \sum_{E}{\|\bm{f}\|_{\bm{L}^2(E)}\|(\bm{u}_I-\bm{u}_h)-\bm{\Pi}^0_{0}(\bm{u}_I-\bm{u}_h))\|_{\bm{L}^2(E)}}\\
    &\leq C_1h\|\bm{f}\|_{\bm{L}^2(\Omega)}\|\bm{u}_I-\bm{u}_h\|_{\Hone (\Omega)}. 
\end{align*}
On combining the terms, we obtain 
\begin{align*}
    C\|\bm{u}_I-\bm{u}_h\|^2_{\Hone (\Omega)} \leq C_1h(\|\bm{f}\|_{\bm{L}^2 (\Omega)}+|\bm{u}|_{\bm{H}^2(\Omega)})\|\bm{u}_I-\bm{u}_h\|_{\Hone (\Omega)}.
\end{align*}
Now we have the estimate of the $\Hone$ error as
\begin{align*}
    \|\bm{u}-\bm{u}_h\|_{\Hone (\Omega)} &\leq C_1h|\bm{u}|_{\bm{H}^2(\Omega)} + C_2h(\|\bm{f}\|_{\bm{L}^2(\Omega)}+|\bm{u}|_{\bm{H}^2(\Omega)})\\
    &\leq Ch(\|\bm{f}\|_{\bm{L}^2 (\Omega)} + |\bm{u}|_{\bm{H}^2 (\Omega)}).
    \tag*{\qedhere}
\end{align*}
\end{proof}

With the error in $\Hone$, we can also find an error estimate for the $\bm{L}^2$ norm.
\begin{proposition}
Let $\bm{u} $ be the exact solution to the strong problem (\ref{strongproblem}), and $\bm{f}$ the associated body force. For $h$ sufficiently small, there exists a constant $C>0$ such that the error of the solution $\bm{u}_h$ to the discrete weak problem is bounded in the $\bm{L}^2$ norm by
\begin{align}
    \|\bm{u}-\bm{u}_h\|_{\bm{L}^2(\Omega)} \leq Ch^2\left(|\bm{u}|_{\bm{H}^2(\Omega)} + \|\bm{f}\|_{\bm{H}^1(\Omega)}\right).
\end{align}
\end{proposition}
\begin{proof}
First, let $\bm{\psi}$ be a solution to the auxiliary problem: find $\bm{\psi} \in \bm{H}^2 \cap \bm{H}^1_0$ such that 
\begin{align}\label{aux_problem}
    a(\bm{\psi},\bm{v}) =(\bm{u}-\bm{u}_h,\bm{v}) \quad \forall \bm{v} \in \bm{H}^1_0.
\end{align}
Then $\bm{\psi}$ can be shown to satisfy the following 
inequalities~\cite{elasticdaveiga}:
\begin{subequations}
\begin{align}
    |\bm{\psi}|_{\bm{H}^2(\Omega)} &\leq C_1\|\bm{u}-\bm{u}_h\|_{\bm{L}^2(\Omega)}\label{psiestimate1}, \\
    \|\bm{\psi}\|_{\bm{H}^1_0(\Omega)} &\leq C_2 \|\bm{u}-\bm{u}_h\|_{\bm{L}^2(\Omega)}\label{psiestimate2}.
\end{align}
\end{subequations}
We estimate 
\begin{align*}
    \|\bm{u}-\bm{u}_h\|^2_{\bm{L}^2} &= (\bm{u}-\bm{u}_h,\bm{u}-\bm{u}_h) \\
    &= a(\bm{\psi},\bm{u}-\bm{u}_h) \\ 
    &= a(\bm{\psi}-\bm{\psi}_I+\bm{\psi}_I,\bm{u}-\bm{u}_h) \\
    &= {a(\bm{\psi}-\bm{\psi}_I,\bm{u}-\bm{u}_h)}+ {a(\bm{\psi}_I,\bm{u}-\bm{u}_h)},
\end{align*}
where $\bm{\psi}_I$ is the interpolation of $\bm{\psi}$. We now estimate each of the terms separately.   
For the second term, we write
\begin{align*}
    a(\bm{\psi}_I,\bm{u}-\bm{u}_h) &= a(\bm{\psi}_I,\bm{u}) - a(\bm{\psi}_I,\bm{u}_h) \\
    &= a(\bm{\psi}_I,\bm{u})-a_h(\bm{\psi}_I,\bm{u}_h)+a_h(\bm{\psi}_I,\bm{u}_h) - a(\bm{\psi}_I,\bm{u}_h) \\
    &= (\bm{f},\bm{\psi}_I) - (\bm{f}_h,\bm{\psi}_I) + a_h(\bm{\psi}_I,\bm{u}_h) -a(\bm{\psi}_I,\bm{u}_h) \\
    &= (\bm{f}-\bm{f}_h,\bm{\psi}_I) + \bigl(a_h(\bm{\psi}_I,\bm{u}_h)-a(\bm{\psi}_I,\bm{u}_h)\bigr). \\ 
\end{align*}
Then we have 
\begin{align}
    \|\bm{u}-\bm{u}_h\|^2_{\bm{L}^2} = \underbrace{{a(\bm{\psi}-\bm{\psi}_I,\bm{u}-\bm{u}_h)}}_{A} + \underbrace{(\bm{f}-\bm{f}_h,\bm{\psi}_I)}_{B} + \underbrace{\bigl(a_h(\bm{\psi}_I,\bm{u}_h)-a(\bm{\psi}_I,\bm{u}_h)\bigr)}_{C}.\label{estimate_u_h_L2}
\end{align}
We estimate each of the terms separately using  Cauchy-Schwarz,~\eqref{L2plusH1estimate},~\eqref{proj0estimate},
and~\eqref{h0estimate}. For term $A$ in~\eqref{estimate_u_h_L2}, 
we estimate 
\begin{align}
    {a(\bm{\psi}-\bm{\psi}_I,\bm{u}-\bm{u}_h)} &\leq  \|\bm{\psi}-\bm{\psi}_I\|_{\bm{H}^1_0(\Omega)} \|\bm{u}-\bm{u}_h\|_{\bm{H}^1_0(\Omega)} \notag\\
    &\leq Ch\|\bm{u}-\bm{u}_h\|_{\bm{H}^1_0(\Omega)}|\psi|_{\bm{H}^2(\Omega)} \notag\\
    &\leq Ch\|\bm{u}-\bm{u}_h\|_{\bm{H}^1_0(\Omega)}\|\bm{u}-\bm{u}_h\|_{\bm{L}^2(\Omega)}\notag \\
    &\leq Ch^2\|\bm{u}-\bm{u}_h\|_{\bm{L}^2(\Omega)}(|\bm{u}|_{\bm{H}^2(\Omega)}+\|\bm{f}\|_{\bm{L}^2(\Omega)}) \label{estimate_a}.
\end{align}
For term $B$ in~\eqref{estimate_u_h_L2}, we compute
\begin{align*}
    (\bm{f}-\bm{f}_h,\bm{\psi}_I) &= (\bm{f}-\bm{\Pi}_{0}^0\bm{f},\bm{\psi}_I) \\
    &= (\bm{f}-\bm{\Pi}_{0}^0\bm{f},\bm{\psi}_I-\bm{\psi}+\bm{\psi}) \\
    &= (\bm{f}-\bm{\Pi}_{0}^0\bm{f},\bm{\psi}_I-\bm{\psi}) + (\bm{f}-\bm{\Pi}_{0}^0\bm{f},\bm{\psi}) \\
    &= (\bm{f}-\bm{\Pi}_{0}^0\bm{f},\bm{\psi}_I-\bm{\psi}) + (\bm{f}-\bm{\Pi}_{0}^0\bm{f},\bm{\psi}-\bm{\Pi}^0_{0}{\bm{\psi}}) + (\bm{f}-\bm{\Pi}_{0}^0\bm{f},\bm{\Pi}^0_{0}{\bm{\psi}}).
\end{align*}
But by definition of $\bm{\Pi}_0^0\bm{f}$, we have $(\bm{f}-\bm{\Pi}_{0}^0\bm{f},\bm{\Pi}^0_{0}{\bm{\psi}})=0$, and hence
\begin{align}
    (\bm{f}-\bm{f}_h,\bm{\psi}_I) &= (\bm{f}-\bm{\Pi}_{0}^0\bm{f},\bm{\psi}_I-\bm{\psi}) + (\bm{f}-\bm{\Pi}_{0}^0\bm{f},\bm{\psi}-\bm{\Pi}^0_{0}{\bm{\psi}}) \notag \\
    &\leq \|\bm{f}-\bm{\Pi}_{0}^0\bm{f}\|_{\bm{L}^2(\Omega)} \|\bm{\psi}_I-\bm{\psi}\|_{\bm{L}^2(\Omega)} + \|\bm{f}-\bm{\Pi}_{0}^0\bm{f}\|_{\bm{L}^2(\Omega)} \|\bm{\psi}-\bm{\Pi}^0_{0}{\bm{\psi}}\|_{\bm{L}^2(\Omega)}\notag \\
    &\leq \|\bm{f}-\bm{\Pi}_{0}^0\bm{f}\|_{\bm{L}^2(\Omega)} (\|\bm{\psi}_I-\bm{\psi} \|_{\bm{L}^2(\Omega)} + \|\bm{\psi}-\bm{\Pi}^0_{0}{\bm{\psi}}\|_{\bm{L}^2(\Omega)}) \notag \\
    &\leq C_1h\|\bm{f}\|_{\bm{H}^1_0(\Omega)}(C_2h|\bm{\psi}|_{\bm{H}^2(\Omega)}+C_3h\|\bm{\psi}\|_{\bm{H}^1_0(\Omega)}) \notag \\
    &\leq Ch^2\|\bm{f}\|_{\bm{H}^1_0(\Omega)}\|\bm{u}-\bm{u}_h\|_{\bm{L}^2(\Omega)}\label{estimate_b}.
\end{align}
For term $C$ in~\eqref{estimate_u_h_L2}, we first apply the definition of the $L^2$ projection to rewrite it as:

\begin{align*}
    a_h(\bm{\psi}_I,\bm{u}_h)-a(\bm{\psi}_I,\bm{u}_h) &= \sum_{E}{\int_{E}{\bigl[ \projLtwo{\bm{\psi}_I}:\mathbb{C}:\projLtwo{\bm{u}_h}-\bm{\varepsilon}(\bm{\psi}_I):\mathbb{C}:\bm{\varepsilon}(\bm{u}_h)}} \bigl] \,d\bm{x} \\ 
    &= \sum_{E}{\int_{E}{\bigl[ {\bm{\varepsilon}(\bm{\psi}_I)}:\mathbb{C}:\projLtwo{\bm{u}_h}-\bm{\varepsilon}(\bm{\psi}_I):\mathbb{C}:\bm{\varepsilon}(\bm{u}_h)} \bigr]\,d\bm{x}} \\
    &= \sum_{E}{\int_{E}{\bm{\varepsilon}(\bm{\psi}_I):\mathbb{C}:(\projLtwo{\bm{u}_h}-\bm{\varepsilon}(\bm{u}_h))}\,d\bm{x}}.
\end{align*}
Now, add and subtract $\projLtwo{\bm{\psi}_I}$ and apply the definition of $\projLtwo{\bm{u}_h}$ to simplify:

\begin{align*}
     a_h(\bm{\psi}_I,\bm{u}_h)-a(\bm{\psi}_I,\bm{u}_h) &=  \sum_{E}{\left[\int_{E}{(\bm{\varepsilon}(\bm{\psi}_I)-\projLtwo{\bm{\psi}_I}):\mathbb{C}:(\projLtwo{\bm{u}_h}-\bm{\varepsilon}(\bm{u}_h))}\, d\bm{x}\right.} \\
     &\qquad \quad \ +\left. \int_{E}{{\projLtwo{\bm{\psi}_I}}:\mathbb{C}:(\projLtwo{\bm{u}_h}-\bm{\varepsilon}(\bm{u}_h))}\, d\bm{x}\right] \\
     &= \sum_{E}{\int_{E}{(\bm{\varepsilon}(\bm{\psi}_I)-\projLtwo{\bm{\psi}_I}):\mathbb{C}:(\projLtwo{\bm{u}_h}-\bm{\varepsilon}(\bm{u}_h))}\, d\bm{x}}.
\end{align*}
Adding and subtracting terms $\projLtwo{\bm{u}}$ and $\bm{\varepsilon}(\bm{u})$, we obtain
\begin{align*}
 &\sum_{E}{\int_{E}{(\bm{\varepsilon}(\bm{\psi}_I)-\projLtwo{\bm{\psi}_I}):\mathbb{C}:(\projLtwo{\bm{u}_h}-\bm{\varepsilon}(\bm{u}_h))}\, d\bm{x}} \\
  \begin{split}
  &=\sum_{E}{\Biggl[\underbrace{ \int_{E}{(\bm{\varepsilon}(\bm{\psi}_I)-\projLtwo{\bm{\psi}_I}):\mathbb{C}:(\projLtwo{\bm{u}_h}-\projLtwo{\bm{u}}})\,d\bm{x}  }_{D} }  \\
   &\qquad \quad \ + \underbrace{\int_{E}{(\bm{\varepsilon}(\bm{\psi}_I)-\projLtwo{\bm{\psi}_I}):\mathbb{C}:(\projLtwo{\bm{u}}-\bm{\varepsilon}(\bm{u})})\,d\bm{x}}_{E} \\
   &\qquad  \quad \ +\underbrace{\int_{E}{(\bm{\varepsilon}(\bm{\psi}_I)-\projLtwo{\bm{\psi}_I}):\mathbb{C}:(\bm{\varepsilon}(\bm{u})-\bm{\varepsilon}(\bm{u}_h)})\,d\bm{x} }_{F}\Biggr].
   \end{split}
\end{align*}

We estimate the three terms separately. For term $D$, we apply the Cauchy-Schwarz inequality and a standard estimate of the $L^2$ projection to write 
\begin{equation}\label{estimate_d}
 \begin{split}
 \sum_{E}{\int_{E}{(\bm{\varepsilon}(\bm{\psi}_I)-\projLtwo{\bm{\psi}_I}):\mathbb{C}:(\projLtwo{\bm{u}_h}-\projLtwo{\bm{u}}})\,d\bm{x}} \\ \leq C_1\|\bm{\varepsilon}(\bm{\psi}_I)-\bm{\Pi}_{\ell}^0\bm{\varepsilon}({\bm{\psi}_I)}\|_{\bm{L}^2(\Omega)} \|\bm{u}-\bm{u}_h\|_{\bm{H}^1_0(\Omega)}.
 \end{split}
\end{equation}
For term $E$, we again apply Cauchy-Schwarz and~\eqref{strainestimate} to write
\begin{equation}\label{estimate_e}
\begin{split}
     \sum_{E}{\int_{E}{(\bm{\varepsilon}(\bm{\psi}_I)-\projLtwo{\bm{\psi}_I}):\mathbb{C}:(\projLtwo{\bm{u}}-\bm{\varepsilon}(\bm{u}))}\,d\bm{x}} \\ \leq C_2h\|\bm{\varepsilon}(\bm{\psi}_I)-\bm{\Pi}_{\ell}^0\bm{\varepsilon}({\bm{\psi}_I})\|_{\bm{L}^2(\Omega)}|\bm{u}|_{\bm{H}^2(\Omega)}.
\end{split}
\end{equation}
Similarly for term $F$, we estimate 
\begin{equation}\label{estimate_f}
    \begin{split}
    \sum_{E}{\int_{E}{(\bm{\varepsilon}(\bm{\psi}_I)-\projLtwo{\bm{\psi}_I}):\mathbb{C}:(\bm{\varepsilon}(\bm{u})-\bm{\varepsilon}(\bm{u}_h)})\,d\bm{x}} \\ \leq C_3\|\bm{\varepsilon}(\bm{\psi}_I)-\bm{\Pi}_{\ell}^0\bm{\varepsilon}({\bm{\psi}_I})\|_{\bm{L}^2(\Omega)} \|\bm{u}-\bm{u}_h\|_{\bm{H}^1_0(\Omega)}.
    \end{split}
\end{equation}
Now combining~\eqref{estimate_d},~\eqref{estimate_e},~\eqref{estimate_f} and using~\eqref{strainestimate} and~\eqref{psiestimate1}, we obtain the estimate
\begin{align}
    a_h(\bm{\psi}_I,\bm{u}_h)-a(\bm{\psi}_I,\bm{u}_h) \leq Ch^2\|\bm{u}-\bm{u}_h\|_{\bm{L}^2(\Omega)}(|\bm{u}|_{\bm{H}^2(\Omega)}+\|\bm{f}\|_{\bm{L}^2(\Omega)}).\label{estimate_c}
\end{align}
Combining all the necessary terms from~\eqref{estimate_a},~\eqref{estimate_b},~\eqref{estimate_c},
the estimate becomes
\begin{align*}
    \|\bm{u}-\bm{u}_h\|_{\bm{L}^2(\Omega)} &\leq Ch^2(|\bm{u}|_{\bm{H}^2(\Omega)}+ \|\bm{f}\|_{\bm{L}^2(\Omega)} +\|\bm{f}\|_{\bm{H}^1_0(\Omega)}) \\ 
    &\leq Ch^2(|\bm{u}|_{\bm{H}^2(\Omega)}+ \|\bm{f}\|_{\bm{H}^1(\Omega)}).  \tag*{\qedhere}
\end{align*}
\end{proof}

\section{Numerical Results}\label{sec:numerical}
We present a series of numerical examples showing the application of the method to well-known benchmark problems in plane elasticity. We examine the errors using the $L^\infty$ and $L^2$ norms, as well as the energy seminorm, and compare the convergence rates of the method with the theoretical estimates. In particular, we use the following discrete measures:
\begin{subequations}
\begin{align}
    \|\bm{u} -\bm{u}_h\|_{\bm{L}^\infty(\Omega)} &= \max_{\bm{x} \in \Omega}{|\bm{u}(\bm{x}) - \bm{u}_h(\bm{x})|}, \\
    \|\bm{u} -\bm{u}_h \|_{\bm{L}^2(\Omega)} &= \sqrt{\sum_{E}{\int_{E}{|\bm{u}-\proj{\bm{u}_h}}|^2 \,d\bm{x}}},  \\
    \|\bm{u} -\bm{u}_h\|_{a} &= \sqrt{\sum_{E}{\int_{E}{(\overline{\bm{\varepsilon}}-\overline{\projLtwo{\bm{u}_h}})^T \bm{C}{(\overline{\bm{\varepsilon}}-\overline{\projLtwo{\bm{u}_h}}) }\,d\bm{x}}}}.
\end{align}
\end{subequations}
To compute the integrals for the $L^2$ norm and the energy seminorm, we use the scaled boundary cubature (SBC) scheme~\cite{2021}. The SBC scheme allows us to convert integration over arbitrary polygons into an equivalent integration over the unit square. In particular, for a polygonal element $E$ and a scalar function $f$, we expand the integral over $E$ to write
\begin{align}
    \int_{E}{f \,d\bm{x}} = \sum_{i=1}^{N_E}\ell_i \,
     \acmajor{|e_i|} 
    \int_{0}^{1}\int_{0}^1{\xi f(\varphi(\xi,t))\, d\xi dt},
\end{align}
where $\ell_i$ is \acmajor{the} 
signed distance from a fixed point to the $i$-th edge, 
\acmajor{$|e_i|$ is the length of the $i$-th edge,}
and $\varphi$ is called the 
SB-parametrization~\cite{2021}. To compute the integrals over the square, we use a tensor-product Gauss quadrature rule.

\subsection{Patch test}
We first consider the displacement patch test. Let $\Omega = (0,1)^2$, and we impose an affine displacement field on the boundary:
\begin{equation*}
    u(\bm{x}) = x \ \ \textrm{and} \  \
    v(\bm{x}) = x+y \quad \textrm{on } \partial \Omega .
\end{equation*}
The exact solution is the extension of the boundary conditions onto the entire domain $\Omega$. We assess the accuracy of the numerical solution for three different types of meshes with 16 elements in each case. The first is a uniform square mesh, the second is a random Voronoi mesh, and the third is a Voronoi mesh that is obtained after applying three Lloyd iterations (see Fig.\ref{fig:three graphs}).
The results are listed in Table~\ref{tab:modifiedpatchtable}, which show that near machine-precision accuracy is realized. 
This indicates that the method passes the linear displacement patch test.
\begin{figure}
     \centering
     \begin{subfigure}{0.32\textwidth}
         \centering
         \includegraphics[width=\textwidth]{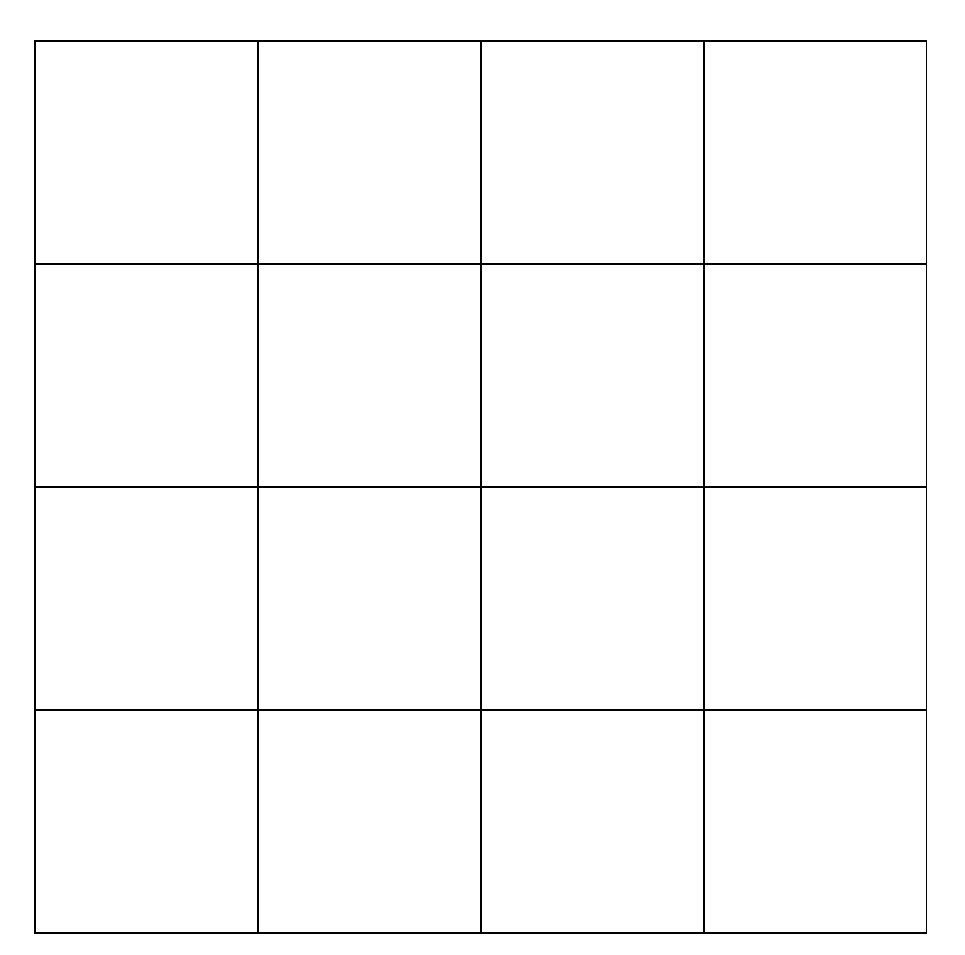}
         \caption{Uniform mesh}
         \label{fig:y equals x}
     \end{subfigure}
     \hfill
     \begin{subfigure}{0.32\textwidth}
         \centering
         \includegraphics[width=\textwidth]{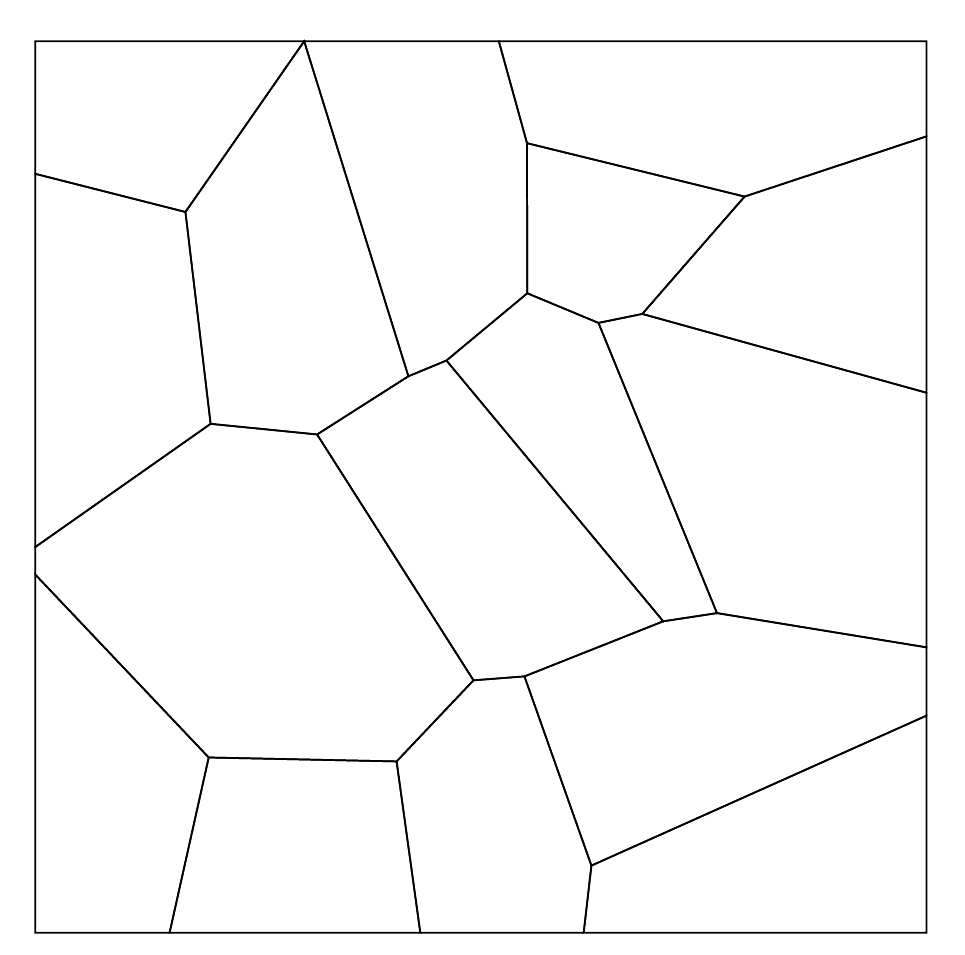}
         \caption{Random mesh}
         \label{fig:three sin x}
     \end{subfigure}
     \hfill
     \begin{subfigure}{0.32\textwidth}
         \centering
         \includegraphics[width=\textwidth]{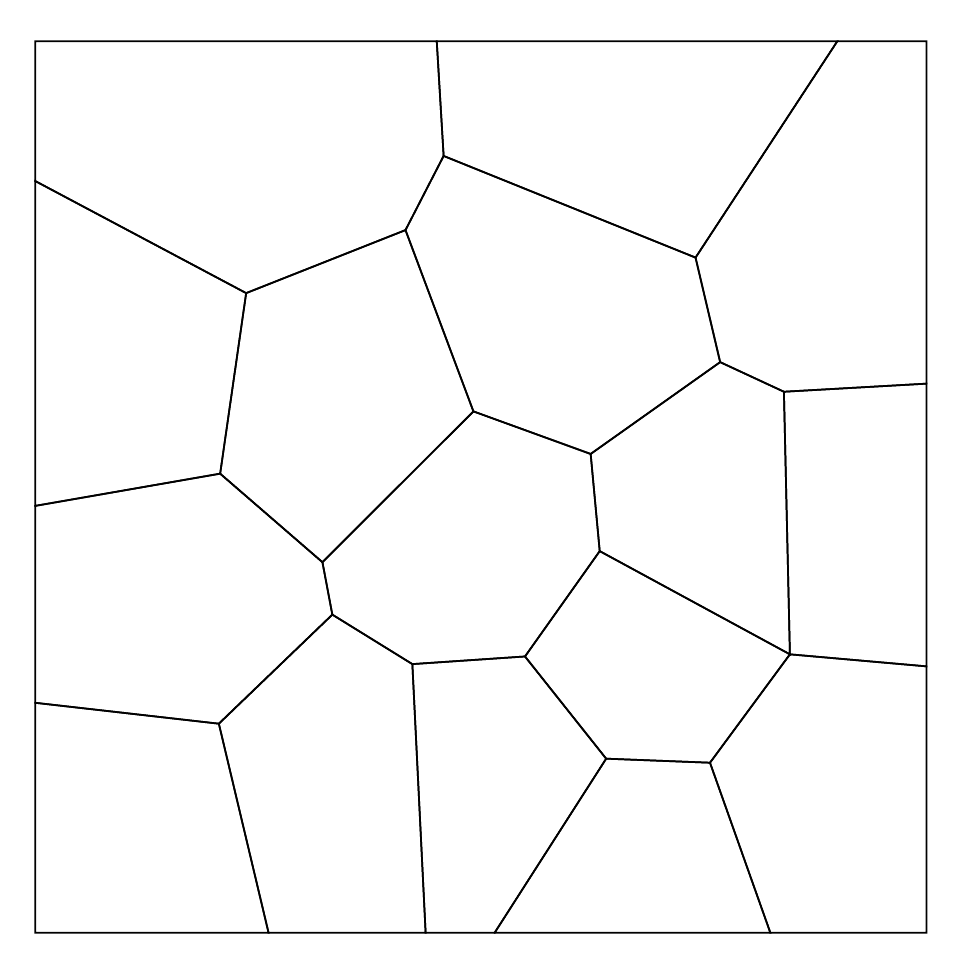}
         \caption{Lloyd iterated}
         \label{fig:five over x}
     \end{subfigure}
        \caption{Sample meshes used for the patch test.}
        \label{fig:three graphs}
\end{figure}
\begin{table}
    \centering
    \begin{tabular}{ |c| c | c |c| }
    \hline
     Mesh type & $L^\infty$ error 
     & $L^2$ error & Energy error \\ [0.5ex] 
    \hline
        Uniform & $3 \times10^{-16}$  & $2 \times10^{-16}$ &  $1 \times10^{-15}$
\\
        Random &$2 \times10^{-13}$&  $5 \times10^{-14}$ & $9 \times10^{-13}$ \\
        Lloyd iterated & $3 \times10^{-14}$ & $8 \times10^{-15}$ & $2 \times10^{-13}$\\
    \hline
    \end{tabular}
    \caption{Errors in the patch test on different types of meshes.}
    \label{tab:modifiedpatchtable}
\end{table}

\subsection{Eigenvalue analysis}
Consider the closed domain (unit square), $\bar{\Omega} = [0,1]^2$, which is
discretized using nine quadrilateral elements.  We are interested in the validity of the bounds in~\eqref{l_condition}. To this end, we solve the
element-eigenvalue problem, $\bm{K}_E \bm{d}_E = \lambda \bm{d}_E$,
to assess the physical and nonphysical (spurious) modes of the element. Each
element has three rigid-body (zero-energy) modes that each
correspond to a vanishing eigenvalue ($\lambda = 0$). For a stable
element, all other eigenvalues must be positive and bounded away from zero. We choose $\ell =0 ,1 ,2 , 3$ and measure 
the maximum number of spurious eigenvalues of the element
stiffness matrix as we artificially increase the number of 
nodes of the central element.  For a well-posed
discrete problem, the number of spurious eigenvalues should 
remain at zero. We show a few sample meshes in 
Figure~\ref{fig:eig_mesh}. In Figure~\ref{fig:eig_plots}, the resulting number of spurious eigenvalues as a function of the number of nodes of an element are plotted for $\ell = 0,1,2,3$. 

We find that for $\ell=0$, any polygon that is not a triangle ($N_E \ge 4 $) has
spurious modes, whereas for $\ell=1$, an element with $N_E \ge 6$ has
spurious modes. For $\ell=2$ and $\ell=3$, spurious eigenvalues appear for $N_E \ge 9$ and $N_E \ge 11$ in the central quadrilateral element, respectively.
This shows that~\eqref{eq:strong_l_condition} is sufficient but not strictly required to ensure that the element stiffness matrix has the correct rank and is devoid of nonphysical zero-energy modes.

\begin{figure}
     \centering
     \begin{subfigure}{.32\textwidth}
         \centering
         \includegraphics[width=\textwidth]{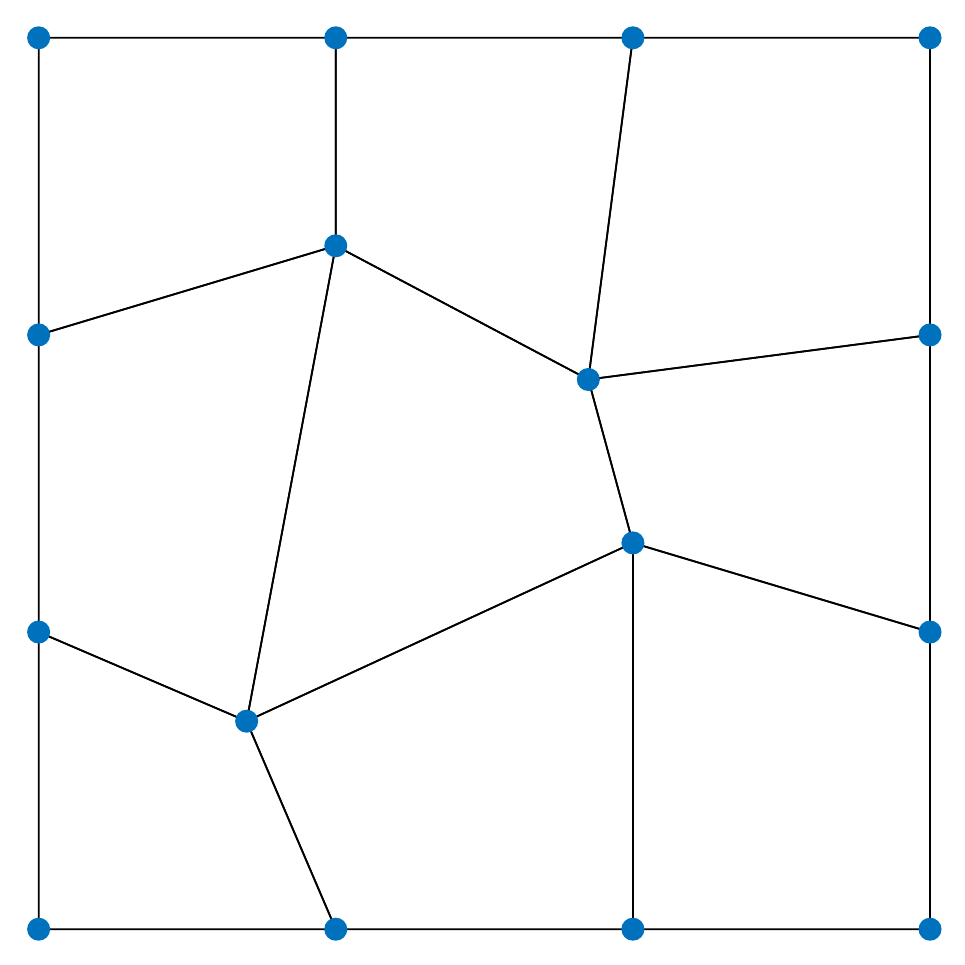}
         \caption{}
     \end{subfigure}
     \hfill
     \begin{subfigure}{.32\textwidth}
         \centering
         \includegraphics[width=\textwidth]{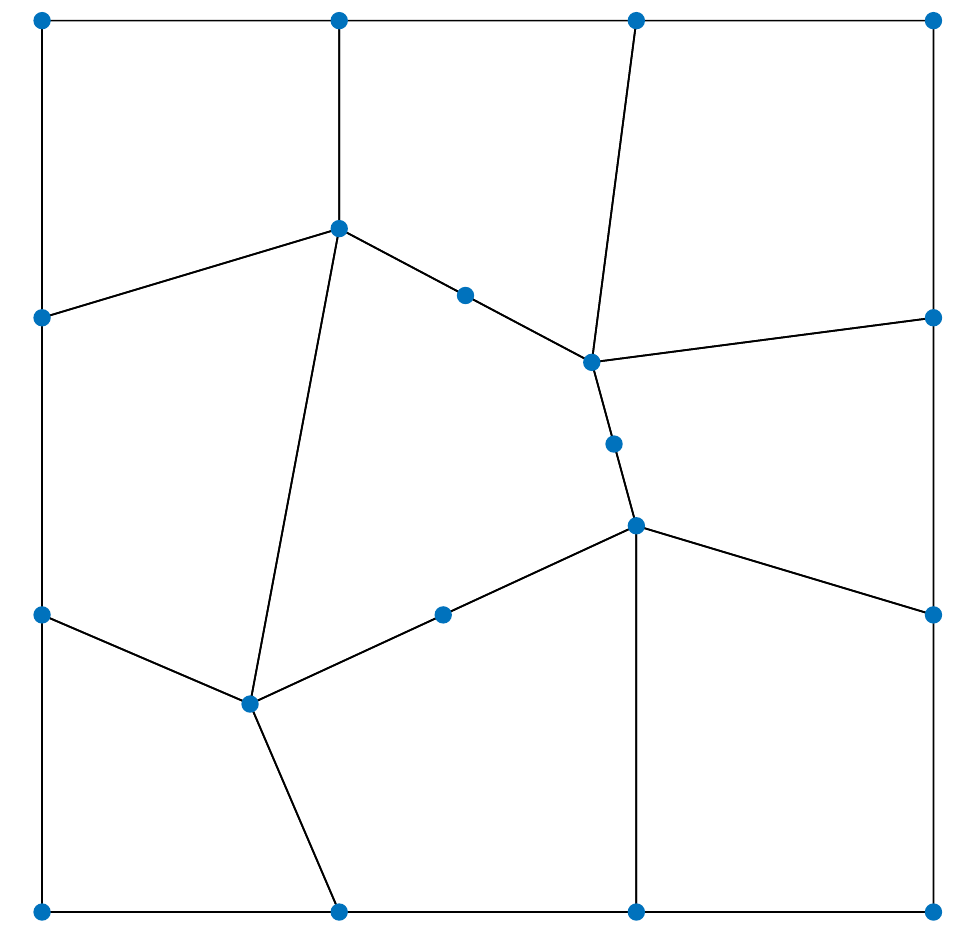}
         \caption{}
     \end{subfigure}
     \hfill
     \begin{subfigure}{.32\textwidth}
         \centering
         \includegraphics[width=\textwidth]{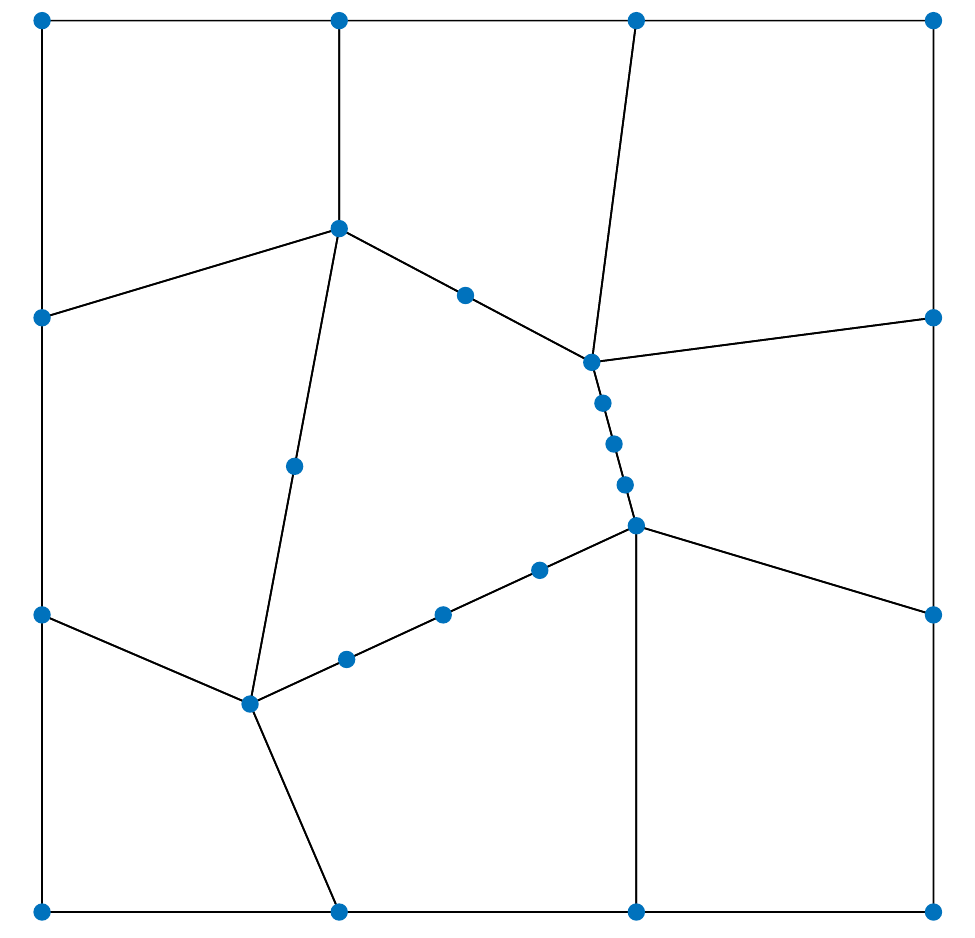}
         \caption{}
     \end{subfigure}
        \caption{ Sample meshes used in the element-eigenvalue analysis
        for $\ell = 0,1,2,3$. The central quadrilateral element has (a) 4 nodes, (b) 7 nodes, and (c) 12 nodes.}
        \label{fig:eig_mesh}
\end{figure}
\begin{figure}
     \centering
     \begin{subfigure}{.49\textwidth}
         \centering
         \includegraphics[width=\textwidth]{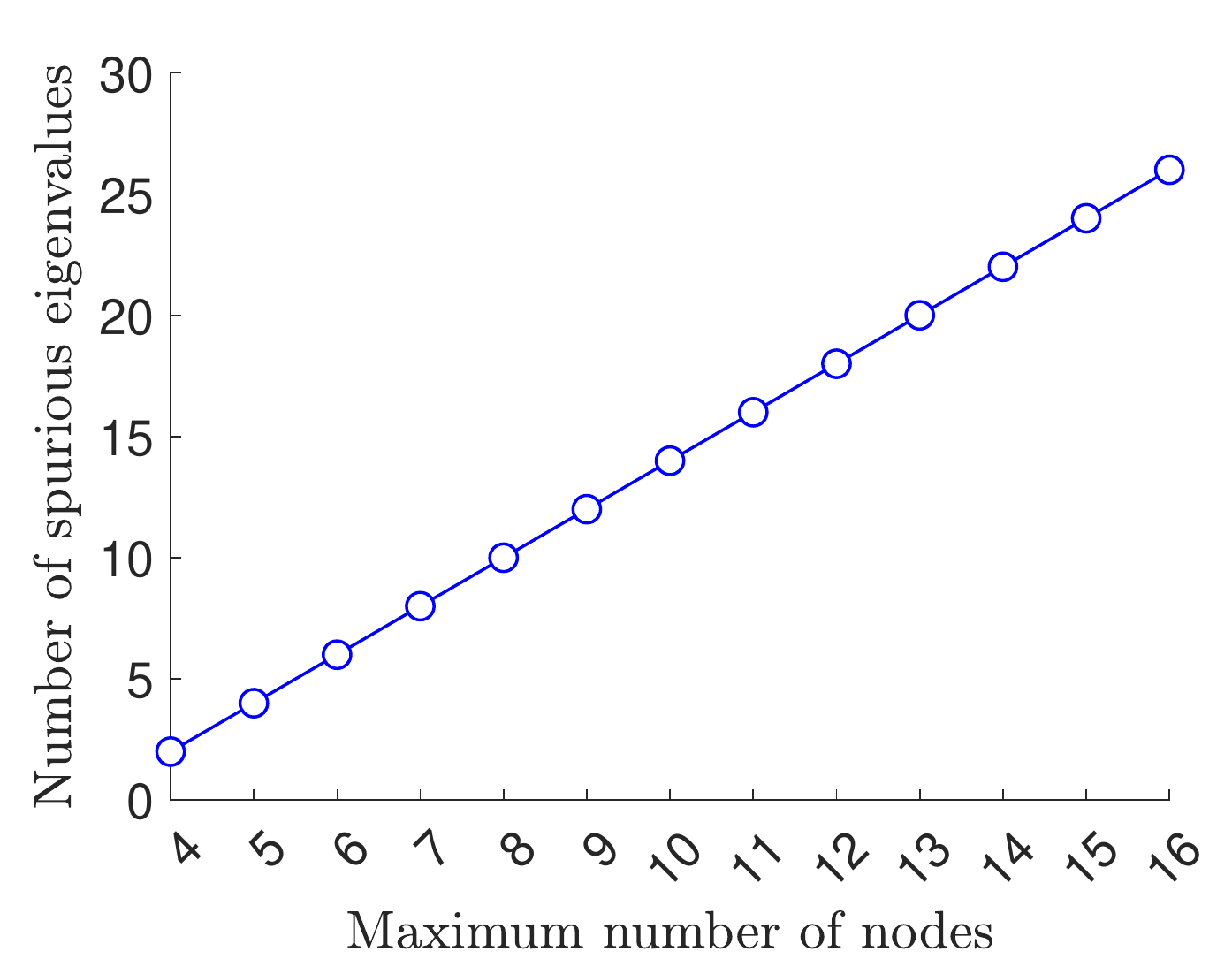}
         \caption{}
     \end{subfigure}
     \hfill
     \begin{subfigure}{.49\textwidth}
         \centering
         \includegraphics[width=\textwidth]{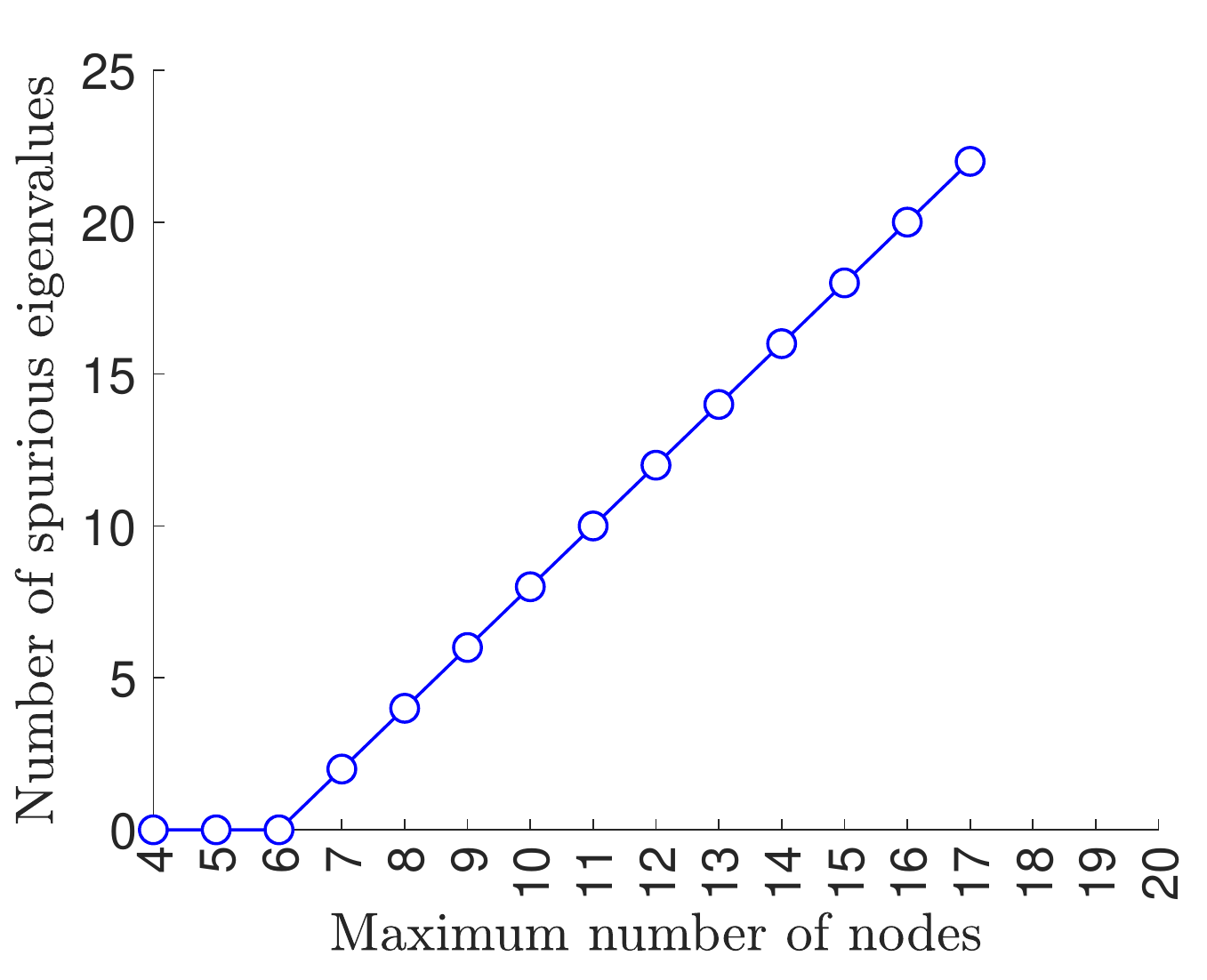}
         \caption{}
     \end{subfigure}
     \vfill
     \begin{subfigure}{.49\textwidth}
         \centering
         \includegraphics[width=\textwidth]{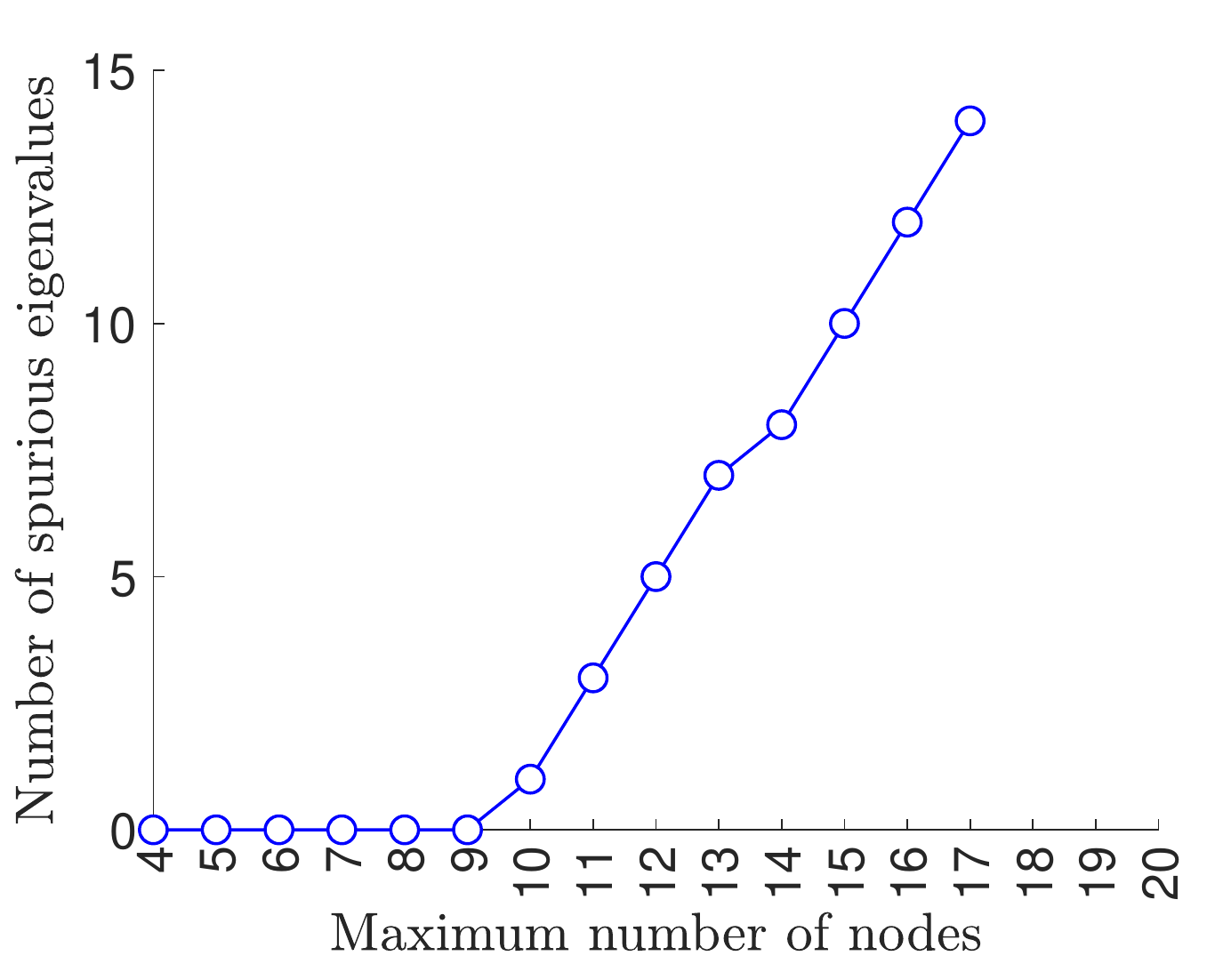}
         \caption{}
     \end{subfigure}
    \hfill     
    \begin{subfigure}{.49\textwidth}
         \centering
         \includegraphics[width=\textwidth]{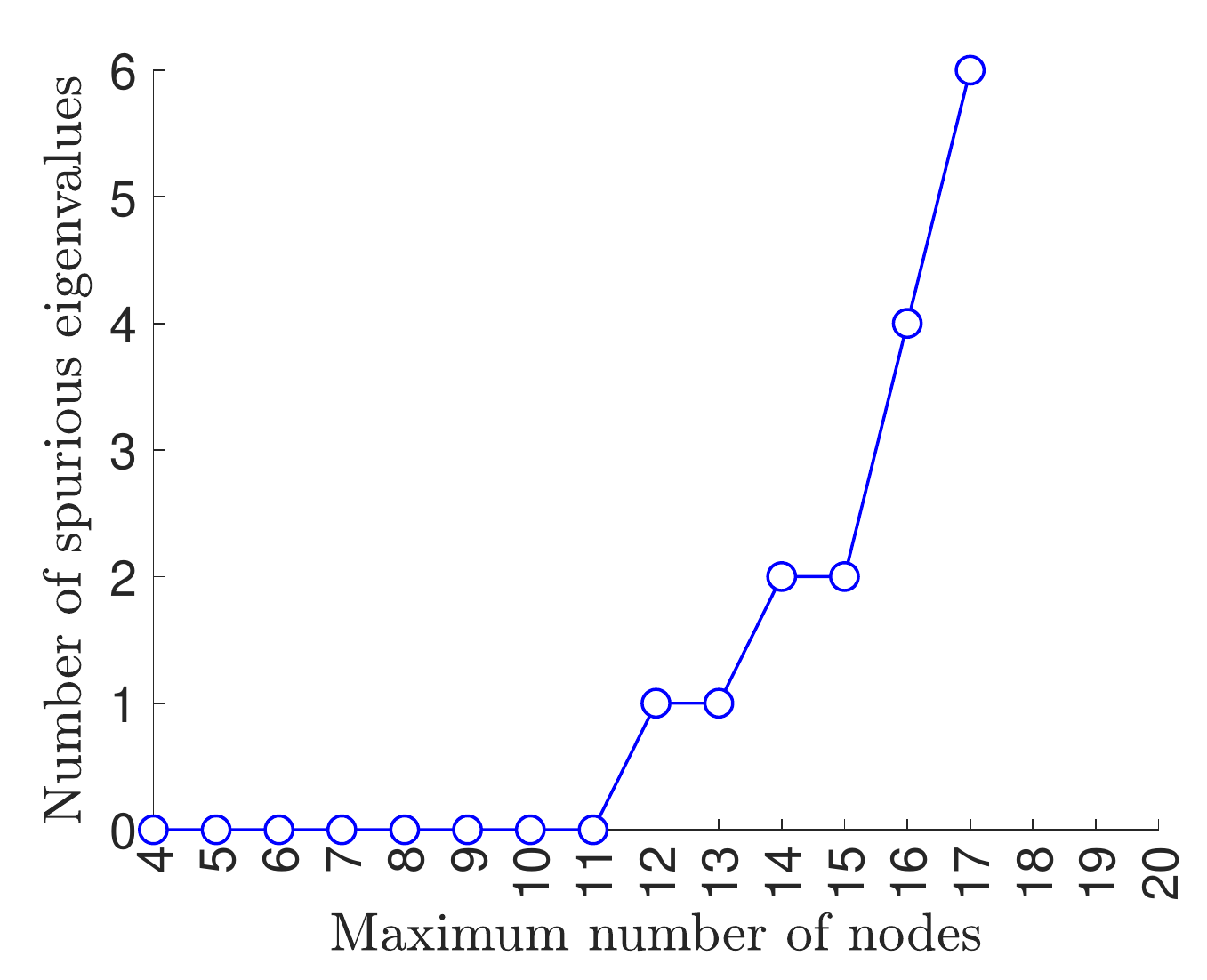}
         \caption{}
     \end{subfigure}
        \caption{Results of the element-eigenvalue analysis for 
        (a) $\ell=0$ , (b) $\ell=1$, (c) $\ell=2$, and (d) 
        $\ell=3$.}
        \label{fig:eig_plots}
\end{figure}
\acmajor{To further test the bound~\eqref{eq:strong_l_condition}, we examine the eigenvalues of the element stiffness matrix over a series of regular polygons (A. Russo, personal communication, April 2022). 
A few sample regular polygons are shown in Figure~\ref{fig:eig_mesh_poly}. In Figure~\ref{fig:eig_plots_poly}, 
we plot the number of spurious eigenvalues as a function of the 
number of nodes of a regular polygon. 
We again find that $\ell=0$ has spurious modes for all regular polygons $N_E\geq4$, and for $\ell=1$, regular polygons with $N_E \geq 5$ have spurious modes. For $\ell=2$ and $\ell=3$, there are additional eigenvalues that appear for $N_E\geq 7$ and $N_E\geq 9$, respectively. This shows that the inequality in~\eqref{eq:strong_l_condition} is strict for regular polygons.} 
\begin{figure}
     \centering
     \begin{subfigure}{.32\textwidth}
         \centering
         \includegraphics[width=\textwidth]{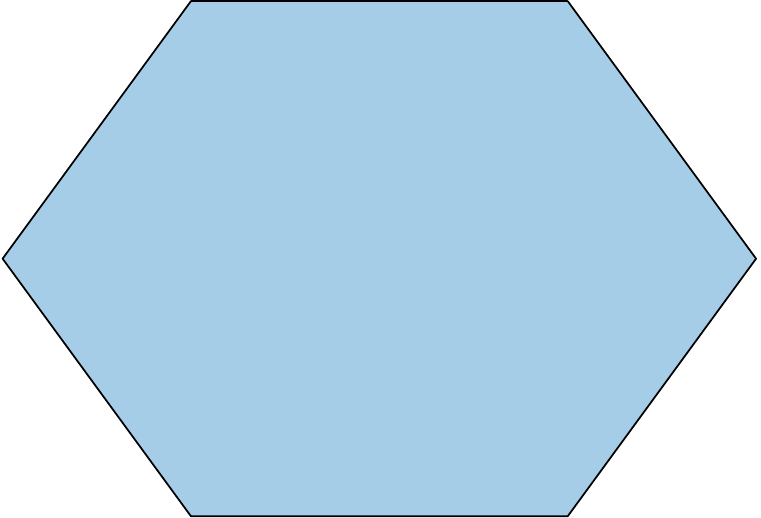}
         \caption{}
     \end{subfigure}
     \hfill
     \begin{subfigure}{.32\textwidth}
         \centering
         \includegraphics[width=\textwidth]{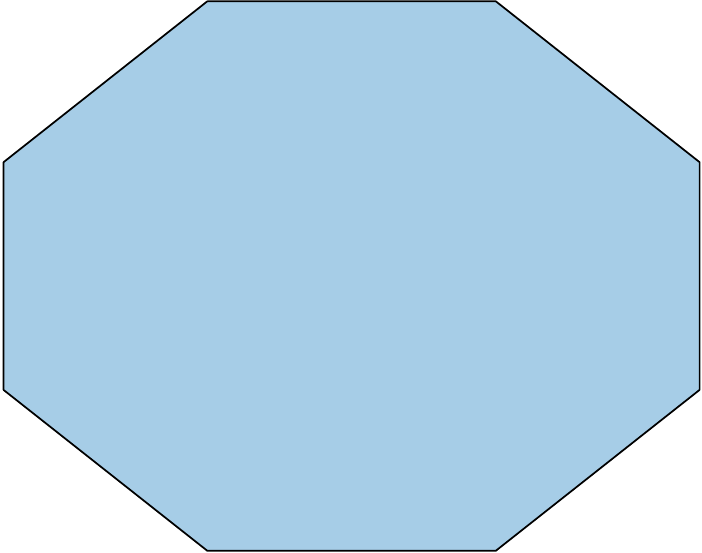}
         \caption{}
     \end{subfigure}
     \hfill
     \begin{subfigure}{.32\textwidth}
         \centering
         \includegraphics[width=\textwidth]{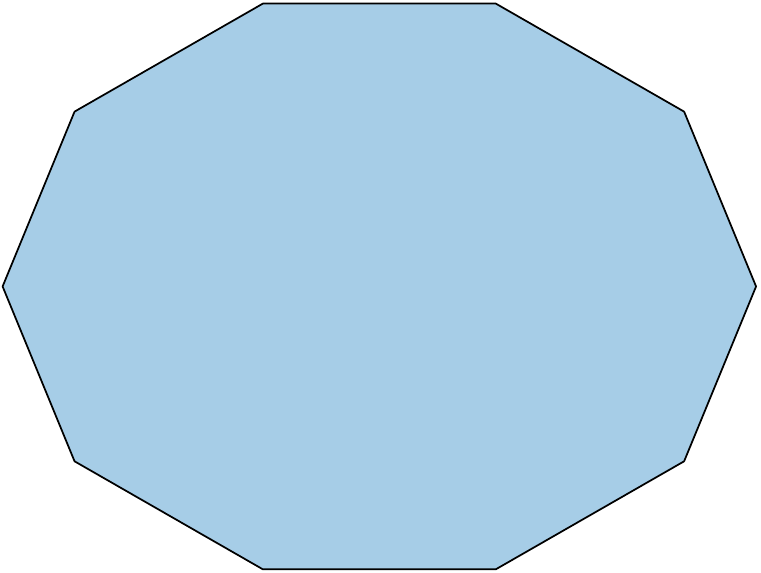}
         \caption{}
     \end{subfigure}
        \caption{ Sample regular polygons used in the element-eigenvalue analysis
        for $\ell = 0,1,2,3$.}
        \label{fig:eig_mesh_poly}
\end{figure}
\begin{figure}[!htb]
     \centering
     \begin{subfigure}{.49\textwidth}
         \centering
         \includegraphics[width=\textwidth]{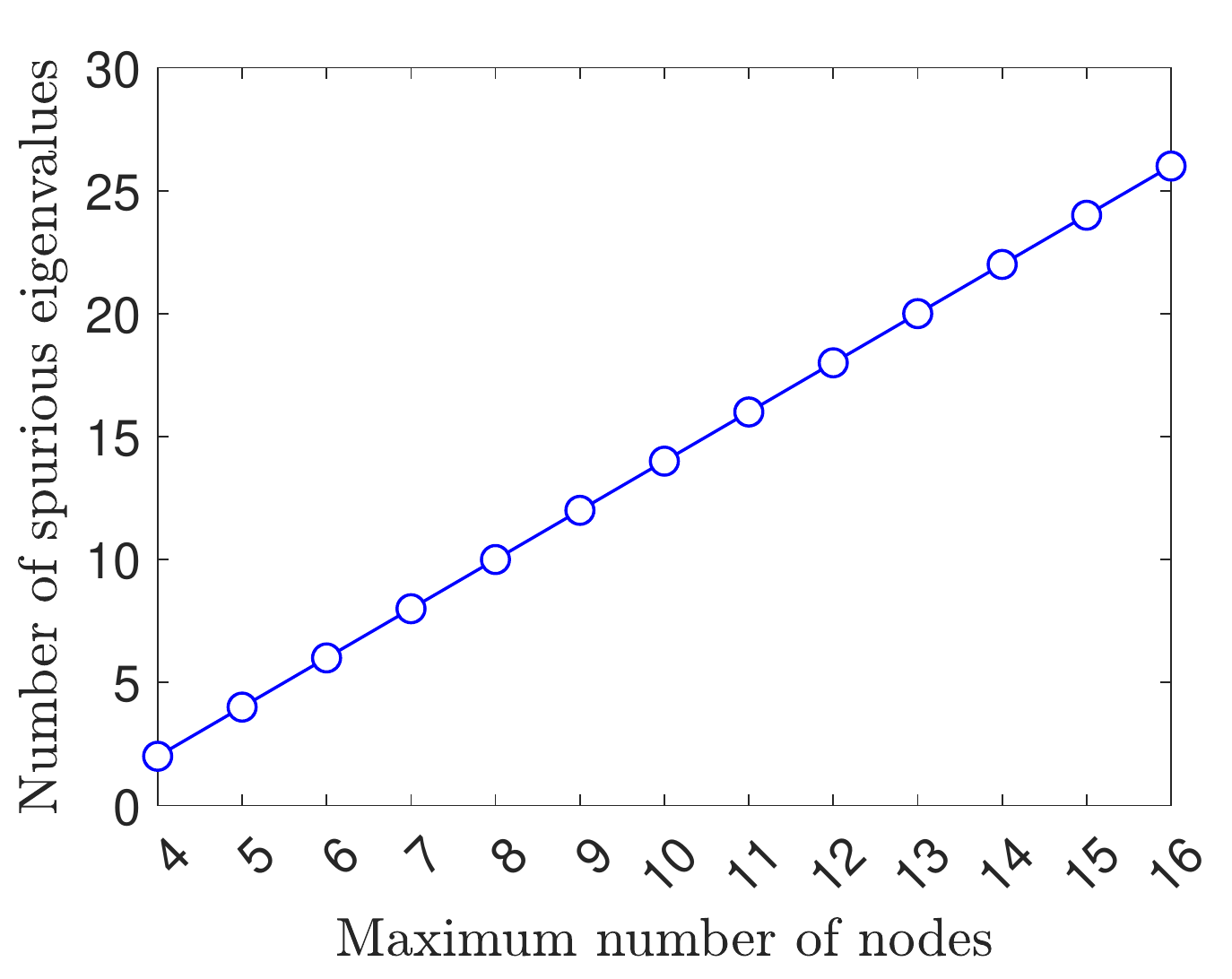}
         \caption{}
     \end{subfigure}
     \hfill
     \begin{subfigure}{.49\textwidth}
         \centering
         \includegraphics[width=\textwidth]{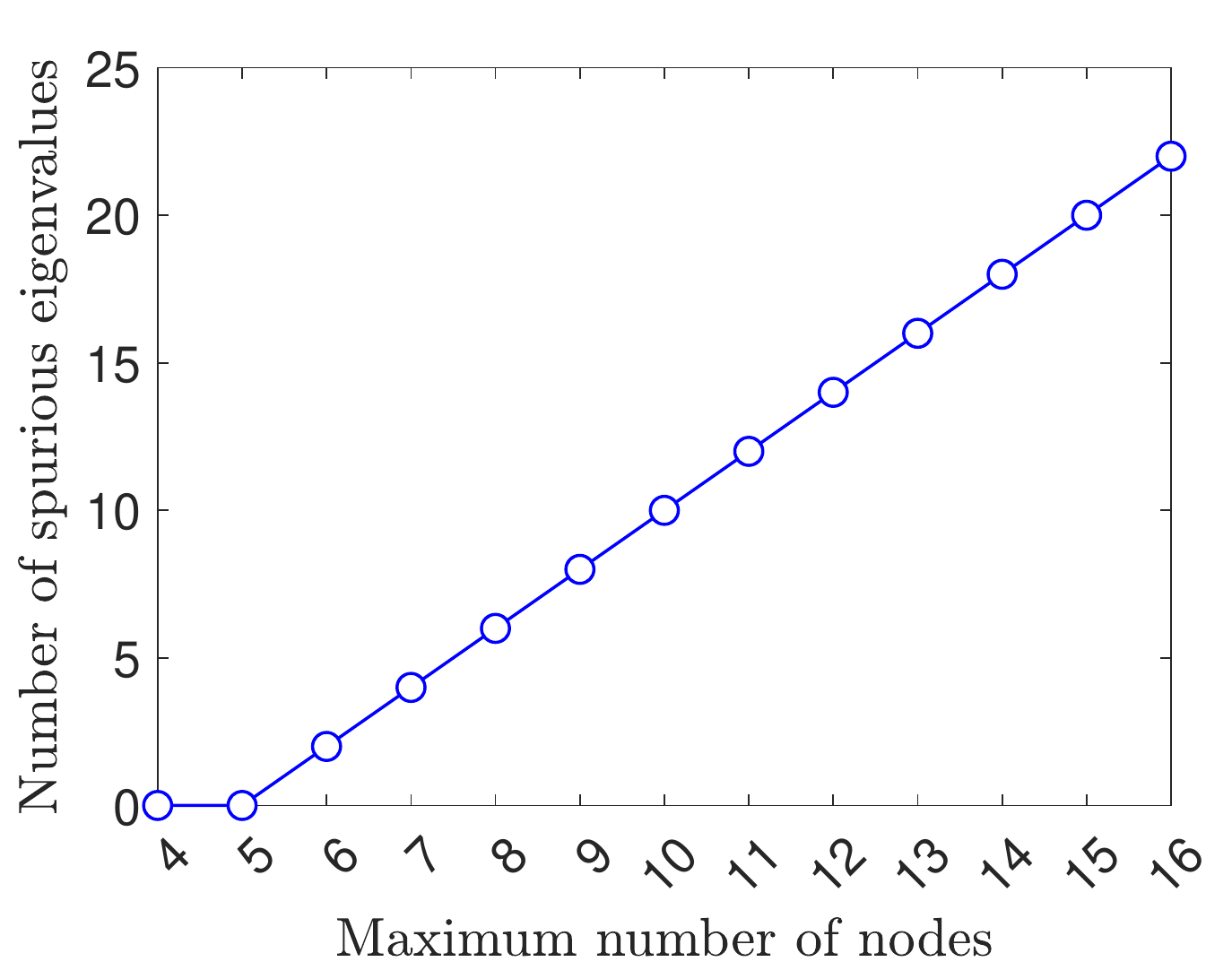}
         \caption{}
     \end{subfigure}
     \vfill
     \begin{subfigure}{.49\textwidth}
         \centering
         \includegraphics[width=\textwidth]{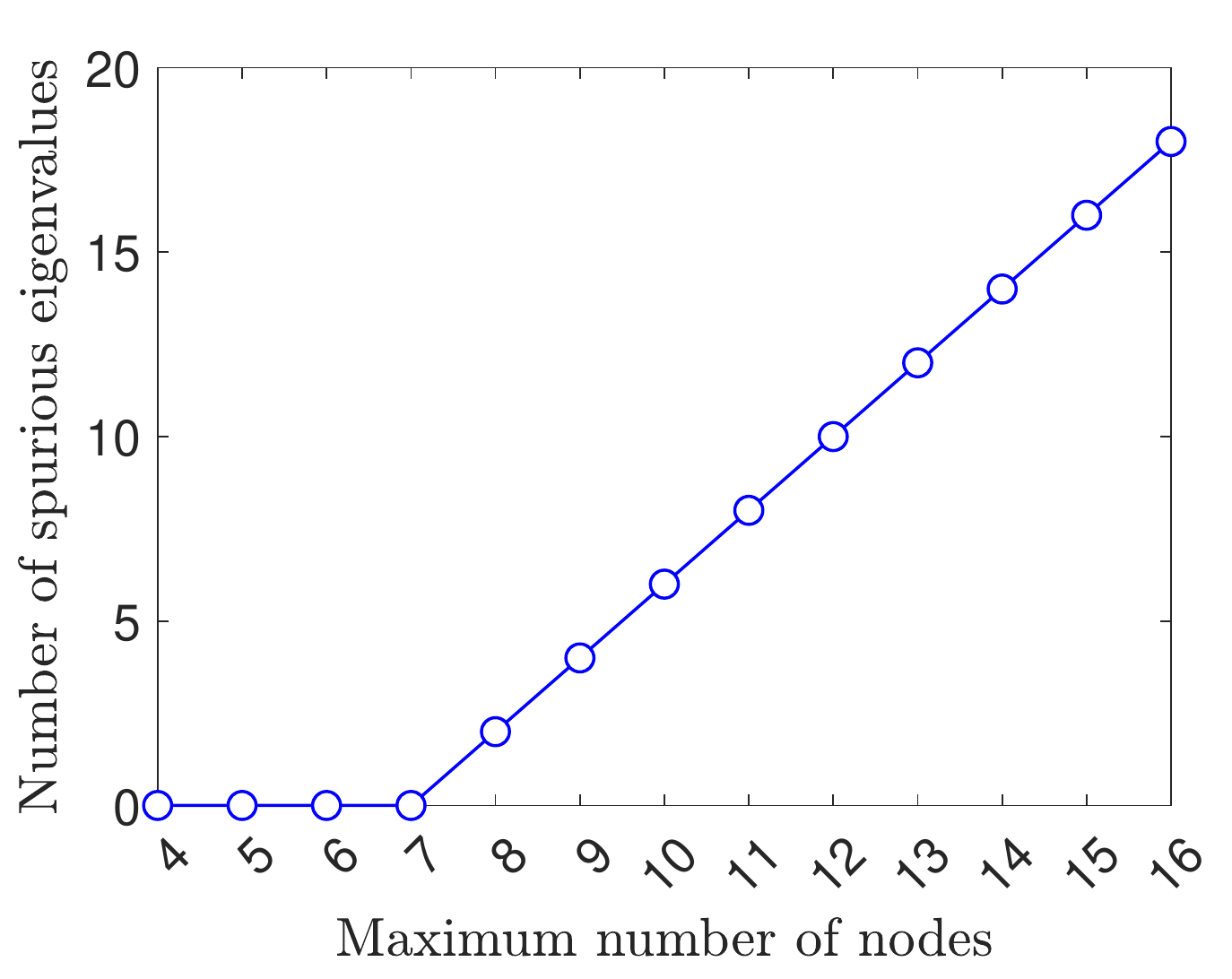}
         \caption{}
     \end{subfigure}
    \hfill     
    \begin{subfigure}{.49\textwidth}
         \centering
         \includegraphics[width=\textwidth]{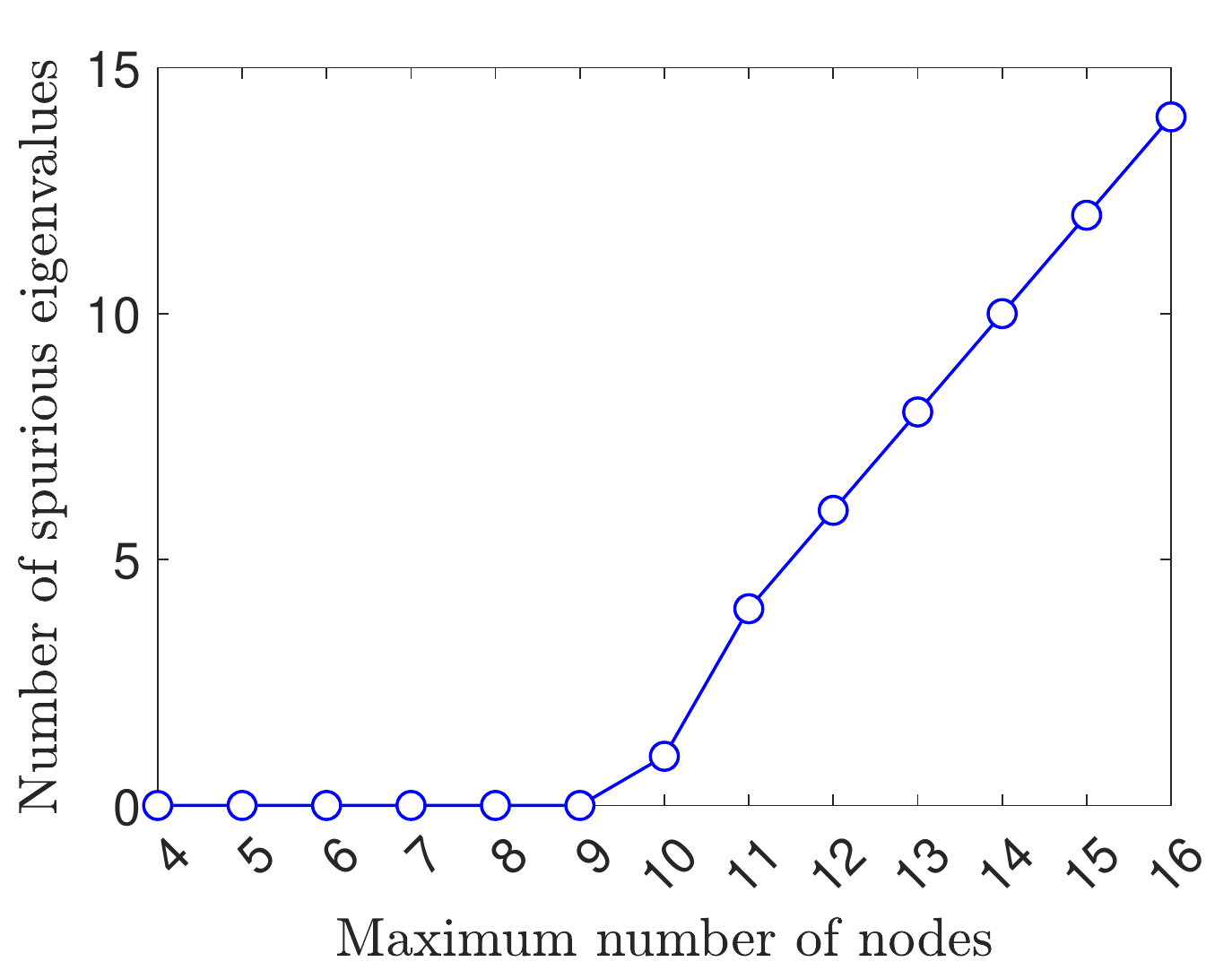}
         \caption{}
     \end{subfigure}
        \caption{Results of the regular polygon element-eigenvalue analysis for 
        (a) $\ell=0$ , (b) $\ell=1$, (c) $\ell=2$, and (d) 
        $\ell=3$.}
        \label{fig:eig_plots_poly}
\end{figure}
\subsection{Cantilever beam}
We now consider the problem of a cantilever beam, subjected to a shear end load~\cite{timoshenko1951theory}. In particular we consider the problem with material properties  $\acmajor{E_Y}=2\times 10^5$ psi and $\nu =0.3$, with plane stress assumptions. The beam has length $L=8$ inch, height $D=1$ inch and unit thickness. We apply a constant load $P=-1000$ psi on the right boundary. We test this problem on 
Lloyd iterated Voronoi meshes~\cite{talischi2012polymesher}. In 
Figure~\ref{fig:beammesh}, we show a few representative meshes.
\begin{figure}[!htb]
     \centering
     \begin{subfigure}{\textwidth}
         \centering
         \includegraphics[width=\textwidth]{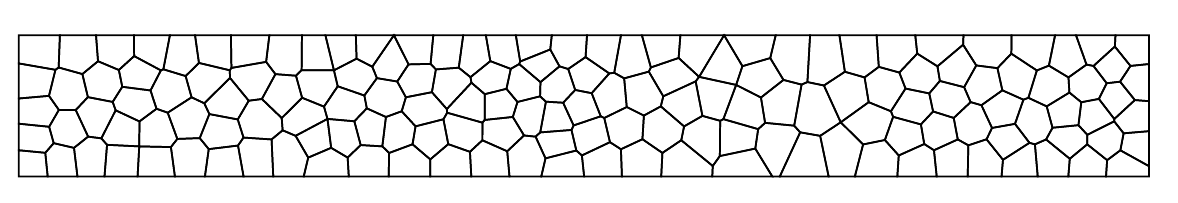}
         \caption{}
     \end{subfigure}
     \vfill
     \begin{subfigure}{\textwidth}
         \centering
         \includegraphics[width=\textwidth]{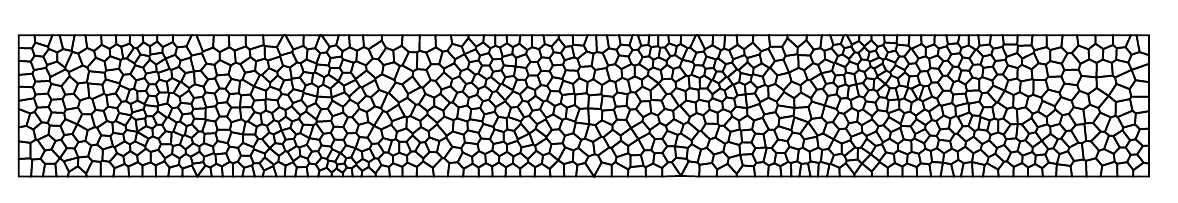}
         \caption{}
     \end{subfigure}
     \vfill
     \begin{subfigure}{\textwidth}
         \centering
         \includegraphics[width=\textwidth]{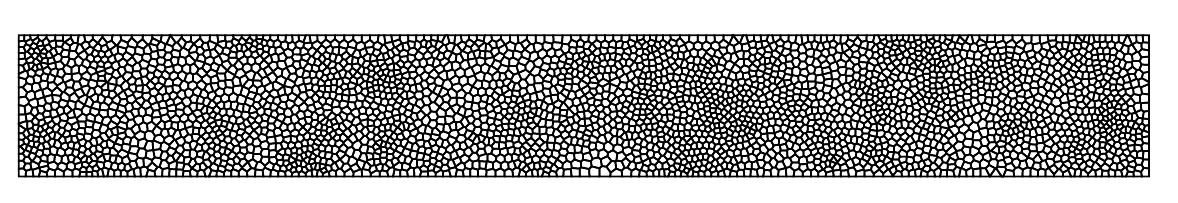}
         \caption{}
     \end{subfigure}
        \caption{Polygonal meshes used for the cantilever beam problem. (a) 150 elements, (b) 1000 elements and (c) 3500 elements.  }
        \label{fig:beammesh}
\end{figure}
For this problem, we compare the results of the stabilization-free VEM to a standard VEM method with a stabilization term~\cite{basicprinciple}. In Figure~\ref{fig:cantilever}, we plot the $L^2$ and energy errors of both the stabilization-free VEM and the standard VEM. 
We find that for the $L^2$ norm and energy seminorm, both methods produce second-order and first-order convergence rates, respectively. This agrees with the theoretical error estimates and demonstrates
that the stabilization-free method compares favorably with the standard stabilized virtual element
method. 
\begin{figure}[!h]
     \centering
     \begin{subfigure}{0.48\textwidth}
         \centering
         \includegraphics[width=\textwidth]{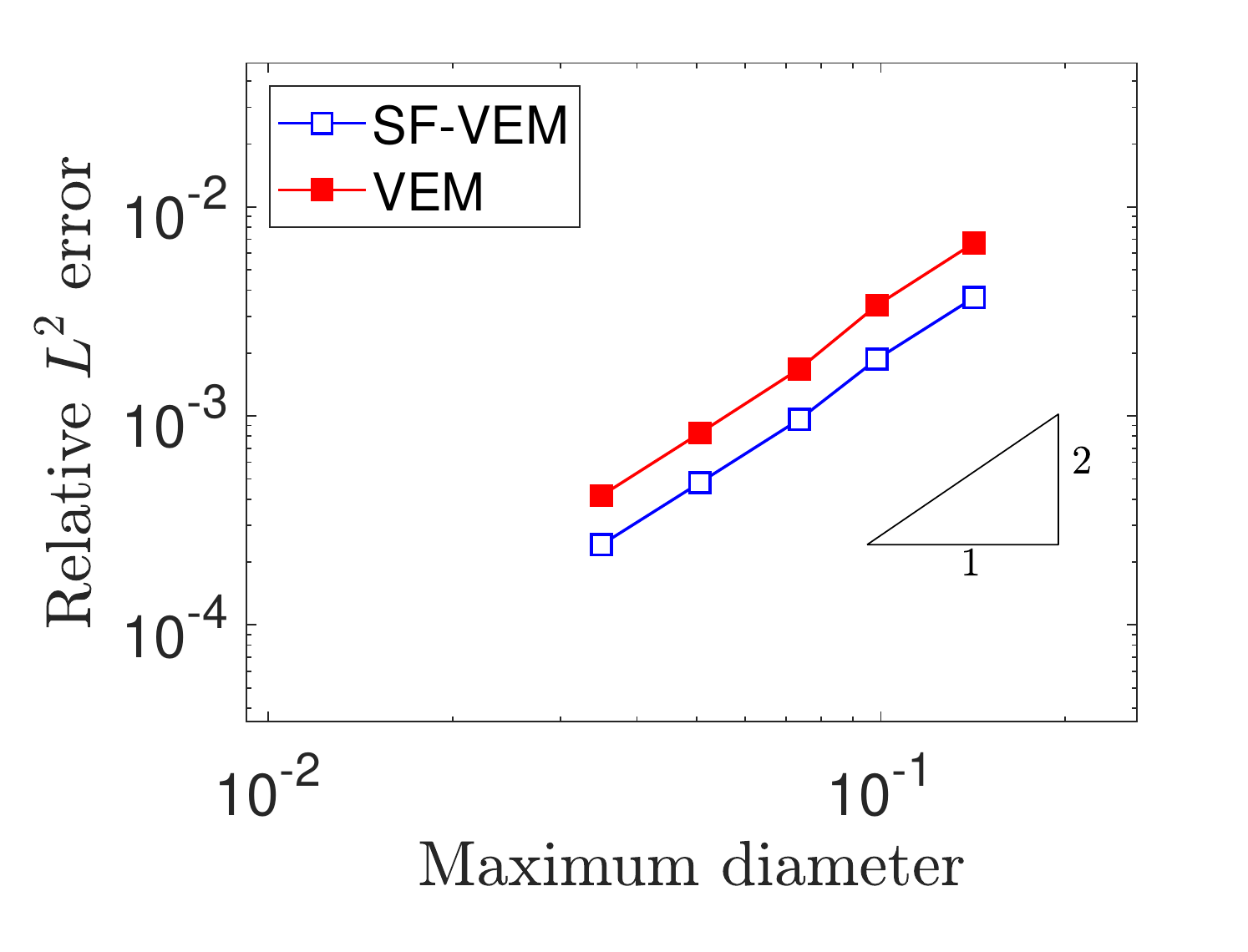}
         \caption{}
     \end{subfigure}
     \hfill
     \begin{subfigure}{0.48\textwidth}
         \centering
         \includegraphics[width=\textwidth]{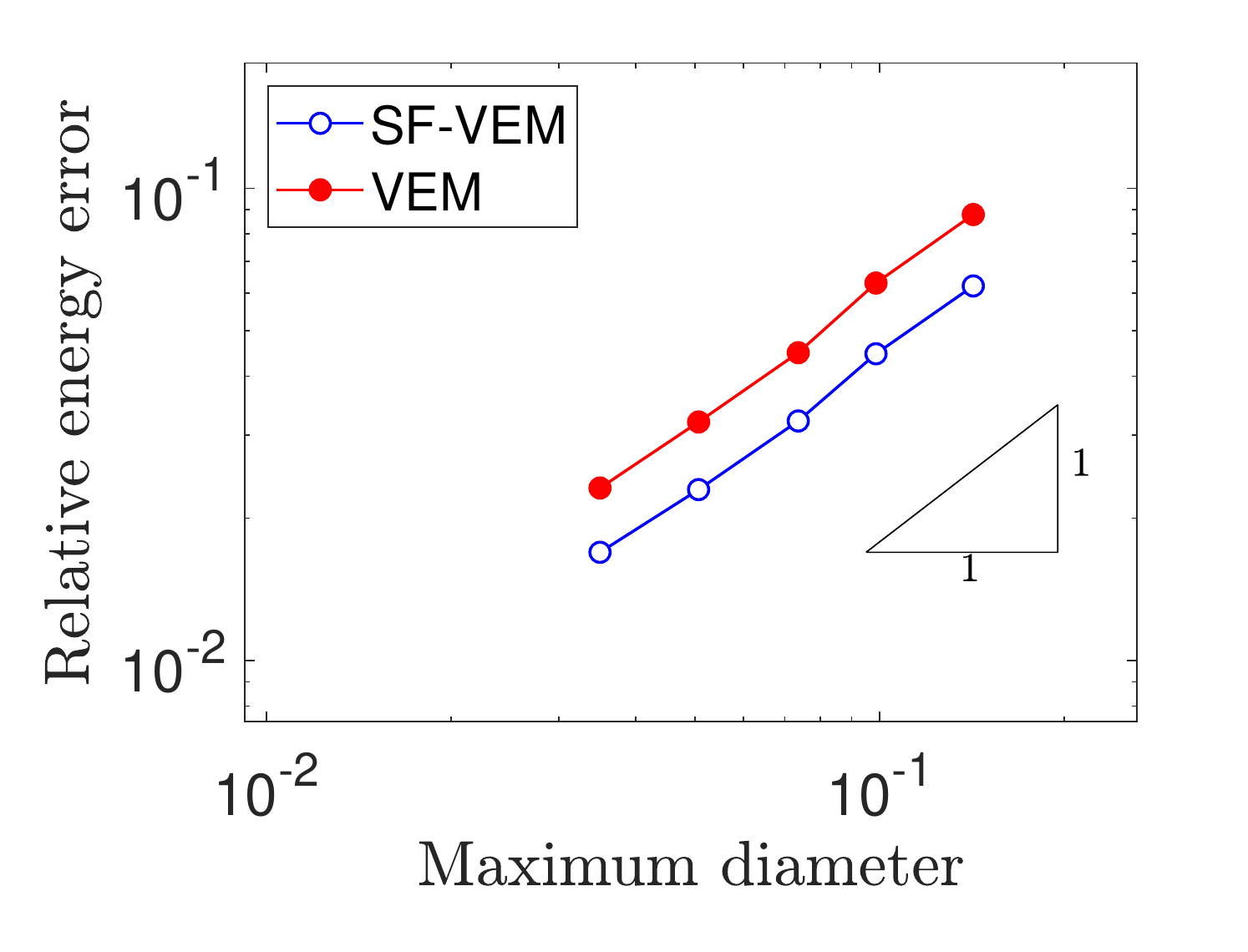}
         \caption{}
     \end{subfigure}
        \caption{Comparison of the convergence of the stabilization-free VEM (SF) and a standard VEM with a stabilization term for the cantilever beam problem. (a) $L^2$ error and (b) energy error. }
        \label{fig:cantilever}
\end{figure}

\acmajor{This problem is also tested on nonconvex meshes. We start with a uniform quadrilateral mesh and split each element into two nonconvex heptagonal elements. In the convergence study, a sequence of successively
refined meshes are used; three meshes from this sequence are presented
in Figure~\ref{fig:beammesh_non_convex}. In Figure~\ref{fig:cantilever_nonconvex}, we plot the $L^2$ and energy errors of both the stabilization-free VEM and the standard VEM. The errors are comparable to the results  
in Figure~\ref{fig:cantilever} and reveals that the stabilization-free method also performs equally well on nonconvex meshes.}
\begin{figure}[!htb]
     \centering
     \begin{subfigure}{\textwidth}
         \centering
         \includegraphics[width=\textwidth]{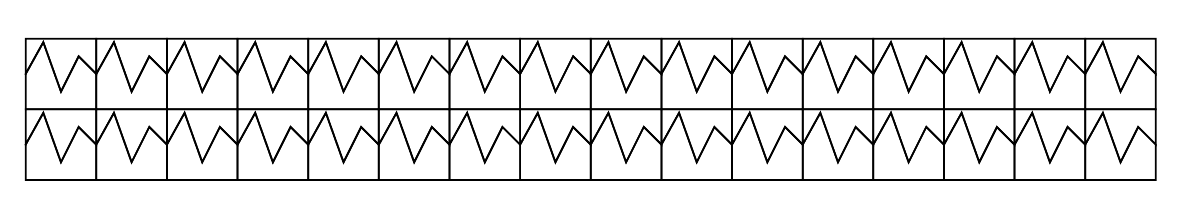}
         \caption{}
     \end{subfigure}
     \vfill
     \begin{subfigure}{\textwidth}
         \centering
         \includegraphics[width=\textwidth]{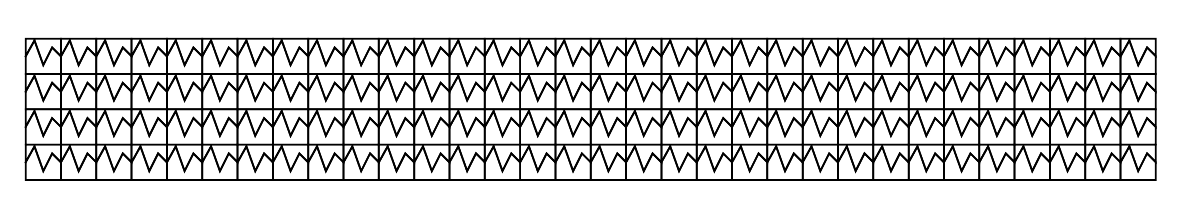}
         \caption{}
     \end{subfigure}
     \vfill
     \begin{subfigure}{\textwidth}
         \centering
         \includegraphics[width=\textwidth]{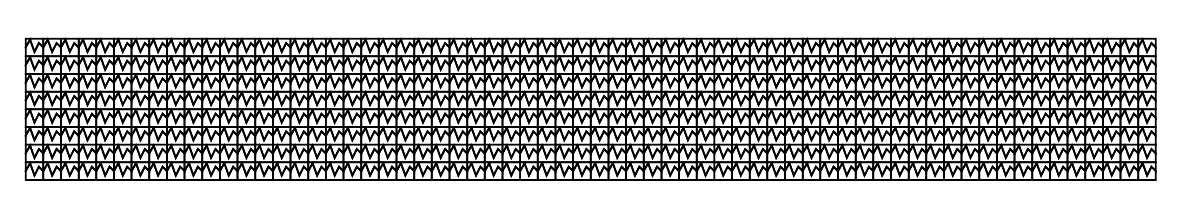}
         \caption{}
     \end{subfigure}
        \caption{\acmajor{Nonconvex polygonal meshes for the cantilever beam problem. (a) 64 elements, (b) 256 elements and (c) 1024 elements.}  }
        \label{fig:beammesh_non_convex}
\end{figure}

\begin{figure}[!h]
     \centering
     \begin{subfigure}{0.48\textwidth}
         \centering
         \includegraphics[width=\textwidth]{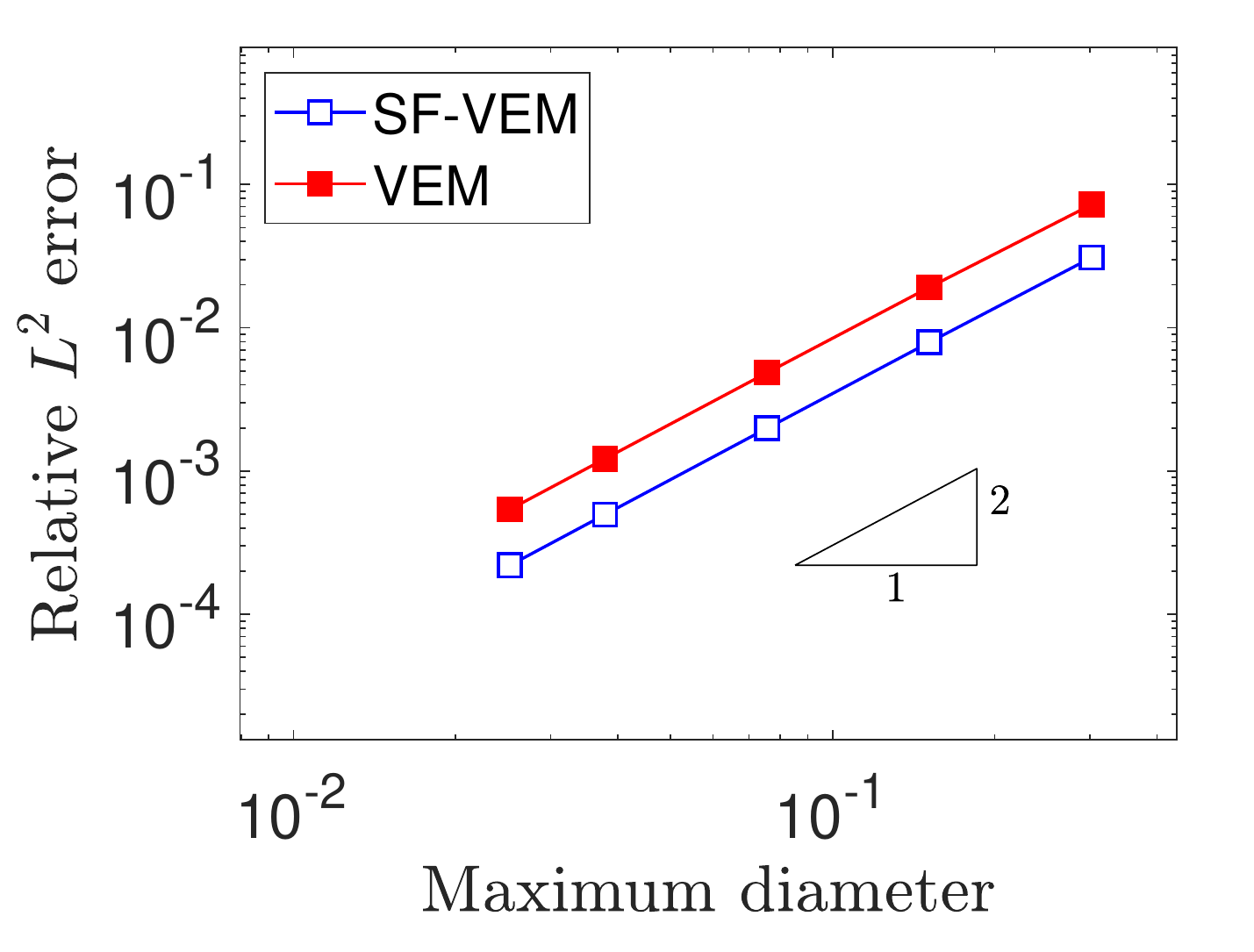}
         \caption{}
     \end{subfigure}
     \hfill
     \begin{subfigure}{0.48\textwidth}
         \centering
         \includegraphics[width=\textwidth]{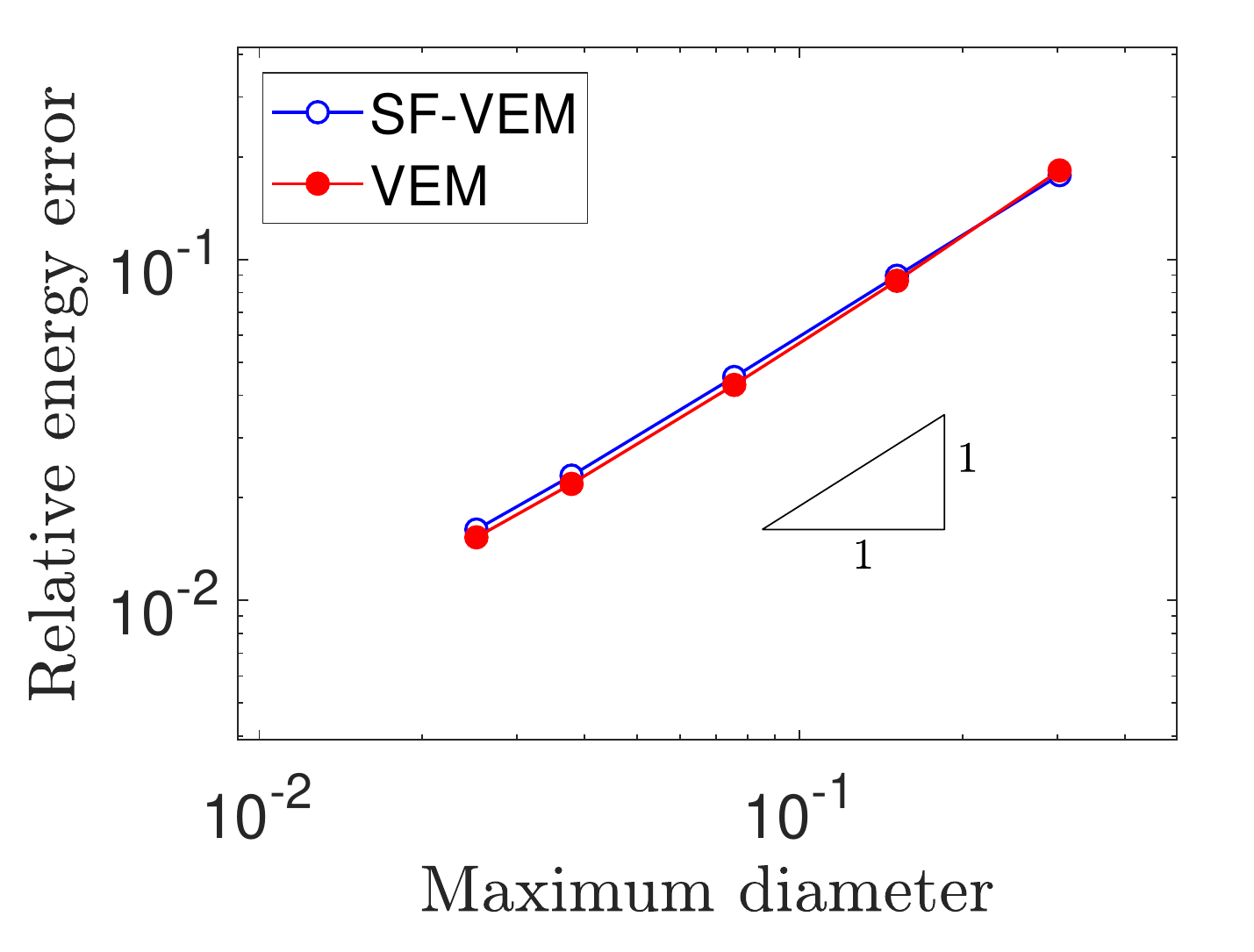}
         \caption{}
     \end{subfigure}
        \caption{\acmajor{Comparison of the convergence of the stabilization-free VEM (SF) and a standard VEM with a stabilization term for the cantilever beam problem on nonconvex meshes. (a) $L^2$ error and (b) energy error.} }
        \label{fig:cantilever_nonconvex}
\end{figure}

\subsection{Infinite plate with a circular hole}
We now consider the problem of an infinite plate with a circular hole under uniaxial tension. The hole is subject to traction-free condition, while a far field uniaxial tension $ \sigma_0=1$ psi, is applied to the plate in the $x$-direction. We use the material properties $\acmajor{E_Y}=2\times 10^7$ psi and $\nu = 0.3$, with a hole radius $a=1$ inch. Due to symmetry, we model a
quarter of the finite plate ($L = 5$ inch), with exact boundary tractions prescribed as data. Plane strain conditions are assumed. A Lloyd iterated Voronoi meshing is used~\cite{talischi2012polymesher}. In Figure~\ref{fig:holemesh}, we show a few illustrative meshes. We also plot the convergence curves for the three associated errors in Figure~\ref{fig:circularhole}. From this plot, we observe that the $L^2$ 
norm converges with order $2$, and the energy is decaying at order $1$, which agree with the theoretical predictions. 
\begin{figure}[!h]
     \centering
     \begin{subfigure}{0.32\textwidth}
         \centering
         \includegraphics[width=\textwidth]{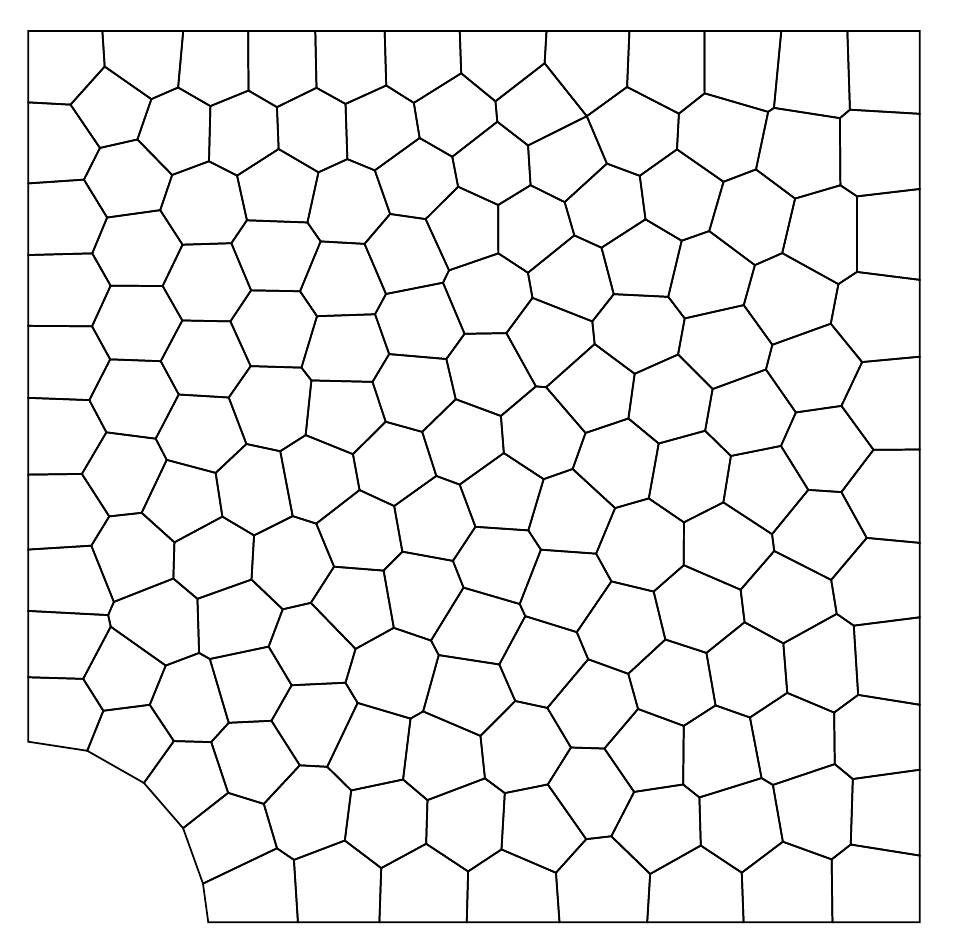}
         \caption{}
     \end{subfigure}
     \hfill
     \begin{subfigure}{0.32\textwidth}
         \centering
         \includegraphics[width=\textwidth]{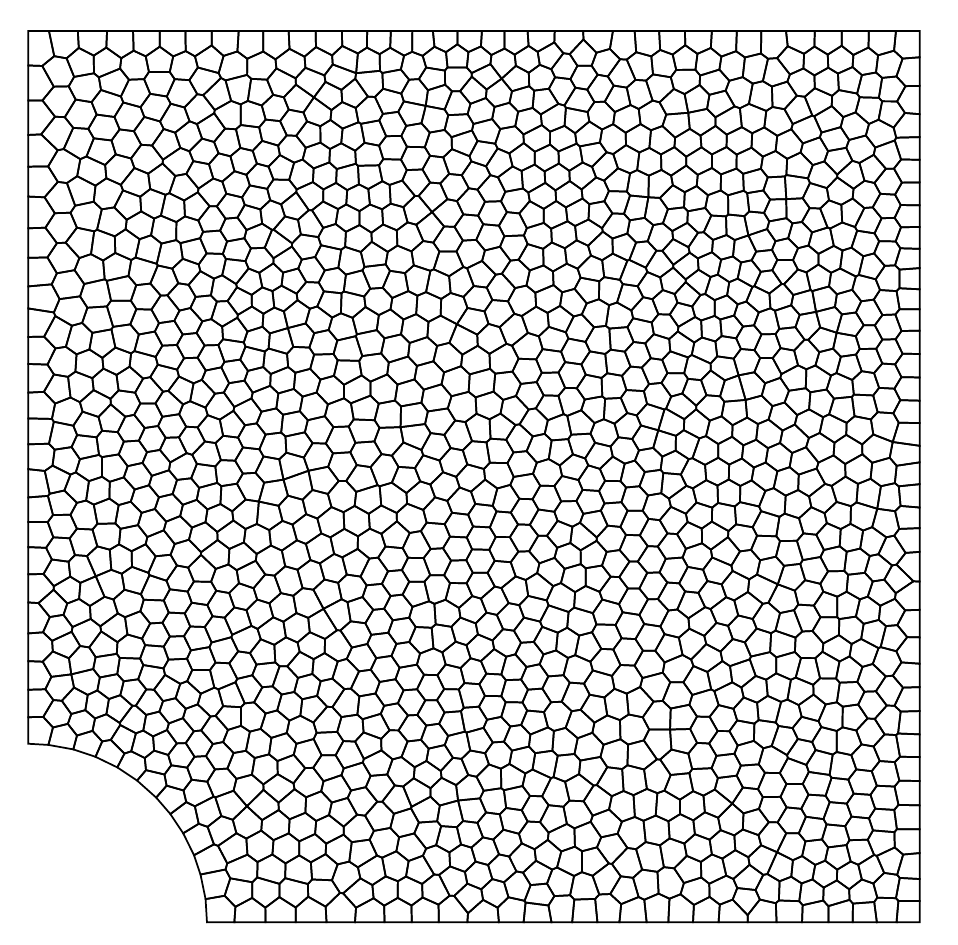}
         \caption{}
     \end{subfigure}
     \hfill
     \begin{subfigure}{0.32\textwidth}
         \centering
         \includegraphics[width=\textwidth]{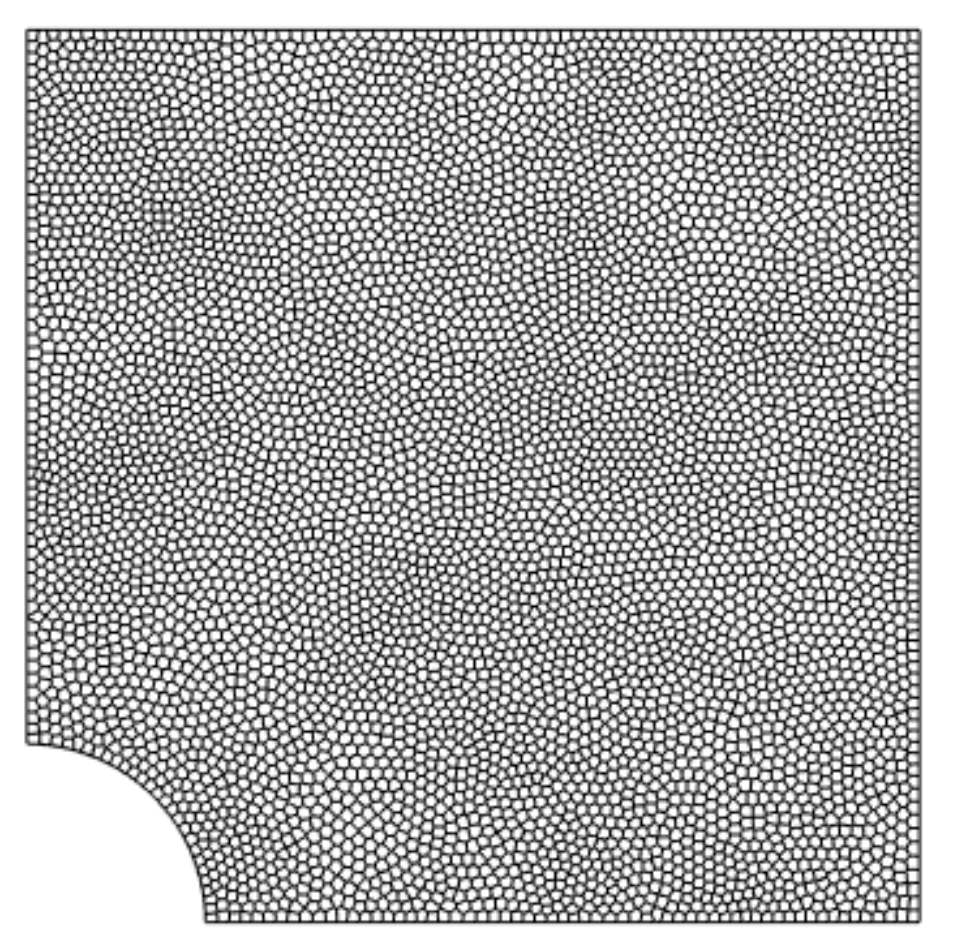}
         \caption{}
     \end{subfigure}
        \caption{Polygonal meshes used for the plate with a circular hole problem. (a) 250 elements, (b) 1500 elements, and (c) 6000 elements.}
        \label{fig:holemesh}
\end{figure}
\begin{figure}[!h]
    \centering
    \includegraphics[width=0.64\textwidth]{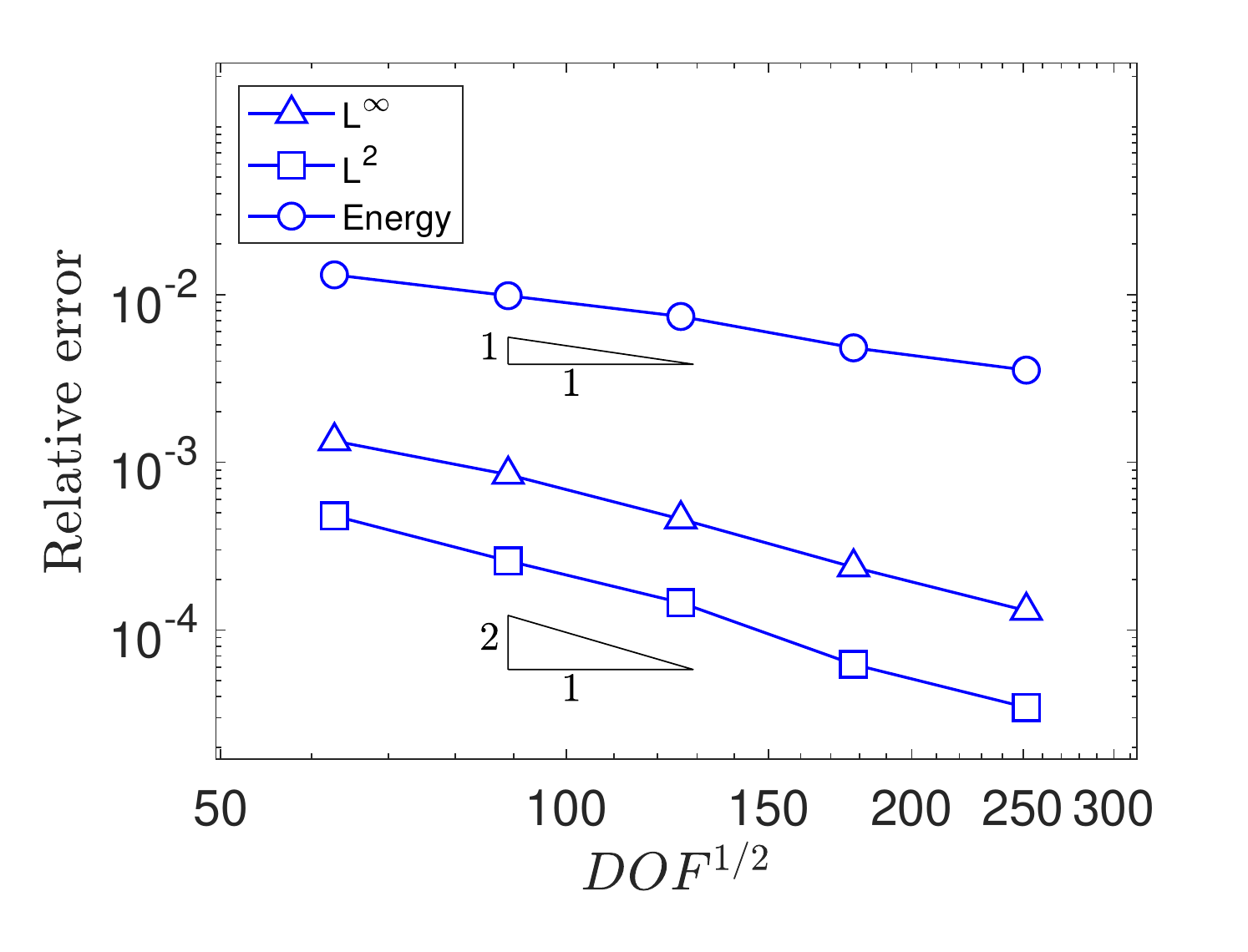}
    \caption{Convergence curves for the plate with a hole problem.}
    \label{fig:circularhole}
\end{figure}

\subsection{Hollow cylinder under internal pressure}
Finally, we consider the problem of a hollow cylinder that is subject to internal pressure~\cite{timoshenko1951theory}.
The inner and outer radii of the cylinder are chosen as $a = 1$ inch and $b = 5$ inch, respectively.
We apply a uniform constant pressure of $p = 10^5$ psi on the inner radius, while the outer radius is traction-free.
In Figure~\ref{fig:cylindermesh}, we present a few sample meshes that are generated 
using~\cite{talischi2012polymesher}.
In Figure~\ref{fig:cylinder}, we plot the errors in the three norms and compare it with the maximum diameter on the mesh. We find that the
convergence rates in both the $L^2$ norm and the energy seminorm are in
agreement with the theoretical rates. 
\begin{figure}
     \centering
     \begin{subfigure}{0.32\textwidth}
         \centering
         \includegraphics[width=\textwidth]{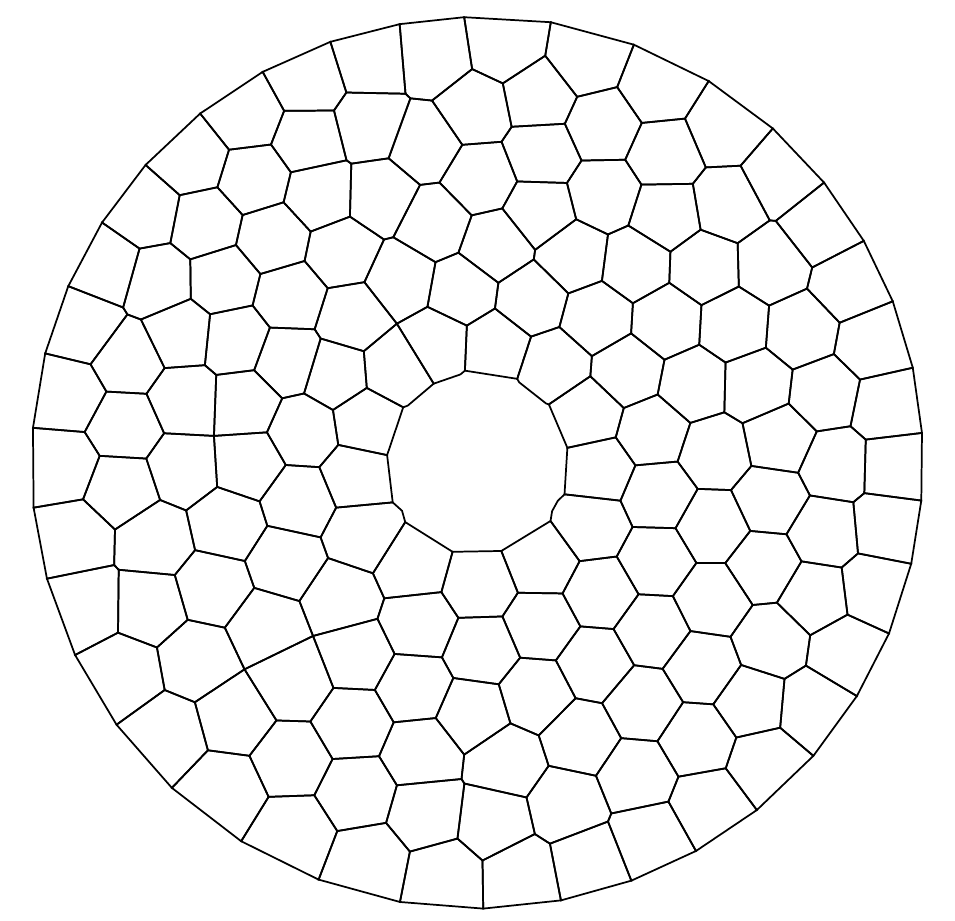}
         \caption{}
     \end{subfigure}
     \hfill
     \begin{subfigure}{0.32\textwidth}
         \centering
         \includegraphics[width=\textwidth]{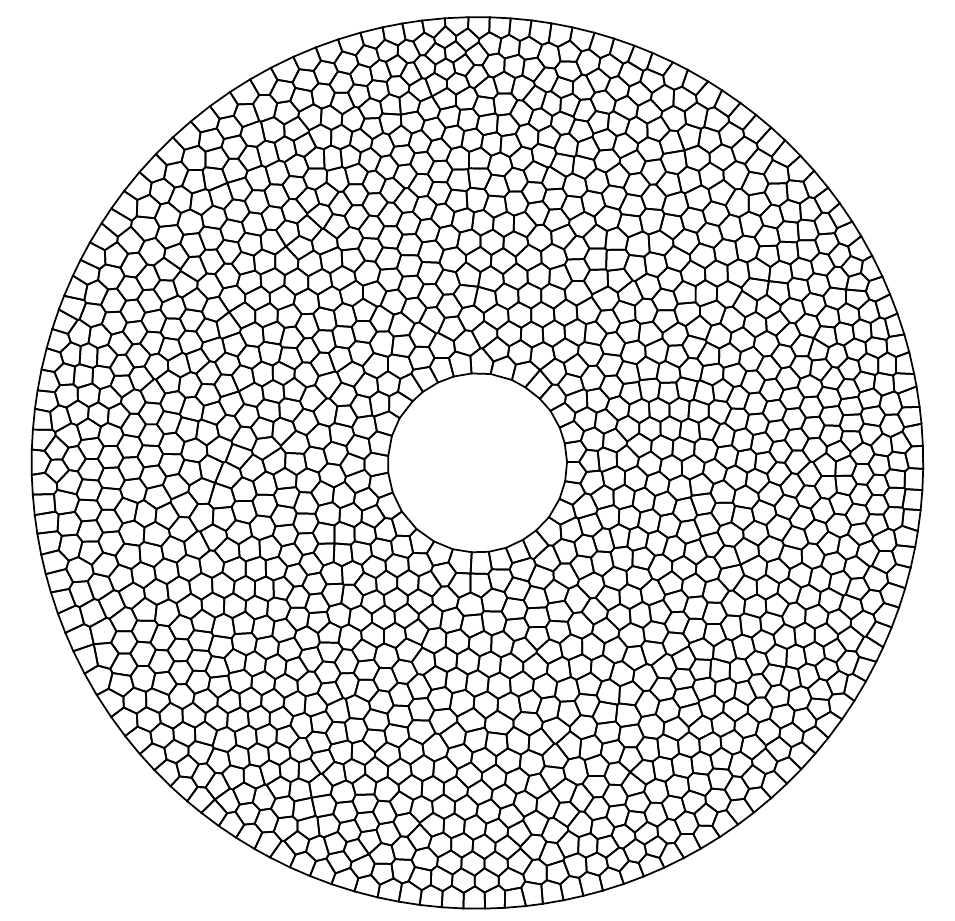}
         \caption{}
     \end{subfigure}
     \hfill
     \begin{subfigure}{0.32\textwidth}
         \centering
         \includegraphics[width=\textwidth]{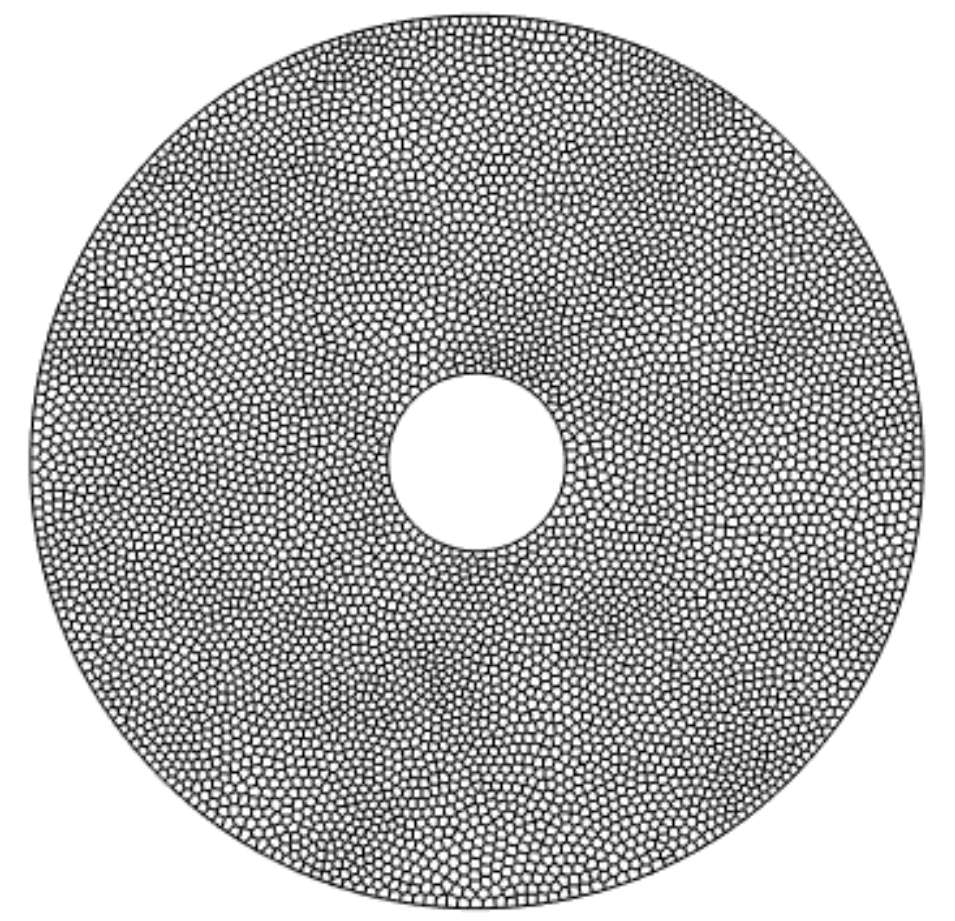}
         \caption{}
     \end{subfigure}
        \caption{Polygonal meshes used for the pressurized cylinder problem. (a) 250 elements , (b) 1500 elements, and (c) 6000 elements.}
        \label{fig:cylindermesh}
\end{figure}
\begin{figure}[!h]
    \centering
    \includegraphics[width=0.64\textwidth]{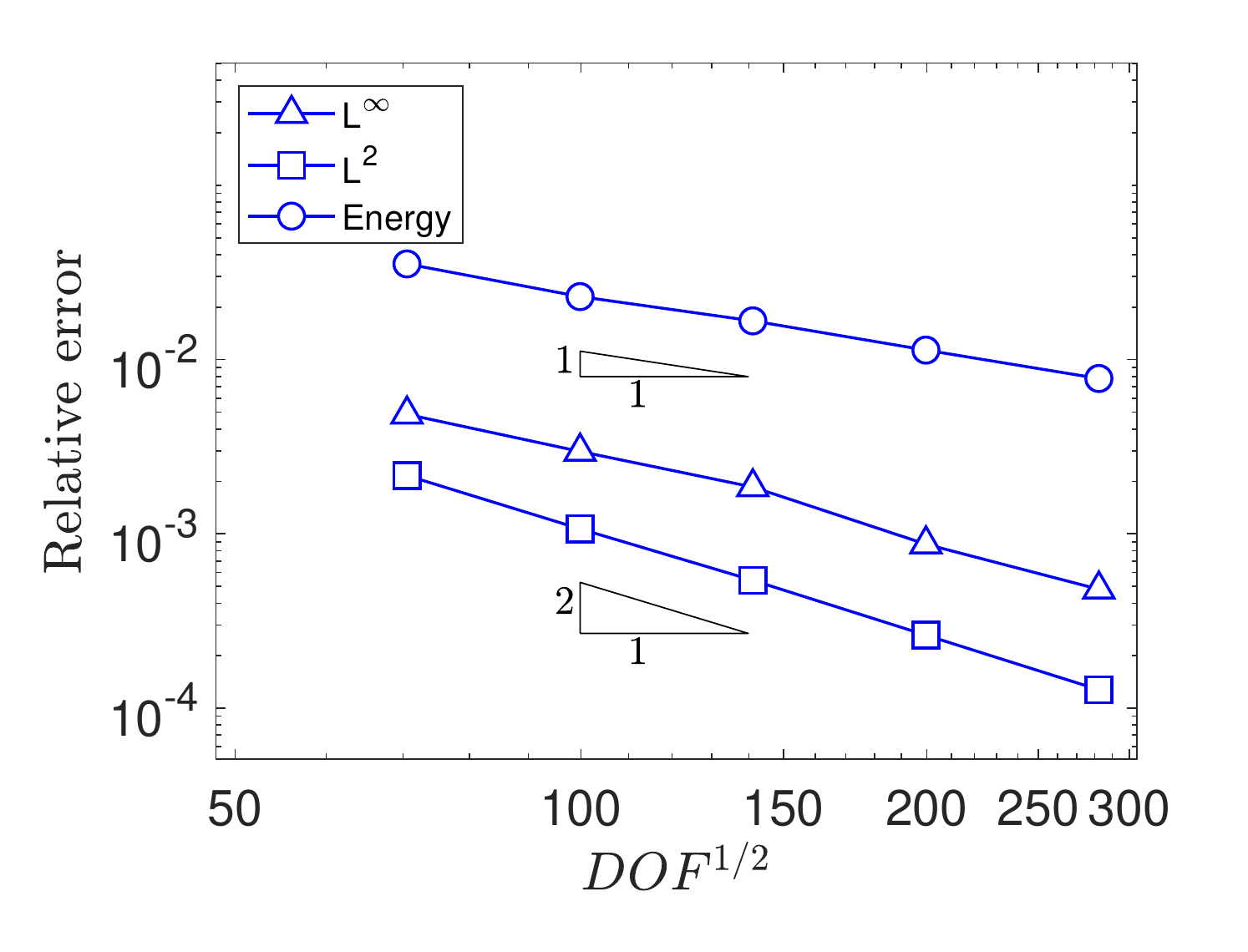}
    \caption{Convergence curves for the hollow cylinder under internal
             pressure problem.}
    \label{fig:cylinder}
\end{figure}

\section{Conclusions}\label{sec:conclusions}
In this paper, we studied an extension of the stabilization-free virtual element method~\cite{berrone2021lowest} to planar elasticity problems. To establish a stabilization-free method for solid continua, we constructed an enlarged VEM space 
that included higher order polynomial approximations of the strain field. On each polygonal element we chose the degree $\ell$ of vector polynomials, and theoretically established that the
discrete problem without a stabilization term was bounded and coercive. Error
estimates of the displacement field in the $L^2$ norm and energy seminorm 
were derived.
We set up the construction of the necessary projections and stiffness matrices, and then solved several
problems from plane elasticity. 
For the patch test, we recovered the displacement and stress fields to near machine-precision. 
From an element-eigenvalue analysis, we numerically
confirmed that the 
choice of $\ell$ was sufficient to 
ensure that the element stiffness matrix had no spurious zero-energy modes,
and hence the element was stable.
For problems such as cantilever beam under shear end load, infinite plate with a circular hole under uniaxial tension, 
and pressurized hollow cylinder under internal pressure, 
we found that the convergence rates of the 
stabilization-free VEM in the $L^2$ norm and energy seminorm were in agreement with the theoretical results.
As part of future work, several topics on stabilization-free VEM hold promise: higher order formulations, applications in three dimensions, and extensions to problems in the mechanics of compressible and incompressible nonlinear solid continua to name a few.

\section*{Acknowledgement}
\acmajor{The authors thank Alessandro Russo for informing us of a counterexample on a regular hexagon that did not satisfy a prior version of the inequality~\eqref{l_condition}, and for ensuing discussions on the extension of the counterexample to regular polygons.}

\end{document}